\documentclass{article}

\usepackage[color=orange]{todonotes}
\usepackage{amsmath}
\usepackage{amsthm}
\usepackage{amsbsy} 
\usepackage{bbm}
\usepackage{amssymb}
\usepackage{amsfonts}
\usepackage{cancel}
\usepackage{centernot}
\usepackage{cite}
\usepackage[dvips]{lscape}
\usepackage{graphicx}
\usepackage{tikz}
\usepackage{lipsum} 
\usepackage{hyperref}
\usepackage{verbatim}
\usepackage{titlesec}
\usepackage{textcase}

\usepackage[margin=3.4cm]{geometry}

\theoremstyle{plain}
\newtheorem{theorem}{Theorem}[section]
\newtheorem{proposition}[theorem]{Proposition}
\newtheorem{lemma}[theorem]{Lemma}
\newtheorem{corollary}[theorem]{Corollary}
\newtheorem{remark}[theorem]{Remark}
\newtheorem{definition}[theorem]{Definition}

\newtheorem*{theorem*}{Theorem}

\let\n\noindent

\newcommand{\cc}{  \bb{C}_{(0,\infty)}  }
\newcommand{\re}{ \textup{Re}  }

\newcommand{\ov}{   \overline{\Pi}   }

\newcommand{\pp}{\bb{P}}
\newcommand{\conv}{\raisebox{-0.6ex}{\scalebox{2.5}{$\ast$}}}

\newcommand{\lb}{\left (}
\newcommand{\rb}{\right )}
\newcommand{\lbb}{\left [}
\newcommand{\rbb}{\right ]}
\newcommand{\labs}{\left |}
\newcommand{\rabs}{\right |}
\newcommand{\lbrb}[1]{\lb #1 \rb}
\newcommand{\lbbrbb}[1]{\lbb#1\rbb}

\newcommand{\lbbrb}[1]{\lbb#1\rb}
\newcommand{\lbrbb}[1]{\lb#1\rbb}

\newcommand{\lbcurly}{\left\{}
\newcommand{\rbcurly}{\right\}}
\newcommand{\lbcurlyrbcurly}[1]{\lbcurly#1\rbcurly}

\newcommand{\intervalII}{\lbrb{-\infty,\infty}}
\newcommand{\intervalOI}{\lbrb{0,\infty}}
\newcommand{\abs}[1]{\labs#1\rabs}

\newcommand{\curly}[1]{\lbcurly#1\rbcurly}

\newcommand{\bo}[1]{\mathrm{O}\lbrb{#1}}
\newcommand{\so}[1]{\mathrm{o}\lbrb{#1}}

\newcommand{\tightoverset}[2]{%
	\mathop{#2}\limits^{\vbox to 0.28ex{\kern -0.2ex\hbox{$#1$}\vss}}}

\newcommand{\Pbb}[1]{\Pb\lb #1\rb}
\newcommand{\Ebb}[1]{\Eb\lbb #1\rbb}

\newcommand{\LL}{L\'{e}vy }
\newcommand{\LLP}{L\'{e}vy process }
\newcommand{\LLPs}{L\'{e}vy processes }
\newcommand{\LLK}{\LL\!\!-Khintchine }

\newcommand{\mubar}[1]{\bar{\mu}\lbrb{#1}}

\newcommand{\muzeta}{\mu_{\zeta}}
\newcommand{\muReta}{\mu_{\Re\lbrb{\zeta}}}
\newcommand{\bmuzeta}{\bar{\mu}_{\zeta}}
\newcommand{\bmuReta}{\bar{\mu}_{\Re(\zeta)}}
\newcommand{\Reta}{\Re\lbrb{\zeta}}
\newcommand{\Rez}{\Re\lbrb{z}}
\newcommand{\limi}[1]{\lim\limits_{#1\to \infty}}

\newcommand{\Cb}{\mathbb{C}}

\newcommand{\Eb}{\mathbb{E}}

\newcommand{\N}{\mathbb{N}}
\newcommand{\Rb}{\mathbb{R}}
\newcommand{\R}{\mathbb{R}}

\newcommand{\Pb}{\mathbb{P}}

\newcommand{\Nbp}{\mathbb{N}^{+}}

\newcommand{\Ac}{\mathcal{A}}
\newcommand{\Bc}{\mathcal{B}}

\newcommand{\Mcc}{\mathcal{M}}

\newcommand{\Prm}{\mathrm{P}}
\newcommand{\ind}[1]{\mathbb{I}_{\{#1\}}}

\newcommand{\IntOI}{\int_{0}^{\infty}}

\newcommand{\phq}{\phi_{q}}
\newcommand{\phzeta}{\phi_{\zeta}}
\newcommand{\phReta}{\phi_{\Re(\zeta)}}

\newcommand{\dr}{{\mathtt{d}}}

\newcommand{\Ekap}{E_\kappa}

\newcommand{\gampkap}{\gamma_\kappa}
\newcommand{\Lkap}{L_\kappa}

\newcommand{\Tkap}{T_\kappa}

\newcommand{\Wp}{W_\phi}
\newcommand{\Wkap}{W_\kappa}

\newcommand{\Wpq}{W_{\phi_q}}

\newcommand{\CbOI}{\Cb_{\intervalOI}}

\renewcommand{\Ac}{\mathtt{A}}
\renewcommand{\Im}{\mathtt{Im}}
\renewcommand{\Re}{\mathtt{Re}}

\let\b\begin
\let\e\end
\let\f\frac
\let\bb\mathbb
\let\cal\mathcal
\let\bbm\mathbbm
\let\l\left
\let\r\right
\let\p\paragraph

 \numberwithin{equation}{section}

\newpagestyle{main}{%
  \sethead[\thepage][][\thesection \quad \sectiontitle] % even
  {\hdrtitle}{}{\thepage}} % odd
\makeatother
\sethead[\normalsize{\thepage}][\small{ADAM BARKER, MLADEN SAVOV}][] % even header
  {}{\small{EXPONENTIAL~FUNCTIONALS~OF~SUBORDINATORS}}{\normalsize{\thepage}} % odd header
\setfoot[][][] % even footer
  {}{}{} % odd footer

\date{}
\title{{\normalsize\textbf{\MakeUppercase{Bivariate Bernstein-gamma functions and moments of exponential functionals of subordinators}}}}

\author{{\sc A. Barker}\thanks{Department of Mathematics and Statistics, University of Reading, Reading, United Kingdom \  E-mail: Adam.Barker@pgr.reading.ac.uk} \: {\sc and}  {\sc M.~Savov}\thanks{Institute of
		Mathematics and Informatics, Bulgarian Academy of Sciences,  "Akad. Georgi Bonchev" bl. 8, Sofia 1113, Bulgaria \  E-mail: mladensavov@math.bas.bg}}   
\begin{document}
\renewcommand{\baselinestretch}{1.5}

\maketitle

\paragraph{Abstract}   In this paper, we extend recent work  on the functions that we call Bernstein-gamma to the class of  bivariate Bernstein-gamma functions. 
In the more general bivariate setting, we determine Stirling-type asymptotic bounds which generalise, improve upon and streamline those found for the univariate Bernstein-gamma functions. 
 Then, we demonstrate the importance and power of these results through an application to exponential functionals of L\'evy processes.
 In more detail, for a subordinator (a non-decreasing L\'evy process) $(X_s)_{s\geq0}$, we study its \textit{exponential functional},   $\int_0^t e^{-X_s}ds $, evaluated at a finite, deterministic time $t>0$.  Our main result here is an explicit infinite convolution formula for the Mellin transform (complex moments) of the exponential functional up to time $t$ which under very minor restrictions is shown to be equivalent to an infinite series. We believe this work can be regarded as a stepping stone towards a more in-depth study of general exponential functionals of \LLPs on a finite time horizon.

\p{Keywords:} L\'evy processes; Complex analysis; Special functions; Financial mathematics
  %\newpage 
 % \vspace{-1cm}
% {\small \tableofcontents}

 \section{Introduction, Background and Motivation}  

Each Bernstein function, $\phi$, %which is the Laplace exponent of a subordinator,
  has been shown to have a unique associated Bernstein-gamma function, $\Wp$, defined  by  the recurrent equation
 \[ 
\Wp(z+1)=\phi(z)\Wp(z), \   \Re(z)>0; \qquad \Wp(1)=1,
 \] 
 see \cite[Section 6]{ps19} or \cite{ps16}.
In this work we    study a suitable generalisation of the class of Bernstein-gamma functions.  
Each bivariate Bernstein function, $\kappa$ (the \LLK exponent of a possibly-killed bivariate subordinator), is shown to have a unique  bivariate Bernstein-gamma function, $\Wkap$, defined by the recurrent equation% by the property that for each  $\zeta,z$ with $\Re(\zeta),\Re(z)>0,$  %$\zeta\in \cc$,
\[
		\Wkap\lbrb{\zeta,z+1}=\kappa\lbrb{\zeta,z}\Wkap\lbrb{\zeta,z}, \ \Re(\zeta)\geq 0,\Re(z)>0; 
		\qquad    \Wkap\lbrb{\zeta,1}=1.
\]
% A Bernstein function $\phi$ has a unique associated subordinator, and similarly a bivariate Bernstein function has a unique associated bivariate subordinator. 
Our main analytical result on bivariate Bernstein-gamma functions, Theorem \mbox{\ref{thm:Stirling}}, provides a general Stirling asymptotic representation for $\Wkap$. It is an improvement upon \cite[Theorem 3.3]{ps16}, which only gives a Stirling representation for the absolute value %$|W_\phi|$
 in the univariate case. Also, building upon \cite[Theorem 6.1]{ps19}, Theorem \mbox{\ref{thm:BernGammaExt}}   gives a Weierstrass product representation for $\Wkap$. The univariate Bernstein-gamma functions are intimately linked to Markovian self-similarity and other important quantities in probability and spectral theory, see \cite{abcpmv19,lps19,ps16,ps19,psy19} and the discussion below. In this work we will demonstrate that bivariate Bernstein-gamma functions also play a role in probability theory via the study of exponential functionals of \LLPs up to a deterministic horizon, although we expect further applications to appear.
 In more detail, we  apply our results on bivariate Bernstein-gamma functions to  exponential  functionals of L\'evy processes, as follows. For a subordinator $(X_s)_{s\geq0}$, we study its exponential functional, defined as $I_\phi(t):= \int_0^t e^{-X_s}ds $, $t\in[0,\infty)$.  Our main results concern information on the Mellin tranform of $I_\phi(t)$, that is $\cal{M}_{I_\phi(t)}(z+1):=\bb{E}[I_\phi(t)^z]$, for $\Re(z)>0$.    % For a broad class of processes, %excluding only highly pathological cases,  
 Here we only highlight the following representation: %  of the exponential functional up to a finite time:   
\b{equation}
\label{introformula}
\cal{M}_{I_\phi(t)}(z+1)  =    \f{ \Gamma(z+1)   }{  W_{\phi}(z+1)    }  -  
 \sum_{k=1}^\infty      \f{       \prod_{i=1}^{k}     [ \phi(z+i) - \phi(k)]    }{     \prod_{j=1}^{k-1}       [ \phi(j) - \phi(k)]  } 
 \f{     e^{-\phi(k)t}     }{   \phi(k)}   \f{ \Gamma(z+1)     }{   W_{\phi_{(k)}}(z+1)  } ,  
\e{equation}
which holds under a minor regularity condition and where    $ \phi_{(k)}(w):=  \phi(w+k)- \phi(k)$ is a Bernstein function and $W_{ \phi_{(k)}}$ is its corresponding univariate Bernstein-gamma function, see Definitions \mbox{\ref{bdefn}}, \mbox{\ref{bgdefn}}. We emphasize that  $\f{ \Gamma(z+1)     }{   W_{\phi_{(k)}}(z+1)}, k\geq  1,$ are the Mellin transforms of exponential functionals of subordinators on infinite horizon.
  
  The   question of finding precise information on the distribution of the exponential functional of a L\'evy process up to a finite, deterministic time was posed in  the 1990's, see e.g.\ \cite[Remark 3.2]{cpy97}, yet very few  works have since been able to cover this case. For the Brownian motion case an extensive study of the law of the exponential functional has been carried out in \cite{hy06}. Recently, interesting results concerning moments of exponential functionals of processes with independent increments have been discussed in \mbox{\cite{sv18}} where in particular $\Ebb{I^n_\phi(t)}, n\geq 1,$ have been computed when $\phi(0)=0$, see Theorem \mbox{\ref{thm:SV}} below. In this work we provide an expression for  any complex moments with positive real part. This work can be considered a stepping stone for a more in-depth study of exponential functionals of more general \LLPs up to a deterministic horizon, and the reason why the Mellin transform is a suitable starting point in such an endeavour can be explained as follows. Consider a \LLP $\xi$ which is killed at independent exponentially distrubuted random time $e_q, q>0$. Then, in a sequence of papers \cite{mz06,ps12,ps13,ps16}, it has been shown gradually that in general 
  \begin{equation*}
  \begin{split}
  &		\Ebb{\lbrb{\int_{0}^{e_q}e^{-\xi_s}ds}^z}=\kappa_{q,-}(0)\frac{\Gamma(z+1)}{W_{\kappa_{q,+}}(z+1)}W_{\kappa_{q,-}}(-z),\,\,\Rez\in\lbrbb{-1,0},    
  \end{split}
  \end{equation*}
   where $\kappa_{q,\pm}$ are bivariate Bernstein functions corresponding to the Wiener-Hopf  factors of $\xi$, see for example \cite[Chapter VI]{B96} for an introduction to Wiener-Hopf factorization.  On the other hand 
   \begin{equation*}
   \begin{split}
   &	\frac{1}{q}\Ebb{\lbrb{\int_{0}^{e_q}e^{-\xi_s}ds}^z}=\int_{  0}^\infty	e^{-qt} \Ebb{\lbrb{\int_{0}^{t}e^{-\xi_s}ds}^z} dt  
   \end{split}
   \end{equation*}
   and one can try to understand $\Ebb{\lbrb{\int_{0}^{t}e^{-\xi_s}ds}^z}$ through standard Laplace inversion. The point where subordinators always appear is the expression $\frac{\Gamma(z+1)}{W_{\kappa_{q,+}}(z+1)}$, which corresponds to an exponential functional of a killed subordinator.
    Therefore, it seems likely  that  our results, such as $\l(\ref{introformula}\r)$, may have implications well beyond subordinators.

The study of exponential functionals of L\'evy processes has received much attention in recent years. Advancements in the general theory can be found in \cite{ajr13,blm16,bpy04,by02,mz06,pps12,prs13,ps12,ps13,ps16,sv18,u95}. These quantities have been used in various areas of probability theory, such as branching processes  and processes in random environments, see \cite{pps16,lx18,p09}, spectral theory of non-self-adjoint semigroups, see \cite{ps19,psy19}, positive self-similar Markov processes, see \cite{by02b,kp13}, financial mathematics, see \cite{hy13} and \cite[Section 6.3]{by05}. Exponential functionals up to random exponential horizon have also appeared in the study of Asian options, see \cite{hk14,jv18,p13}. For Asian options, which are valued according to an integral of the form $\int_0^t e^{-X_u}du$, where $X_u$ denotes an asset price at time $u$, one needs to consider exponential functionals up to deterministic horizon, but the latter have proved to be extremely hard to deal with, and that is why researchers have focused on their Laplace transform. From this perspective
our studies, which deal with obtaining knowledge of the exponential functional up to a finite, deterministic time, can be relevant to pricing of Asian options. %Further applications include studying a one-dimensional diffusion in a random L\'evy environment,    random dynamical systems. 
  We refer to \cite{by05,cpy01,y12} for further details on applications of exponential functionals.

 The remainder of the paper is structured as follows: Section \mbox{\ref{subsec:not}} introduces notation and key quantities; Section \hbox{\ref{bgresults}} provides the statements of  the main results  on Bernstein-gamma functions and bivariate Bernstein-gamma functions; Section \hbox{\ref{expresults}} contains the  statements of  the main results  on  exponential functionals, including the formula $\l(\ref{introformula}\r)$ for the Mellin transform; Sections  \mbox{\ref{sec:ProofsBernGammaBiv}} and \mbox{\ref{expproofs}} contain the proofs of the main results; Section  \mbox{\ref{sec:B2}}  collects functional properties and results on bivariate Bernstein functions, which can be of independent interest; Section  \hbox{\ref{lemmasproofs}} contains proofs of the remaining lemmas.

\section{Main Results} \label{mainresults}

\subsection{Preliminary Definitions and Notation}\label{subsec:not}

We start by defining some complex-analytical quantities. We use $\Cb$ to denote the complex plane. For any $z\in\Cb$, we write $z=\Re(z)+i\Im(z)$ and we set $z = |z|e^{i arg z}$
with the branch of the argument function defined via the convention $arg : \Cb\mapsto \lbrbb{-\pi,\pi}$. For any $z\in\Cb$, set $z=|z|e^{i \arg z}$ with the branch of the argument function defined via the convention $\arg:\Cb\mapsto\lbrbb{-\pi,\pi}$. We put $\log_0:\Cb\setminus\lbrbb{-\infty,0}\mapsto \Cb$ for the main branch of the complex logarithm whereby $\log_0(z)=\ln|z|+i\arg z$. For any $-\infty\leq a<b\leq \infty$, we denote by $\Cb_{\lbrb{a,b}}=\lbcurlyrbcurly{z\in\Cb:\,a<\Re(z)<b}$ and for any $a\in\intervalII$ we set $\Cb_a=\lbcurlyrbcurly{z\in\Cb:\,\Re(z)=a}$\label{Ca}. The notation $\Cb_{\lbbrb{a,b}}=\lbcurlyrbcurly{z\in\Cb:\,a\leq \Re(z)<b}$ and all possible variations thereof denote strips whose boundary lines are included or not in the respective subset of the complex plane.  We use $\Ac_{\lbrb{a,b}}$ for the set of holomorphic functions on $\Cb_{\lbrb{a,b}}$, whereas if $-\infty<a$ then $\Ac_{\lbbrb{a,b}}$ stands for the holomorphic functions on $\Cb_{\lbrb{a,b}}$ that can  be extended continuously to $\Cb_a$. Similarly, we have the spaces $\Ac_{\lbbrbb{a,b}}$ and $\Ac_{\lbrbb{a,b}}$. We employ $\Cb^2$ for the two dimensional complex numbers with $\Cb^2_{\lbrb{a,b}}$ standing for $\Cb_{\lbrb{a,b}}\times \Cb_{\lbrb{a,b}}$ and $\Ac^2_{\lbrb{a,b}}$ standing for the class of bivariate holomorphic functions on $\Cb^2_{\lbrb{a,b}}$.

Now,    let us state key definitions for  Bernstein-gamma functions, L\'evy processes, and exponential functionals.

\begin{definition}
	\label{bdefn}  
	A function $\phi$ is a Bernstein function, that is $\phi\in\Bc$, if for all $z\in\bb{C}_{[0,\infty)}$,% $\Re(z)>0$, %and only if 
	\begin{align}\label{eq:Bern}
	\phi(z)&=\phi(0)+\textup{d} z+\IntOI \lbrb{1-e^{-zy}}\Pi(dy)\\
	\label{eq:Bern'}
	&=\phi(0)+\textup{d} z+z\IntOI e^{-zy}\ov(y)dy, 
	\end{align}
	where $\phi(0),\textup{d}\in[0,\infty)$, $\Pi$ denotes a measure on $[0,\infty)$ satisfying $\IntOI \min\curly{y,1}\Pi(dy)<\infty$, and $\ov(x):=\int_{x}^{\infty}\Pi(dy),$ for $x\geq0$.   Note that in any case $\phi\in \Ac_{\lbbrb{0,\infty}}$.
	For further  background on Bernstein functions, we refer to the book \cite{ssv12} or to the paper \cite[Section 3]{ps16}. 
\end{definition}

\n  With each $\phi\in\Bc$ there is an associated, possibly-killed subordinator (non-decreasing \LL process) $X=(X_t)_{t\geq0}$, whose \LLK exponent is defined by the relation
\begin{equation}\label{eq:LK}
\begin{split}
&\Ebb{e^{- z X_t}}= e^{-\phi(z)t}\text{ for $z\in\Cb_{\lbbrb{0,\infty}}$ and $ t\geq0$.} 
\end{split}
\end{equation}
For a subordinator with \LLK exponent $\phi$ as in $\l(\ref{eq:Bern}\r)$,    $ \textup{d} \geq0$ is the linear drift, and $\Pi$ is the L\'evy measure, %a sigma-finite measure on $(0,\infty$), 
which determines the size and intensity of its jumps. % of $X$, and satisfies $\int_0^\infty   \min\{1, x\}    \Pi(dx)<\infty$.  
If $\phi(0)>0$, then we say that the subordinator $X$ is killed at rate $\phi(0)$, and it follows that  for an independent exponential random variable $e_{\phi(0)}$ with rate parameter $\phi(0)$,  
\[
X_t =   \begin{cases}    X_t    , \qquad    t<e_{\phi(0)}  \\    \hspace{1pt}    \infty    , \qquad     \hspace{0.5pt}   t\geq e_{\phi(0)}. 
\end{cases}
\]
If our original, unkilled subordinator $X$ has Laplace exponent $\phi$, then the process $X$ killed at rate $q>0$ has Laplace exponent $\phi_q(\lambda)= q + \phi(\lambda)$.

We recall that if $Y$ is an almost surely positive random variable then $\Mcc_{Y}(z):=\Ebb{Y^{z-1}}$ is by definition its Mellin transform which is always well-defined at least for $z\in\Cb_{1}=1+i\Rb$. Now we define Bernstein-gamma functions.
%Moreover, it is well-known, see e.g. \cite[Chapter III]{b98}, that the potential measure of $\xi$, say $\Up$, is linked to $\phi$ in the following manner
%\begin{equation}\label{eq:potential}
%	\frac{1}{\phi(z)}=\IntOI e^{-zy}U(dy),\,z\in\CbOI.
%\end{equation}
%The identity extends to $i\R$ if and only if $\phi(0)>0$.
\begin{definition}%[Bernstein-Gamma Function]
	\label{bgdefn} 
	For each $\phi\in\Bc$, its associated Bernstein-gamma function   $\Wp$ %see \cite[Section 6]{ps19} 
	is defined, for $z\in\CbOI$, as the solution in the space of Mellin transforms of positive random variables of the recurrent equation
	\begin{equation}\label{eq:Bern-Gamma}
	\Wp(z+1)=\phi(z)\Wp(z),\text{ for $z\in\CbOI$ with }\Wp(1)=1.
	\end{equation}
	%and $\Wp(z+1)=\bb{E}[\Yp^{z}],\,z\in\Cb_{\lbbrb{0,\infty}}$. %, with $\Yp$ a positive random variable. 
\end{definition}
The existence of $\Wp$ for any $\phi\in\Bc$ is proven in \cite[Section 4]{ps16}, where an extensive study of its complex-analytical properties has been carried out. Finally, we define the exponential functional of a subordinator:

\begin{definition}
	\label{exponentialfunctional}   
	For a subordinator with Laplace exponent $\phi\in\cal{B}$, its exponential functional is
	\b{equation*}
	I_\phi(t) :=   \int_0^t    e^{-X_s} ds=\int_{0}^{\min\curly{t,e_{\phi(0)}}} e^{-X_s} ds, \qquad t\in[0,\infty],
	\e{equation*}
	where  the terminal value, for $t=\infty$, can also be  denoted by $I_\phi:=I_\phi(\infty)= \int_0^\infty  e^{-X_s} ds$. Note that if $\phi(0)=0$ then $e_{\phi(0)}=\infty$ almost surely.
\end{definition}

\subsection{Bivariate \  Bernstein-Gamma \  Functions \  and \  their \  Stirling \  Type \ Approximation} \label{bgresults}

	 %With each $\phi\in\Bc$ we have an associated Bernstein-gamma function, $\Wp$, see \cite[Section 6]{ps16} such that
%\begin{equation}\label{eq:Bern-Gamma1}
%\Wp(z+1)=\phi(z)\Wp(z),\,z\in\CbOI; \Wp(1)=1
%\end{equation}
%and $\Wp(z+1)=\bb{E}[ \Yp^{z} ],\,z\in\Cb_{\lbbrb{0,\infty}}$, with $\Yp$ a positive random variable.% It is known from \cite[Section 6]{ps16} that $\Wp$ admits the following absolutely convergent Weierstrass product representation
%\begin{equation}\label{eq:Bern-GammaProd1}
%\Wp(z)=\frac{e^{-\gamph z}}{\phi(z)}\prod_{k=1}^{\infty}\frac{\phi(k)}{\phi(k+z)}e^{\frac{\phi'(k)}{\phi(k)}z},\,z\in\CbOI,
%\end{equation}
%where 
%\begin{equation}\label{eq:gamph1}
%\gamph=\limi{n}\lbrb{\sum_{k=1}^{n}\frac{\phi'(k)}{\phi(k)}-\ln\phi(n)}\in\lbbrbb{-\ln\phi(1),\frac{\phi'(1)}{\phi(1)}-\ln\phi(1)}
%\end{equation}
%can be thought of as a generalization of the celebrated Euler-Mascheroni constant $\gamma$ and corresponds to the case when $\phi(z)=z$. 

\n To extend the theory of Bernstein-gamma functions to the bivariate setting, % we regard at the quantities $\Wp$ in a more general manner. To do so
 we first define $\Bc^2$, the class  of bivariate Bernstein functions,  which generalises the class $\cal{B}$ to the bivariate case.  
\begin{definition}\label{def:Bc2} We say that a function $\kappa$ is a bivariate Bernstein function %, that is $\kappa\in\cal{B}^2$, 
if for all $\zeta, z \in \bb{C}_{[0,\infty)}$,
	%We say that $\kappa\in\Bc^2$ if and only if $\kappa\neq 0$ and $\kappa$ has the form
	\begin{equation}\label{eq:kappa}
		\kappa\lbrb{\zeta,z}=\kappa(0,0)+\dr_1\zeta+\dr_2z+\IntOI\IntOI \lbrb{1-e^{-\zeta x_1-zx_2}}\mu(dx_1,dx_2), %\, \lbrb{\zeta,z}\in\Cb^2_{\lbbrb{0,\infty}},
	\end{equation}
	where $\kappa(0,0),\dr_1,\dr_2\in [0,\infty)$ and $\mu$ is a measure on $(0,\infty)\times(0,\infty)$ 
	such that \[\IntOI\IntOI \min\curly{x_1,1}\min\curly{x_2,1}\mu(dx_1,dx_2)<\infty.\]
	Note that according to Lemma \mbox{\ref{lem:repKappa1}} we have that $\kappa\in\Ac^2_{\lbbrb{0,\infty}}$.
\end{definition}
Observe that $\kappa\in\Bc^2$ if and only if $\kappa$ is the bivariate Laplace exponent of a possibly killed bivariate subordinator, see \cite[p.27]{d07} for further details.  For some important properties of the class $\cal{B}^2$, see Proposition \mbox{\ref{prop:kappa}}, which collects key results on the class $\cal{B}^2$. These are natural extensions of known properties of the univariate class $\cal{B}$ but seem not to have appeared in the literature.  Now, let us define the class of  bivariate Bernstein-gamma functions.
\begin{definition}\label{def:BernGammaBiv}
	We say that $\Wkap$ is a bivariate Bernstein-gamma function if 
	\begin{equation}\label{eq:BivBernGam}
		\Wkap\lbrb{\zeta,z+1}=\kappa\lbrb{\zeta,z}\Wkap\lbrb{\zeta,z}, \quad z\in\CbOI, \zeta\in \bb{C}_{[0,\infty)};
	\end{equation}
   for each $\zeta\in\Cb_{\lbbrb{0,\infty}}$, $\Wkap\lbrb{\zeta,1}=1$; $\Wkap\in \Ac^2_{\intervalOI}$ and for any $q\in\lbbrb{0,\infty}$ the function $\Wkap(q,\cdot)$ is the Mellin transform of a positive random variable.
\end{definition}
\begin{remark}\label{rem:BivBernGam}
	Taking $ \zeta = q \in  [0,\infty)$ in the formula in  \eqref{eq:kappa},  we can write $\kappa\lbrb{q,z}$ as
	\begin{equation}\label{eq:repKappa}
	\begin{split}
	\kappa\lbrb{q,z}&=\kappa\lbrb{q,0}+\dr_2z+\IntOI\lbrb{1-e^{-zx_2}}\lbrb{\IntOI e^{-qx_1}\mu\lbrb{dx_1,dx_2}}\\
	&=: \phq(0)+\dr_2z+\IntOI\lbrb{1-e^{-zx}}\mu_q(dx)=:\phq(z),
	\end{split}
	\end{equation}
	where, crucially, $\phq\in\Bc$. Then, using \eqref{eq:Bern-Gamma}, note that  $\Wkap\lbrb{q,z} \equiv \Wpq(z)$. 
	 This  gives a starting point from which we can begin to understand bivariate Bernstein-gamma functions through known univariate results, % on univariate Bernstein-gamma functions.  
	 then we can extend results from $\Wkap\lbrb{q,z}$, $q\in[0,\infty)$, to $\Wkap\lbrb{\zeta,z}$, $\zeta\in\Cb_{\lbbrb{0,\infty}}$. 
\end{remark}
\begin{remark}\label{rem:BivBernGam1}
	Note the the reduction from the bivariate to the univariate case is simply done by taking $\kappa(q,z)=q+\phi(z), \phi\in\Bc$. This corresponds simply to the killing of one-dimensional subordinator.
\end{remark}
 It is proven in \cite[Section 6]{ps19} that $\Wp$ admits   an absolutely convergent Weierstrass product representation. In the following Theorem \mbox{\ref{thm:BernGammaExt}}, we  extend this infinite product representation of $\Wp$ %, see \eqref{eq:Bern-GammaProd}, 
 to $\Wkap$.

\begin{theorem}\label{thm:BernGammaExt}
	If $\kappa\in\Bc^2$, then $\Wkap$ as in Definition \ref{def:BernGammaBiv} exists and is unique, and the following product representation,  defined for $\lbrb{\zeta,z}\in \Cb_{\lbbrb{0,\infty}}\times \Cb_{\lbrb{0,\infty}}$ by
	\begin{equation}\label{eq:Bern-GammaExt}
	%\Wkap(\zeta,z)=
	\frac{e^{-\gampkap(\zeta) z}}{\kappa\lbrb{\zeta,z}}\prod_{k=1}^{\infty}\frac{\kappa\lbrb{\zeta,k}}{\kappa\lbrb{\zeta,k+z}}e^{\frac{\kappa'_z(\zeta,k)}{\kappa\lbrb{\zeta,k}}z}, 
	\end{equation}
	%where 
	\begin{equation}\label{eq:gamphz}
	\gampkap(\zeta)=\limi{n}\lbrb{\sum_{k=1}^{n}\frac{\kappa'_z(\zeta,k)}{\kappa(\zeta,k)}-\log_0\lbrb{\kappa(\zeta,n)}},  
	\end{equation}
	satisfies  $\l(\ref{eq:BivBernGam}\r)$. In particular,  it follows that  for $\lbrb{\zeta,z}\in \Cb_{\lbbrb{0,\infty}}\times \Cb_{\lbrb{0,\infty}}$  
	\begin{equation}\label{eq:Bern-GammaExt0}
	\Wkap(\zeta,z)=
	\frac{e^{-\gampkap(\zeta) z}}{\kappa\lbrb{\zeta,z}}\prod_{k=1}^{\infty}\frac{\kappa\lbrb{\zeta,k}}{\kappa\lbrb{\zeta,k+z}}e^{\frac{\kappa'_z(\zeta,k)}{\kappa\lbrb{\zeta,k}}z}, 
	\end{equation}
	%$W_\kappa(\zeta,z)=\l(\ref{eq:Bern-GammaExt}\r)$, 
	so that the product    in  $\l(\ref{eq:Bern-GammaExt}\r)$ is indeed a product representation of $W_\kappa$.   
	Moreover, $\gampkap\in\Ac_{\lbbrb{0,\infty}}$,  and for $\lbrb{\zeta,z}\in \Cb_{\lbbrb{0,\infty}}\times \Cb_{\lbrb{0,\infty}}$, we can express $\Wkap(\zeta,z)$ as
	\begin{equation}\label{eq:Bern-GammaExtLim}
	\Wkap(\zeta,z)=\frac{1}{\kappa\lbrb{\zeta,z}}\limi{n}e^{z\log_0\kappa\lbrb{\zeta,n}}\prod_{k=1}^{n}\frac{\kappa\lbrb{\zeta,k}}{\kappa\lbrb{\zeta,k+z}}.
	\end{equation}
	%Finally.
\end{theorem}
The proof of this theorem is provided in Section \ref{sec:ProofsBernGammaBiv}.
We proceed with the derivation of the Stirling type approximation for $\Wkap$. For this purpose we need some more notation. Firstly, we introduce the function $\Lkap$, which  %represents the main part of the Stirling asymptotic of 
contains the main asymptotic contribution in 
$\Wkap$, and is defined, for $\lbrb{\zeta,z}\in \Cb_{\lbbrb{0,\infty}}\times \Cb_{\lbrb{0,\infty}}$, as
\begin{equation}\label{eq:A}
\Lkap\lbrb{\zeta,z}:=\int_{1\rightarrow 1+z}\log_0\lbrb{\kappa\lbrb{\zeta,\chi}}d\chi,
\end{equation}
where the integral denotes the path integral along a contour starting from $1$ and ending at $1+z$ which lies in the domain of analyticity of  $\log_0\lbrb{\kappa\lbrb{\zeta,\cdot}}$. If $\Re (z) >-1$ then there is a straight line connecting $1$ to $1+z$ in the domain of analyticity of $\log_0\lbrb{\kappa\lbrb{\zeta,\cdot}}$ and we have
\begin{equation}\label{eq:A1}
\Lkap\lbrb{\zeta,z}=\int_{1\rightarrow 1+z}\log_0\lbrb{\kappa\lbrb{\zeta,\chi}}d\chi=z\int_{0}^{1}\log_0\lbrb{\kappa\lbrb{\zeta,1+zv}}dv.
\end{equation}
Now, for $\lbrb{\zeta,z}\in \Cb_{\lbbrb{0,\infty}}\times \Cb_{\lbrb{0,\infty}}$, $ \Re\lbrb{\kappa\lbrb{\zeta,z}}>0$ by \eqref{eq:limSup}, so  $\Lkap$ is well-defined on $\Cb_{\lbbrb{0,\infty}}\times\Cb_{\intervalOI}$.
We denote the floor function $\lfloor u\rfloor=\max\curly{n\in\N:\,n\leq u}$,  and define \[\Prm(u):=\lbrb{u-\lfloor u\rfloor}\lbrb{1-\lbrb{u-\lfloor u\rfloor}}.\]  The function $\Ekap$ corresponds to the error term in our Stirling approximation, and is defined, for $\lbrb{\zeta,z}\in \Cb_{\lbbrb{0,\infty}}\times \Cb_{\lbrb{0,\infty}}$,  as
\begin{equation}\label{eq:E}
\Ekap(\zeta,z)=\frac{1}{2}\int_{1}^{\infty}\Prm(u)\lbrb{\log_0\lbrb{\frac{\kappa\lbrb{\zeta,u+z}}{\kappa\lbrb{\zeta,u}}}}''du. % , \qquad  \zeta,z\in\Cb_{\intervalOI} .
\end{equation}
Now we are ready to state Theorem \mbox{\ref{thm:Stirling}}, the Stirling asymptotic representation for $\Wkap$. For the absolute value of the univariate case a similar, but less wieldy, asymptotic representation has been derived in \cite[Theorem 4.2]{ps16}. %, which can be considered as an extension of \cite[Theorem 3.3]{ps16}.
\begin{theorem}\label{thm:Stirling}
	Let $\kappa\in\Bc^2$. Then,  for $\lbrb{\zeta,z}\in \Cb_{\lbbrb{0,\infty}}\times \Cb_{\lbrb{0,\infty}}$, 
	we have that
	\begin{equation}\label{eq:Stirling}  
		\Wkap\lbrb{\zeta,z}=\frac{\kappa^{\frac{1}{2}}\lbrb{\zeta,1}}{\kappa\lbrb{\zeta,z}\kappa^{\frac{1}{2}}\lbrb{\zeta,1+z}}e^{\Lkap\lbrb{\zeta,z}}e^{-\Ekap\lbrb{\zeta,z}},
	\end{equation}
	%We have that
	\begin{equation}\label{eq:Ek}   
		\sup_{\kappa\in\Bc^2}\sup_{\lbrb{\zeta,z}\in\Cb_{\lbbrb{0,\infty}}\times\CbOI}\abs{\Ekap\lbrb{\zeta,z}}\leq 2,
	\end{equation}
	
		\begin{equation}\label{eq:limEk}
	\limi{\Rez}\Ekap\lbrb{\zeta,z}=\frac{1}{2}\int_{1}^{\infty}\Prm(u)\lbrb{\lbrb{\frac{\kappa_z'\lbrb{\zeta,u}}{\kappa\lbrb{\zeta,u}}}^2-\frac{\kappa_z''(\zeta,u)}{\kappa\lbrb{\zeta,u}}}du:=\Tkap\lbrb{\zeta},
	\end{equation}
	and for each fixed $z\in\CbOI$, we have that $\limi{\Re(\zeta)}\Ekap\lbrb{\Reta,z}=0$.
	%\begin{equation}\label{eq:limqEkap}
	%\limi{q}\Ekap\lbrb{q,z}=0.
	%\end{equation}
	%If
	%\begin{equation}\label{eq:WkN}
		%\Wkap\lbrb{\zeta,z,N}:=\frac{1}{\kappa\lbrb{\zeta,z}}e^{z\log_0\kappa\lbrb{\zeta,z}}\prod_{k=1}^{N}\frac{\kappa\lbrb{\zeta,k}}{\kappa\lbrb{\zeta,k+z}}
	%\end{equation} 
	%then for any $N\geq 1$  and $z=\so{N}$
	%\begin{equation}\label{eq:WkApp}
%e^{-\frac{2|z|^2+|z|+2}{N}}	\leq \abs{\frac{\Wkap\lbrb{\zeta,z}}{\Wkap\lbrb{\zeta,z,N}}}\leq e^{\frac{2|z|^2+|z|+2}{N}}.
	%\end{equation} 
\end{theorem}
The proof of this theorem is provided in Section \mbox{\ref{sec:ProofsBernGammaBiv}}. Now we state two key lemmas.
%
%\n Now, before we prove Lemma \mbox{\ref{Wlemma}}, we require Lemma \mbox{\ref{philemma}}, which is stated without proof. 
Lemma  \mbox{\ref{philemma}}, is a complex generalisation of the result  \cite[Prop. 3.1 (8)]{ps16},  that for   all $\phi\in\cal{B}$, $a\in\bb{R}$, and  $u\in (0,\infty)$, uniformly among $a$ in compact intervals in $\bb{R}$,  
\b{equation}
\label{phirealresult}
  \lim_{u \to \infty}    \f{   \phi(u   +  a)}{\phi(u)} =1 .
\e{equation}
  \b{lemma}    \label{philemma}   For each $\phi\in\cal{B}$, $z\in\bb{C}$, uniformly among $z$ in compact subsets of the complex plane $\Cb$, 
  \vspace{-8pt}
\b{equation}
\label{philemmaeqn}
  \lim_{u \to \infty}    \f{   \phi(u   +  z)}{\phi(u)} =1 .
\e{equation}
\e{lemma}
   The next lemma is a generalisation, from the standard gamma function to the class of Bernstein-gamma functions, of the following result. For each $z\in\bb{C}$, %as $n\to\infty$, 
\[
\lim_{u\to\infty}   \f{   \Gamma(u+1)   (u+1)^z   }{\Gamma(u+1+z)}   =1  . 
\]

\b{lemma} \label{Wlemma}
For each $\phi\in\cal{B}$, $z\in\CbOI$, uniformly among $z$ in compact subsets of the complex half-plane $\Cb_{\lbbrb{0,\infty}}$,
  \vspace{-8pt}
\b{equation}
\label{Wlemmaeqn}
\lim_{u\to\infty}  \f{ W_\phi(u+1)  \phi^{z}(u+1)   }{  W_\phi(u+1+z)    }  =1.
\e{equation}
\e{lemma}
The proof of Lemma  \mbox{\ref{Wlemma}} builds upon Lemma  \mbox{\ref{philemma}}. Both proofs  are contained in Section \mbox{\ref{lemmasproofs}}.

\subsection{Applications to Exponential Functionals up to a Finite Time}  \label{expresults}   
 The first of our key results on  exponential functionals up to a finite time is an infinite convolution formula for the Mellin transform: 
 \begin{theorem} \label{convolutionformula}   For each possibly killed subordinator with Laplace exponent $\phi\in \Bc$, for $t\in(0,\infty)$, and for $z\in\cc$, the Mellin transform, $\cal{M}_{I_\phi(t)}(z+1)=\bb{E}[I_\phi^z(t)]$,  satisfies 
  \begin{equation}
  \begin{split}
    \label{convformula}
 \bb{E}[I_\phi^z(t)] = 1 \ast  \f{z t^{z-1}}{ e^{\phi(1)t}     } 
\ast \overset{\infty}{\underset{k=1}{\conv}} \Bigg[   \delta_0(dt)    &+   \sum_{m=1}^\infty   \f{  (z)^{(m)} (-[\phi(k+1)-\phi(k)])^m }{  (m!)^2  }   \f{ m t^{m-1}}{ e^{\phi(k)t}}
\\
 &+  [\phi(z+k)-\phi(k)] \hspace{-2pt}   \sum_{m=0}^\infty      \f{  (z)^{(m)} ( - [\phi(k+1)-\phi(k)])^{m} \hspace{-4pt} }{  (m!)^2  }     \f{t^{m}}{e^{\phi(k)t}  }    \Bigg],
\end{split}
  \end{equation} 
   where the symbol $\ast$ denotes convolution $(f\ast g)(t):= \int_0^t f(s) g(t-s)ds$, the symbol $\conv$ means infinite convolution, $\delta_0(dt)$ is the Dirac measure at the point $0$, and $(z)^{(m)} := z(z+1)\cdots(z+m-1)$ is the rising factorial function. 

 \end{theorem}
 While this result is interesting on its own, it makes the computations of the moments hard, and for this purpose we shall express our formula as an infinite sum rather than an infinite convolution. % work on this Mellin transform,  % This requires a dominated convergence type argument, and we need to impose 
This requires a slight regularity condition on the underlying Bernstein function:% so that absolute convergence holds throughout our proof:
 \begin{definition}[Regularity Condition]
 \label{regularitycondition} We impose that the derivative of our subordinator's Laplace exponent satisfies $\beta(\phi')>-1$, where $\beta(\phi')$ denotes the lower Matuszewska index of $\phi'$, defined as the infimum of $\beta\in\bb{R}$ for which there exists $C>0$ such  that for each $\Lambda>1$, $\phi'(\lambda x)/\phi'(x) \geq (1+o(1))C \lambda^\beta$, uniformly in $\lambda\in [1, \Lambda] $, as $x\to\infty$.      See \cite[p68]{bgt89} for more details but we highlight that cases like $\phi\in \Bc$ with $d>0$ and/or $\bar{\Pi}(x)\sim x^{-\alpha}, x\to 0, \alpha\in\lbrb{0,1}$ satisfy this condition but are a small sample of cases that fall under this definition.
 \end{definition}
 
%  \begin{definition}[Lower Matuszewska Index]
% \label{matusz}  We define   
% \end{definition}

 \b{theorem}
 \label{mainthm} For each subordinator whose Laplace exponent $\phi$  satisfies the condition in Definition \mbox{\ref{regularitycondition}}, for all $t\in(0,\infty]$,  
  and for all $z\in\cc$,
 \b{equation}
\label{theoremformula}
\cal{M}_{I_\phi(t)}(z+1)=\bb{E}[I_\phi^z(t)]  =   \f{ \Gamma(z+1)   }{  W_{\phi}(z+1)    }  -  
 \sum_{k=1}^\infty      \f{       \prod_{j=1}^{k}     [ \phi(z+j) - \phi(k)]    }{     \prod_{j=1}^{k-1}       [ \phi(j) - \phi(k)]  } 
 \f{     e^{-\phi(k)t}     }{   \phi(k)}   \f{ \Gamma(z+1)     }{   W_{\phi_{(k)}}(z+1)  } ,  
\e{equation}
  where    $ \phi_{(k)}(w):=  \phi(w+k)- \phi(k)$ is a Bernstein function,  and $W_{ \phi_{(k)}}$ is  its corresponding Bernstein-gamma function.
 \e{theorem}

 \b{remark} Our constraint excludes  cases in which $\phi$ is \textit{slowly varying} (see \cite[p6]{bgt89}), but it should be noted that even for e.g.\ $\phi(x)=\ln(1+x)$,  Theorem \mbox{\ref{mainthm}} still holds in a region of the form $\Re(z)>c_{t,\phi}$, with the region depending on  $t$ and $\phi$. 
We are unable to obtain any partial results in only the most pathological cases, e.g.\ $\phi(x)=\ln(\ln(e+x))$, for which the rate of growth of $\phi(x)$ to $\infty$, as $x\to\infty$,  is too slow. We emphasise that the Matuszewska indices of the derivatives of the Bernstein functions in this remark are of value precisely $-1$.
%The difficulty lies in finding suitable bounds for the product in $\l(\ref{productbounds}\r)$.
 \e{remark}
 \begin{remark}\label{rem:minusMom}
   Note that since $\int_{ 0}^{e_q}e^{-X_s}ds$ has negative moments of order betwen $(-1,0)$, see [Theorem 2.4]\cite{ps16}, then clearly $I_{\phi}(t)$ has the same moments for any $t>0$. Their evaluation is excluded in the statement of this theorem as there is a technical difficulty to obtain a similar expression. However, from \eqref{eq:LT} below, one can attack the problem from the fact that
   \[ \f{\Gamma(z+1)      }{   W_{ \varphi_q}(z+1)      }=\frac{z}{\varphi_q(z)} \f{\Gamma(z)      }{   W_{ \varphi_q}(z)      },\]
   where $\varphi_q(z)=\phi(z)+q$ and therefore $\bb{E}[I_\phi^z(t)]$ is the convolution in $t$ of $\bb{E}[I_\phi^{z-1}(t)]$ with $ze^{-\phi(z)t}$ or
   \begin{equation*}
   \begin{split}
   &		   \bb{E}[I_\phi^z(t)] =ze^{-\phi(z)t}\int_{  0}^t\bb{E}[I_\phi^{z-1}(s)]e^{\phi(z)s}ds.
   \end{split}
   \end{equation*}
   This equation can be analysed by suitable differentiation and rearrangement.
 \end{remark}
 One can quite easily obtain, using elementary methods, an explicit formula for the positive integer moments of the exponential functional of a subordinator, as it appears in the recent work \cite[Corollary 1]{sv18}. 
\begin{theorem}[Salminen, Vostrikova 2018]\label{thm:SV} For all $\phi\in\cal{B}$ with $\phi(0)=0$, and for all $n\in\bb{N}$,
\begin{equation} 
\label{sv}
\Ebb{I_\phi^n(t)}=      n!  \sum_{k=0}^{n-1}    \f{  e^{-\phi(k)t} -  e^{-\phi(n)t}    }{ \    \underset{0\leq j \leq n; \hspace{2pt}  j\neq k}{\prod}     [ \phi(j) - \phi(k)]      }.
\end{equation}
\end{theorem}
Using our methodology, based on Laplace inversion, 
we can deduce the following formula for integer moments, which extends \cite[Corollary 1]{sv18} to integer moments of killed subordinators.
\begin{corollary} 
\label{integermoments}
 For each subordinator with Laplace exponent $\phi\in\cal{B}$, for $t\in(0,\infty)$ and $n\in\bb{N}$,
 \b{equation}
 \label{intmoments}
 \begin{split}
 	\Ebb{I_\phi^n(t)} &=  n!   \sum_{k=0}^{n-1}     \f{       e^{-\phi^*(k)t}-e^{-\phi^*(n)t}     }{  \   \underset{0\leq j \leq n; \hspace{2pt}  j \neq k}{\prod}     [ \phi^*(j) - \phi^*(k)]  }\\
 	&=n!\sum_{k=0}^{n}     \f{       e^{-\phi^*(k)t}}    {  \   \underset{0\leq j \leq n; \hspace{2pt}  j \neq k}{\prod}     [ \phi^*(j) - \phi^*(k)]  },
 \end{split}
 \e{equation}
 where $\phi^*(j)=\phi(j),j\geq 1,$ and $\phi^*(0)=0$ and if $\phi(0)=0$ then $\phi^*=\phi$.
 %\b{equation}
 %\label{intmoments}
%\cal{M}_{I_\phi(t)}(n+1)  =\Ebb{I_\phi^n(t)} =  n!   \sum_{k=0}^{n}     \f{       e^{-\phi(k)t}     }{  \   \underset{0\leq j \leq n; \hspace{2pt}  j \neq k}{\prod}     [ \phi(j) - \phi(k)]  }.
 %\e{equation}
\end{corollary}
\begin{remark}\label{rem:mom}
Note that as $t\to\infty$ the second relation of \eqref{intmoments} yields the well known formula
\begin{equation*}
\begin{split}
&		\Ebb{I_\phi^n(t)}=\frac{n!}{\prod_{j=1}^n\phi(j)}    
\end{split}
\end{equation*}
but it also offers an asymptotic expansion of the speed of convergence in $t$, the first term of which is exponential of value $e^{-\phi(1)t}$.
\end{remark}
We proceed with the proofs of our results.
%\n  Compare this with Salminen and Vostrikova's result:
%
%
%
\begin{comment}\n Note that the last result holds also for $\phi(0)>0$. While the results $\l(\ref{sv}\r), \l(\ref{intmoments}\r)$ look subtly different, one can verify that they are in fact equivalent: 
\begin{proposition} 
\label{svcomparison}
For each Bernstein function $\phi$, the above formulae agree, so that $\l(\ref{sv}\r)=\l(\ref{intmoments}\r)$.
%\b{equation} 
%n! \  \sum_{k=0}^{n} \ \ \     \f{       e^{-\phi(k)t}     }{   \underset{0\leq j \leq n; \  j \neq k}{\prod}     [ \phi(j) - \phi(k)]  } =    n! \ \sum_{k=0}^{n-1} \ \ \    \f{  e^{-\phi(k)t} -  e^{-\phi(n)t}    }{  \underset{0\leq j \leq n; \  j\neq k}{\prod}     [ \phi(j) - \phi(k)]      }.
%\e{equation}
\end{proposition}
\end{comment}

  \section{Proofs for Bivariate Bernstein-Gamma Functions}      
  \label{sec:ProofsBernGammaBiv}    
  Before providing the proofs for Section \mbox{\ref{bgresults}}, we state some key results on the class of Bernstein functions $\Bc$, %that will  be useful also in several remaining parts of the paper.  The following results
  which  can be found in \cite[Section 4]{ps19}.
\begin{proposition}\label{propAsymp1}     
	%Let $\phi\in\Bc$.
	\begin{enumerate}
		\item  For all $\phi\in\Bc$ and $z\in\Cb_{\intervalOI}$, we can express the derivative of $\phi$ as
		\begin{equation}\label{eq:phi'}
		\begin{split}
		\phi'(z)&=\dr +\IntOI ye^{-zy}\Pi(dy)\\
		&=\dr+\IntOI e^{-zy}\bar{\Pi}(y)dy-z\IntOI e^{-zy}y\bar{\Pi}(y)dy.
		\end{split}	
		\end{equation}
		\item \label{it:bernstein_cm} Each $\phi\in\cal{B}$ is  non-decreasing on $[0,\infty)$, and
		$\phi'$ is completely monotone,  positive, and non-increasing on $[0,\infty)$. Hence, $\phi$ is strictly log-concave on $[0,\infty)$, so  for all $u\in [0,\infty)$,
		\begin{equation}\label{eq:phi'_phi}
		0\leq u \phi'(u)=\phi(u)-\phi(0)-u^2\int_{0}^{\infty} e^{-uy}y\bar\Pi(y)dy\leq \phi(u)
		.
		\end{equation}
		%and
		%\begin{equation}\label{specialEstimates11}
		%|\phi''(u)|\leq 2 \frac{\phi(u)-\phi(0)}{u^2}\leq 2\frac{\phi(u)}{u^2}.
		%\end{equation}
		%\item \label{it:asyphid}  $\phi(u)\stackrel{\infty}{=} \dr u +\so{u}$ and $\phi'(u)\stackrel{\infty}{=}\dr+\so{1}$. Fix  $a>\aph$, then $\labsrabs{\phi\lbrb{a+ib}}=\labsrabs{a+ib}\lbrb{ \dr+\so{1}}$ as $|b|\to\infty$.
		%\item\label{it:finPhi} If $\phi\lbrb{\infty}<\infty$ and $\mu$ is absolutely continuous then for any fixed $a>\aph$, $\limi{|b|}\phi\lbrb{\ab}=\phi\lbrb{\infty}$.
		%\item \label{it:bernstein_cmi} The mapping $u\mapsto \frac{1}{\phi(u)},\, u\in \R^+,$ is completely monotone, i.e.~there exists a positive measure $U$, whose support is contained in $\lbbrb{0,\infty}$, called the potential measure, such that the Laplace transform of $U$ is given via the identity
		%\[ \frac{1}{\phi(u)} = \int_0^{\infty} e^{-uy}U(dy).\]
		%\item \label{it:bernstein_log_concavity_u/p} \mladen{\textbf{used?}}The mapping $u\mapsto \frac{u}{\phi(u)}$ is positive and log-concave on $\R^+$.
		%\item In any case, \label{it:flatphi}
		%\begin{equation}\label{lemmaAsymp1-1}
		%\lim_{u\to\infty}\frac{\phi(u\pm a)}{\phi(u)}=1\,\,\text{ uniformly for $a$-compact intervals  on $\R^+$.}
		%\end{equation}
		%\item\label{it:unif_rev1}
		%Uniformly, for $u\in\R^+$, we have that\label{it:asympphi}
		%\begin{equation} \label{eq:asympphi}
		%\phi(u) \asymp u \int_0^{\frac{1}{u}} \bar{\mu}(y) dy + \sigma^2 u + m.
		%\end{equation}
	\end{enumerate}
\end{proposition}
We are ready to start with the proof of the main results.
\begin{proof}[Proof of Theorem \mbox{\ref{thm:BernGammaExt}}]
	Since for $\zeta=q\geq 0$ we have that $\Wkap\lbrb{q,z}=\Wpq(z)$ with $\gamma_\kappa(q)$ computed as in \eqref{eq:gamphz}, see Remark \mbox{\ref{rem:BivBernGam}}, then 
	\begin{equation}\label{eq:Bern-GammaProd2}
	\Wkap\lbrb{q,z}=\frac{e^{-\gampkap(q) z}}{\kappa(q,z)}\prod_{k=1}^{\infty}\frac{\kappa(q,k)}{\kappa(q,k+z)}e^{\frac{\kappa'_z(q,k)}{\kappa(q,k)}z}=:\overline{\Wkap}\lbrb{q,z},\,z\in\CbOI.
	\end{equation}
	We  extend analytically $\overline{\Wkap}\lbrb{\cdot,z}$ to $\Cb_{\lbbrb{0,\infty}}$. First, we extend  $\gampkap(\cdot)$. Fix $\zeta\in\Cb_{\intervalOI}$ and an open ball $B_\zeta$ centred at $\zeta$ with the closed ball $\overline{B_\zeta}$ satisfying $\overline{B_\zeta}\subset\Cb_{\intervalOI}$. From item\,\mbox{\ref{it:limit}} of Proposition \mbox{\ref{prop:kappa}}, that is $\Re(\kappa(\zeta, k))>0, k\geq 1$, and $\kappa\in \Ac^2_{\intervalOI}$, see Lemma \mbox{\ref{lem:repKappa1}}, we deduce that for any $n\in\Nbp:=\curly{1,2,\cdots}$ 
	\begin{equation}\label{eq:fn}
		f_n(\zeta):=\sum_{k=1}^n\frac{\kappa'_z\lbrb{\zeta,k}}{\kappa(\zeta,k)}-\log_0\lbrb{\kappa\lbrb{\zeta,n}}\in\Ac_{\intervalOI}.
	\end{equation}
	Moreover, we observe that
	\begin{equation}\label{eq:f_nEst}
	\begin{split}
	\abs{f_n(\zeta)}&=\abs{\sum_{k=1}^n\frac{\kappa'_z\lbrb{\zeta,k}}{\kappa(\zeta,k)}-\int_{1}^{n}\frac{\kappa'_z\lbrb{\zeta,x}}{\kappa(\zeta,x)}dx-\log_0\lbrb{\kappa\lbrb{\zeta,1}}}\\
	&=\abs{\sum_{k=1}^n\int_{0}^{1}\lbrb{\frac{\kappa'_z\lbrb{\zeta,k}}{\kappa(\zeta,k)}-\frac{\kappa'_z\lbrb{\zeta,x+k}}{\kappa(\zeta,x+k)}}dx-\log_0\lbrb{\kappa\lbrb{\zeta,1}}}\\
	&\leq \sum_{k=1}^n\sup_{y\in\lbbrbb{0,1}}\lbrb{\abs{\frac{\kappa''_z\lbrb{\zeta,y+k}}{\kappa(\zeta,y+k)}}+\abs{\frac{\kappa'_z\lbrb{\zeta,y+k}}{\kappa(\zeta,y+k)}}^2}+\abs{\log_0\lbrb{\kappa\lbrb{\zeta,1}}}\\
	&\leq \sum_{k=1}^n\frac{8}{k^2}+\abs{\log_0\lbrb{\kappa\lbrb{\zeta,1}}}\\
	&<\frac{8\pi^2}{6}+\abs{\log_0\lbrb{\kappa\lbrb{\zeta,1}}},
	\end{split}
	\end{equation}
	where in the first inequality we have used first order Taylor's expansion of $\frac{\kappa'_z\lbrb{\zeta,x+k}}{\kappa(\zeta,x+k)}$ about $x=0$ and immediate bounds, in the second we have applied \eqref{eq:kappa'_kappa} and  \eqref{eq:kappa''_kappa}, and  the third follows from the well-known $\sum_{k=1}^\infty k^{-2}=\pi^2/6$.
	Henceforth, 
	\begin{equation*}
	\begin{split}
		&\sup_{n\geq 1} \sup_{\chi\in\overline{B_\zeta}}\abs{f_n\lbrb{\chi}}<\frac{4\pi^2}{3}+\sup_{\chi\in\overline{B_\zeta}}\abs{\log_0\lbrb{\kappa\lbrb{\chi,1}}}<\infty.
	\end{split}
	\end{equation*}	
	Therefore, $f_n$ are uniformly bounded holomorphic functions on $\overline{B_\zeta}$ and from the dominated convergence theorem which is applicable thanks to \eqref{eq:f_nEst} we deduce that
	\[\limi{n}f_n\lbrb{\chi}=\sum_{k=1}^{\infty}\int_{0}^{1}\lbrb{\frac{\kappa'_z\lbrb{\chi,k}}{\kappa(\chi,k)}-\frac{\kappa'_z\lbrb{\chi,x+k}}{\kappa(\chi,x+k)}}dx-\log_0\lbrb{\kappa\lbrb{\chi,1}}:=f(\chi).\]
	From classical result of complex analysis we conclude that $f\in\Ac_{B_\zeta}$ and since $\zeta\in\Cb_{\intervalOI}$ we get that $f\in\Ac_{\intervalOI}$. Since by the very definition of $\gamma_\kappa$, see \eqref{eq:gamphz}, $\gamma_\kappa=\limi{n}f_n=f$ we conclude that $f=\gamma_\kappa\in\Ac_{\intervalOI}$. The claim that $\gamma_\kappa\in\Ac_{\lbbrb{0,\infty}}$ is then affirmed by the fact that $\kappa\lbrb{\zeta,k}, \kappa'_z\lbrb{\zeta,x+k},\,k\geq1, x\in\lbbrbb{0,1},$ extend continuously to $\zeta\in i\R$ and that the uniform bound \eqref{eq:f_nEst} is valid for $\zeta \in i\Rb$.  In view of $\gamma_\kappa\in\Ac_{\lbbrb{0,\infty}}$ to analytically extend in $\zeta$ the right-hand side of  the infinite product in \eqref{eq:Bern-GammaExt}, it suffices from Hartog's theorem, see \cite[Section 2.4]{k01}, to fix $z\in\Cb_{\intervalOI}$ and show using Montel's theorem that the infinite product, which we record in \eqref{eq:InfProduct} below, converges absolutely
	\begin{equation}\label{eq:InfProduct}
	Z_\kappa(\zeta,z):=\prod_{k=1}^{\infty}\frac{\kappa\lbrb{\zeta,k}}{\kappa\lbrb{\zeta,k+z}}e^{\frac{\kappa'_z(\zeta,k)}{\kappa\lbrb{\zeta,k}}z}=\prod_{k=1}^{\infty}A_k\lbrb{\zeta}.
	\end{equation}
	For any $k\geq 1$, we get using the Taylor's expansion
	\begin{equation}\label{eq:TaylorRatio}
	\log_0\lbrb{\frac{\kappa\lbrb{\zeta,k+z}}{\kappa\lbrb{\zeta,k}}}=\log_0\lbrb{1+z\frac{\kappa'_z\lbrb{\zeta,k}}{\kappa\lbrb{\zeta,k}}+\frac{z^2}2\frac{\kappa''_z\lbrb{\zeta,k+a_kz}}{\kappa\lbrb{\zeta,k}}},
	\end{equation}
	where $a_k\in\lbbrbb{0,1}$. Thus,
	\begin{equation}\label{eq:Ak}
		A_k\lbrb{\zeta}=e^{z\frac{\kappa'_z(\zeta,k)}{\kappa\lbrb{\zeta,k}}-\log_0\lbrb{1+z\frac{\kappa'_z\lbrb{\zeta,k}}{\kappa\lbrb{\zeta,k}}+\frac{z^2}2\frac{\kappa''_z\lbrb{\zeta,k+a_kz}}{\kappa\lbrb{\zeta,k}}}}.
	\end{equation}
	From \eqref{eq:kappa'_kappa} we have that on any ball $\overline{B_\zeta}\subset \CbOI$ centred at $\zeta$ and any $k\geq 1$
	\begin{equation}\label{eq:Ak1}
		\sup_{\chi\in \overline{B_\zeta}}\abs{z\frac{\kappa'_z\lbrb{\chi,k}}{\kappa\lbrb{\chi,k}}}\leq 2\frac{|z|}{k}.
	\end{equation}
	From \eqref{eq:kappa''}, \eqref{eq:ratioBounds}, \eqref{eq:kappa''_kappa} and the fact that $\abs{\phReta''}=-\phReta''$ decreases on $[0,\infty)$ we also observe that for any $k\geq 1$
	\begin{equation}\label{eq:Ak2}
	\begin{split}
	&\sup_{\chi\in \overline{B_\zeta}}\sup_{v\in\lbbrbb{0,1}}\abs{z\frac{\kappa''_z\lbrb{\chi,k+vz}}{\kappa\lbrb{\chi,k}}}\leq \sup_{\chi\in \overline{B_\zeta}}\sup_{v\in\lbbrbb{0,1}}\abs{z\frac{\kappa''_z\lbrb{\Re(\chi),k+v\Re(z)}}{\kappa\lbrb{\chi,k}}}\\
	&\leq \sup_{\chi\in \overline{B_\zeta}}\abs{z\frac{\kappa''_z\lbrb{\Re(\chi),k}}{\kappa\lbrb{\chi,k}}}\leq\sup_{\chi\in \overline{B_\zeta}}\abs{z\frac{\kappa''_z\lbrb{\Re\lbrb{\chi},k}}{\kappa\lbrb{\Re\lbrb{\chi},k}}}\\
	& \leq 4\frac{|z|}{k^2}.
	\end{split}
	\end{equation}
	Therefore, using $\log_0\lbrb{1+w}= w+\bo{w^2}$, as $|w|\sim 0$, we get for all $k$ large enough
		\begin{equation*}
		\begin{split}
	\sup_{\chi\in \overline{B_\zeta}}\abs{A_k\lbrb{\chi}}\leq  \sup_{\chi\in \overline{B_\zeta}}\abs{e^{\frac12\lbrb{z\frac{\kappa'_z\lbrb{\chi,k}}{\kappa\lbrb{\chi,k}}+\frac{z^2}2\frac{\kappa''_z\lbrb{\zeta,k+a_kz}}{\kappa\lbrb{\zeta,k}}}^2-\frac{z^2}2\frac{\kappa''_z\lbrb{\zeta,k+a_kz}}{\kappa\lbrb{\zeta,k}}}}\leq e^{\frac{16}{k^2}\lbrb{|z|^2+|z|^4+\bo{\frac1k}}}
		\end{split}
	\end{equation*}
	and hence from \eqref{eq:InfProduct}
		\begin{equation*}
		\sup_{\chi\in \overline{B_\zeta}}\abs{Z_\kappa(\chi,z)}=\sup_{\chi\in \overline{B_\zeta}}\prod_{k=1}^\infty\abs{A_k\lbrb{\chi}}<\infty.
	\end{equation*}
	This shows that $\overline{\Wkap}\in\Ac^2_{\intervalOI}$ and it equates the infinite product in \eqref{eq:Bern-GammaExt}. To deduce that $\Wkap=\overline{\Wkap}$ and therefore $\Wkap$ is represented as in \eqref{eq:Bern-GammaExt} we proceed to demonstrate first that \eqref{eq:Bern-GammaExtLim} holds. For this it suffices to use that the infinite product in \eqref{eq:Bern-GammaExt} is absolutely convergent and the limit that deifines $\gampkap$, see \eqref{eq:gamphz}. Indeed, we simply write
	\begin{equation*}
	\begin{split}
	\overline{\Wkap}\lbrb{\zeta,z}&=\frac{e^{-\gampkap(\zeta) z}}{\kappa\lbrb{\zeta,z}}\limi{n}\prod_{k=1}^{n}\frac{\kappa\lbrb{\zeta,k}}{\kappa\lbrb{\zeta,k+z}}e^{\frac{\kappa'_z(\zeta,k)}{\kappa\lbrb{\zeta,k}}z}\\
	&=\frac{1}{\kappa\lbrb{\zeta,z}}\limi{n}e^{z\log_0\kappa\lbrb{\zeta,n}}e^{z\sum_{k=1}^n\frac{\kappa'_z(\zeta,k)}{\kappa\lbrb{\zeta,k}}-\gampkap(\zeta) z-\log_0\kappa\lbrb{\zeta,n}}\prod_{k=1}^{n}\frac{\kappa\lbrb{\zeta,k}}{\kappa\lbrb{\zeta,k+z}}\\
	&=\frac{1}{\kappa\lbrb{\zeta,z}}\limi{n}e^{z\log_0\kappa\lbrb{\zeta,n}}\prod_{k=1}^{n}\frac{\kappa\lbrb{\zeta,k}}{\kappa\lbrb{\zeta,k+z}}.
	\end{split}
	\end{equation*}
	Then
	\begin{equation*}
	\begin{split}
	\overline{\Wkap}\lbrb{\zeta,z+1}&=\limi{n}e^{z\log_0\kappa\lbrb{\zeta,n}}\frac{\kappa\lbrb{\zeta,n}}{\kappa\lbrb{\zeta,n+1}}\prod_{k=1}^{n+1}\frac{\kappa\lbrb{\zeta,k}}{\kappa\lbrb{\zeta,k+z}}\\
	&=\kappa\lbrb{\zeta,z}\overline{\Wkap}\lbrb{\zeta,z}
	\end{split}
	\end{equation*}
	provided $\limi{n}\frac{\kappa\lbrb{\zeta,n+1}}{\kappa\lbrb{\zeta,n}}=1$. This is an elementary consequence of \[\kappa\lbrb{\zeta,n+1}=\kappa\lbrb{\zeta,n}+\int_{n}^{n+1}\kappa'_z\lbrb{\zeta,x}dx,\] inequality \mbox{\eqref{eq:kappa'_kappa_1}} and item \mbox{\ref{it:limi}} of Lemma \mbox{\ref{prop:kappa}}. Clearly, $\overline{\Wkap}\lbrb{\zeta,1}=1$.  Now, $\overline{\Wkap}$ satisfies all conditions of Definition \ref{def:BernGammaBiv} and therefore as it coincides with $\Wkap$ for $\zeta=q\in\lbbrb{0,\infty}$ we deduce they coincide on $\Cb_{\lbbrb{0,\infty}}\times\Cb_{\intervalOI}$. The uniqueness of $\Wkap$ follows from the fact that the positive random variable associated with the Mellin transform $\Wkap(q,z)$ is unique, for each $q\geq 0$, see e.g. \cite[Section 4]{ps16} and the discussion in \cite[Section 6]{ps19}. This concludes the proof of the theorem.	 
\end{proof}
Next, we prove the Stirling asymptotic representation for a bivariate Bernstein-gamma function.
\begin{proof}[Proof of Theorem \ref{thm:Stirling}]
	%To develop the Stirling asymptotic representation of $\Wkap$ we use the limit representation of $\Wkap$, which we record and rewrite here in a different and more useful manner. 
	Let us denote $f(x):=\log_0\lbrb{\kappa\lbrb{\zeta,x+z}/\kappa\lbrb{\zeta,x}}$ and
	$S_n:=\sum_{k=1}^n f(k)$. Then applying $\l(\ref{eq:Bern-GammaExtLim}\r)$,  for each $\lbrb{\zeta,z}\in \Cb_{\lbbrb{0,\infty}}\times \Cb_{\lbrb{0,\infty}}$, we can express $\Wkap(\zeta,z)$ as
	\begin{align}%   \label{eq:Bern-GammaExtLim_1}
	%\begin{split}
	\Wkap(\zeta,z)
	&=
	\frac{1}{\kappa\lbrb{\zeta,z}}\limi{n}e^{z\log_0\kappa\lbrb{\zeta,n}}\prod_{k=1}^{n}\frac{\kappa\lbrb{\zeta,k}}{\kappa\lbrb{\zeta,k+z}}
	\\
	&=
	\frac{1}{\kappa\lbrb{\zeta,z}}\limi{n}e^{z\log_0\kappa\lbrb{\zeta,n}-\sum_{k=1}^{n}\log_0\lbrb{\frac{\kappa\lbrb{\zeta,k+z}}{\kappa\lbrb{\zeta,k}}}}
	\\
	%\\
	%&=\limi{n}\Wkap\lbrb{\zeta,z,n},
	%\end{split}
	%\end{equation}
	%see  and \eqref{eq:WkN}. 
	%D we can 
		%\begin{equation}
		\label{eq:Bern-GammaExtLim_2}
	%\begin{split}
	%\Wkap(\zeta,z)
	&=
	\frac{1}{\kappa\lbrb{\zeta,z}}\limi{n}e^{z\log_0\kappa\lbrb{\zeta,n}-S_n}.
	\end{align}
	Then applying \cite[Section 8.2, (2.01),(2.03)]{o97} with $m=1$ in their notation, we can write $S_n$ as% arrive at %the following integral approximation of $S_n$
	\begin{equation}\label{eq:SnIn}
	\begin{split}
		S_n=\sum_{k=1}^n f(k)=\frac{1}{2}\lbrb{f(1)+f(n)}+\int_{1}^n f(x)dx+\frac{1}{2}\int_{1}^n\Prm(u)f''(u)du,
	\end{split}
	\end{equation}
	where recalling that $\Prm(u)=\lbrb{u-\lfloor u \rfloor}\lbrb{1-\lbrb{u-\lfloor u \rfloor}}$, we set
	\begin{equation}\label{eq:EkN}
	\Ekap\lbrb{\zeta,z,n}=\frac{1}{2}\int_{1}^n\Prm(u)f''(u)du=\frac{1}{2}\int_{1}^n\Prm(u)\lbrb{\log_0 \frac{\kappa(\zeta,u+z)}{\kappa\lbrb{\zeta,u}}}''du.
	\end{equation}
	Let us first show that 
	\begin{equation}\label{eq:EknEk}
		\limi{n}\Ekap\lbrb{\zeta,z,n}=\Ekap\lbrb{\zeta,z}
	\end{equation}
	and estimate uniformly $\abs{\Ekap}$. For this purpose we estimate using  \eqref{eq:kappa'_kappa_1} and \eqref{eq:kappa''_kappa_1} that
	\begin{equation}\label{eq:boundLog''}
	\begin{split}
	\abs{\lbrb{\log_0 \frac{\kappa(\zeta,u+z)}{\kappa\lbrb{\zeta,u}}}''}&\leq \abs{\frac{\kappa_z''(\zeta,u+z)}{\kappa(\zeta,u+z)}}+\abs{\frac{\kappa_z'(\zeta,u+z)}{\kappa(\zeta,u+z)}}^2+\abs{\frac{\kappa_z'(\zeta,u)}{\kappa(\zeta,u)}}^2+\abs{\frac{\kappa_z''(\zeta,u)}{\kappa(\zeta,u)}}\\
	&\leq \frac{8}{\lbrb{u+\Re(z)}^2}+\frac{8}{u^2}\leq \frac{16}{u^2},
	\end{split}
	\end{equation}
	where in the last inequality we have invoked the fact that $\Rez\geq 0$. Therefore with the help of the last inequality and 
	\begin{equation}\label{eq:Prm}
	\sup_{u\geq 1}\Prm(u)=\sup_{u\geq 1}\lbrb{u-\lfloor u\rfloor}\lbrb{1-\lbrb{u-\lfloor u\rfloor}}=\frac{1}{4}
	\end{equation}
	 we obtain that
	\begin{equation*}
	\begin{split}
		\sup_{\kappa\in\Bc^2}\sup_{\lbrb{\zeta,z}\in\Cb_{\lbbrb{0,\infty}}\times\Cb_{\intervalOI}}\sup_{n\geq 1}\abs{\Ekap\lbrb{\zeta,z,n}}\leq \frac{1}{8}\sup_{n\geq 1}\int_{1}^{n}\frac{16}{u^2}du\leq 2
	\end{split}
	\end{equation*}
	and hence \eqref{eq:EknEk}. The bound \eqref{eq:Ek} also follows from \eqref{eq:boundLog''}. The limit \eqref{eq:limEk} is easily deduced from  \eqref{eq:boundLog''} wherein the first two terms in the upper bound vanish as $\Rez\to\infty$. Let us show that for any fixed $z$, $\limi{\Re(\zeta)}\Ekap\lbrb{\Re(\zeta),z}=0$. From \eqref{eq:kappa} we easily get that for any $z\in\Cb_{\intervalOI}$
	\begin{equation}\label{eq:kinfz}
		\limi{\Re(\zeta)}\kappa\lbrb{\Re(\zeta),z}=\kappa(0,0)+d_2z+\infty\ind{d_1>0}+\mu\lbrb{\intervalOI,\intervalOI}=\kappa(\infty,z),
	\end{equation}
	whereas from \eqref{eq:kappa'} we get since $\limi{\Re(\zeta)}\mu_{\Re(\zeta)}(dy)\stackrel{w}{=}0dy$ and 
	\[\sup_{\Re(\zeta)\geq 1}\abs{\IntOI y\mu_{\zeta}(dy)}\leq \IntOI y\mu_1(dy)<\infty\]
	that for any $z\in\Cb_{\intervalOI}$
	\begin{equation}\label{eq:k'infz}
	\limi{\Re(\zeta)}\kappa'_z\lbrb{\zeta,z}=d_2.
	\end{equation}
	By the same reasoning $\limi{\Re(\zeta)}\kappa''_z(\Re(\zeta),z)=0$. This together with \eqref{eq:kinfz}, \eqref{eq:k'infz} and the dominated convergence theorem (applicable due to \eqref{eq:boundLog''}) yield that
	\begin{equation*}
	\begin{split}
	&\limi{\Re(\zeta)}\Ekap\lbrb{\zeta,z}\\
	&=\frac{1}{2}\limi{\Re(\zeta)}\int_{1}^{\infty}\Prm(u)\lbrb{{\frac{\kappa_z''(\Re(\zeta),u+z)}{\kappa(\Re(\zeta),u+z)}}-\frac{\lbrb{\kappa_z'(\Re(\zeta),u+z)}^2}{\kappa^2(\Re(\zeta),u+z)}-\frac{\kappa_z''(\Re(\zeta),u)}{\kappa(\Re(\zeta),u)}+\frac{\lbrb{\kappa_z'(\Re(\zeta),u)}^2}{\kappa^2(\Re(\zeta),u)}}du\\
	&=\frac{1}{2}\int_{1}^{\infty}\Prm(u)\limi{\Re(\zeta)}\lbrb{{-\frac{\lbrb{\kappa_z'(\zeta,u+z)}^2}{\kappa^2(\zeta,u+z)}+\frac{\lbrb{\kappa_z'(\zeta,u)}^2}{\kappa^2(\zeta,u)}}}du\\
	&=\frac{1}{2}\int_{1}^{\infty}\Prm(u)\lbrb{{-\frac{d_2^2}{\kappa^2(\infty,u+z)}+\frac{d_2^2}{\kappa^2(\infty,u)}}}du.
	\end{split}	
	\end{equation*}
	Therefore, since  $\abs{\kappa(\infty,z)}=\infty$ whenever $d_2>0$
	it follows that $\limi{\Re(\zeta)}\Ekap\lbrb{\Re(\zeta),z}=0$.
	 Next, we investigate the term $f\lbrb{n}$ in \eqref{eq:SnIn}. We write
	\begin{equation}\label{eq:fn_1}
	\abs{f(n)}=\abs{\log_0\lbrb{1+\frac{\kappa\lbrb{\zeta,n+z}-\kappa\lbrb{\zeta,n}}{\kappa\lbrb{\zeta,n}}}}.
	\end{equation}
	Clearly, from \eqref{eq:kappa''}, Proposition \mbox{\ref{propAsymp1}} \mbox{\ref{it:bernstein_cm}} and $\Rez\geq 0$
	\begin{equation*}
		\begin{split}
		\abs{\kappa\lbrb{\zeta,n+z}-\kappa\lbrb{\zeta,n}}&\leq|z|\sup_{v\in\lbbrbb{0,1}}\abs{\kappa'_z\lbrb{\zeta,n+vz}}\leq |z|\sup_{v\in\lbbrbb{0,1}}\kappa_z'\lbrb{\Reta,n+v\Rez}\\
		&=|z|\sup_{v\in\lbbrbb{0,1}}\phReta'\lbrb{n+v\Rez}\leq |z|\phReta'(n).
		\end{split}
	\end{equation*}
	Thus, from \eqref{eq:ratioBounds} and \eqref{eq:kappa'_kappa} we arrive at
	\[\abs{\frac{\kappa\lbrb{\zeta,n+z}-\kappa\lbrb{\zeta,n}}{\kappa\lbrb{\zeta,n}}}\leq |z|\frac{\phReta'(n)}{\phReta(n)}\leq \frac{|z|}n.\]
	Thus, if $|z|=\so{n}$
\begin{equation}\label{eq:fnB}
\abs{f(n)}=\abs{\log_0\lbrb{1+\frac{\kappa\lbrb{\zeta,n+z}-\kappa\lbrb{\zeta,n}}{\kappa\lbrb{\zeta,n}}}}\leq \frac{2|z|}{n}=\so{1}
\end{equation}
	and,
	 as $n\to\infty, z=\so{n}$, we get that
	\begin{equation}\label{eq:fn0}
	\begin{split}
	\limi{n}\abs{f(n)}=\limi{n}\abs{\log_0\lbrb{1+\frac{\kappa\lbrb{\zeta,n+z}-\kappa\lbrb{\zeta,n}}{\kappa\lbrb{\zeta,n}}}}=\limi{n}\abs{\frac{\kappa\lbrb{\zeta,n+z}-\kappa\lbrb{\zeta,n}}{\kappa\lbrb{\zeta,n}}}=0.
	\end{split}
	\end{equation} 
	Henceforth we obtain from \eqref{eq:SnIn} that
	\begin{equation}\label{eq:SnIn_1}
	\begin{split}
	\limi{n}\lbrb{S_n-\int_{1}^{n}f(x)dx}&=\frac{1}{2}f(1)+\limi{n}\lbrb{\Ekap(\zeta,z,n)+\frac12f(n)}\\
	&=\frac{1}{2}f(1)+\Ekap(\zeta,z).
	\end{split}	
	\end{equation}
	Next, we estimate $I_n=\int_{1}^{n}f(x)dx$, as $n\to\infty$. From \eqref{eq:limSup} and since $\lbrb{\zeta,z,x}\in\Cb_{\lbbrb{0,\infty}}\times\CbOI\times\lbbrb{1,\infty}$ both $\Re\lbrb{\kappa\lbrb{\zeta,x}}>0,\Re\lbrb{\kappa\lbrb{\zeta,z+x}}>0$.  Therefore, $\arg\lbrb{\frac{\kappa\lbrb{\zeta,x+z}}{\kappa\lbrb{\zeta,x}}}\in \lbrb{-\pi,\pi}$ and thus 
	\begin{equation}\label{eq:log0}
	f(x)=\log_0\lbrb{\frac{\kappa\lbrb{\zeta,x+z}}{\kappa\lbrb{\zeta,x}}}=\log_0\lbrb{\kappa\lbrb{\zeta,x+z}}-\log_0\lbrb{\kappa\lbrb{\zeta,x}}.
	\end{equation}
	Thus, fixing the parallelogram in $\Cb$ with vertices $1,1+z,z+n,n$, we see from the Cauchy  integral theorem applied to the function $\log_0\lbrb{\kappa\lbrb{\zeta,\cdot}}$ which is holomorphic in an open neighbourhood of the parallelogram that 
	\begin{equation}\label{eq:InCauchy}
	\begin{split}
	I_n&=\int_{1}^{n}\lbrb{\log_0\lbrb{\kappa\lbrb{\zeta,x+z}}-\log_0\lbrb{\kappa\lbrb{\zeta,x}}}dx\\
	&=\int_{n\rightarrow n+z} \log_0\lbrb{\kappa\lbrb{\zeta,\chi}}d\chi-\int_{1\rightarrow 1+z} \log_0\lbrb{\kappa\lbrb{\zeta,\chi}}d\chi.
	\end{split}
		\end{equation}
		From the latter and \eqref{eq:SnIn_1} we get that
		\begin{equation}\label{eq:Sn_1}
		\begin{split}
	S_n&=\int_{n\rightarrow n+z} \log_0\lbrb{\kappa\lbrb{\zeta,\chi}}d\chi-\int_{1\rightarrow 1+z} \log_0\lbrb{\kappa\lbrb{\zeta,\chi}}d\chi\\
	&+\frac{1}{2}\log_0\lbrb{\frac{\kappa\lbrb{\zeta,1+z}}{\kappa\lbrb{\zeta,1}}}+\Ekap(\zeta,z)+\so{1}
		\end{split}	
		\end{equation}
		and hence \eqref{eq:Bern-GammaExtLim_2} can be re-expressed as follows
			\begin{equation}\label{eq:Bern-GammaExtLim_3}
		\begin{split}
		\Wkap(\zeta,z)&=\frac{1}{\kappa\lbrb{\zeta,z}}\limi{n}e^{z\log_0\kappa\lbrb{\zeta,n}-S_n}\\
		&=\frac{1}{\kappa\lbrb{\zeta,z}}\limi{n}\Big(e^{z\log_0\kappa\lbrb{\zeta,n}-\int_{n\rightarrow n+z} \log_0\lbrb{\kappa\lbrb{\zeta,\chi}}d\chi+\int_{1\rightarrow 1+z} \log_0\lbrb{\kappa\lbrb{\zeta,\chi}}d\chi}\\
		&\times e^{-\frac{1}{2}\log_0\lbrb{\frac{\kappa\lbrb{\zeta,1+z}}{\kappa\lbrb{\zeta,1}}}-\Ekap\lbrb{\zeta,z}}\Big)%\\
		%&=\\
		%&\times \limi{n}e^{z\log_0\kappa\lbrb{\zeta,n}-z\int_{0}^{1}\log_0\lbrb{\kappa\lbrb{\zeta,n+vz}}dv}.
		\end{split}
		\end{equation}
		However, as \eqref{eq:log0} holds true we get using \eqref{eq:fnB} when $|z|=\so{n}$ that
		\begin{equation*}
		\begin{split}
		&\abs{z\log_0\kappa\lbrb{\zeta,n}-\int_{n\rightarrow n+z} \log_0\lbrb{\kappa\lbrb{\zeta,\chi}}d\chi}=|z|\abs{\int_{0}^1 \log_0\lbrb{\frac{\kappa\lbrb{\zeta,n+vz}}{\kappa(\zeta,n)}}dv}\leq 2\frac{|z|^2}{n}
		\end{split}
		\end{equation*}
		and hence
		\begin{equation}\label{eq:bound1}
		\begin{split}
		&\limi{n}\abs{e^{z\log_0\kappa\lbrb{\zeta,n}-\int_{n\rightarrow n+z} \log_0\lbrb{\kappa\lbrb{\zeta,\chi}}d\chi}}=1.
		\end{split}
		\end{equation}
		Finally, combining together the results \eqref{eq:EknEk}, \eqref{eq:fnB}, \eqref{eq:Bern-GammaExtLim_3} and \eqref{eq:bound1}, we deduce, as required for \eqref{eq:Stirling}, that
		\[\Wkap(\zeta,z)=\frac{1}{\kappa\lbrb{\zeta,z}}e^{\int_{1\rightarrow 1+z} \log_0\lbrb{\kappa\lbrb{\zeta,\chi}}d\chi-\frac{1}{2}\log_0\lbrb{\frac{\kappa\lbrb{\zeta,1+z}}{\kappa\lbrb{\zeta,1}}}-\Ekap(\zeta,z)}.
		\]
		This completes the proof of the theorem.
\end{proof}

  \section{Proofs for Exponential Functionals of \LL Processes}    \label{expproofs}
  Our  starting point in determining the Mellin transform of the exponential functional up to a finite time  is the following equation, proven in  \cite[Proposition 6.1.2]{ps19}. For all $\phi\in\cal{B}$ and  $z\in\bb{C}_{(-1,\infty)}$,
  \b{equation}
  \label{bgreln}
    \bb{E}\big[    I_{\phi}^z(\infty)          \big] =   \f{\Gamma(z+1)      }{   W_{ \phi}(z+1)      } .
  \e{equation}
  
This allows us to understand $I_\phi(t)$, $t<\infty$, in terms of the exponential functionals  $I_{\varphi_q}(\infty)$ of the subordinators with \LLK exponent $\varphi_q(w):=\phi(w)+q$, for $q\geq 0$,  % killed with $t=\infty$, 
through the following  argument:
 %
 %
% \color{red}{$\bullet$} \color{black}   Since $e_q$ is an independent exponential random variable, 
  \begin{align}\label{eq:LT} 
       \nonumber 
            \f{\Gamma(z+1)      }{   W_{ \varphi_q}(z+1)      }
            =
             \bb{E}\big[    I_{\varphi_q}^z(\infty)          \big]   \vphantom{\Bigg(}
   =
     \bb{E}\big[     I_{\phi}^z(e_q)        \big]  
    &=   
        \int_0^\infty           \bb{E}\big[  \hspace{0.5pt}  I_{\phi}^z(t)   \hspace{0.5pt}     \big]  \  \pp( e_q \in dt )     
%  \\  %
%  \nonumber
   %   \vphantom{\Bigg(}
 %  
%    \\  
   % \nonumber
        \\ 
        & =  
             \int_0^\infty          \bb{E}\big[   \hspace{0.5pt} I_{\phi}^z(t)       \hspace{0.5pt}     \big]     q e^{-qt}   dt     
%  \\    
%   \label{laplaceexptime}
 %    \vphantom{\Bigg(}
       %&
       =   q   \   \cal{L}\big|_q \l\{   \hspace{0.5pt}  \bb{E}\big[   \hspace{0.5pt}   I_{\phi}^z(\cdot)    \hspace{0.5pt}    \big]   \hspace{0.5pt}    \r\}.     
 \end{align}
 %We are going to invert the Laplace transform in order to find an explicit formula for $\bb{E}[   \hspace{0.5pt}   I_{\phi}^z(t)    \hspace{0.5pt}    ] $, 
 The formula in Theorem \mbox{\ref{convolutionformula}} comes from inverting this Laplace transform. First, we will express the Bernstein-gamma function $W_{ \varphi_q}(z+1)$ as a   product of simple Laplace transforms in the variable $q$, which  facilitates our proof of Theorem \mbox{\ref{convolutionformula}}.

 \begin{lemma} \label{productformula}   For each  subordinator with Laplace exponent $\varphi_q$, for all $z\in\bb{C}_{(-1,\infty)}$,
 \begin{equation}
 \label{prodformula}  
      \cal{L}|_q   \{  \bb{E}[  I_{\phi}^z(\cdot) ]   \}   =      \f{\Gamma(z+1)}{q \hspace{0.5pt} W_{ \varphi_q}(z+1)} 	         =   \f{\Gamma(z+1)}{ q\hspace{0.5pt} \varphi^z_q(1)} 	   \prod_{k=1}^\infty  \l[ \f{ \varphi_q(z+k)}{ \varphi_q(k)}  \l(    1 +   \f{    \varphi_q(k+1)  -   \varphi_q(k)  }{  \varphi_q(k)  }   \r)^{\hspace{-3pt}-z}    \r]   .
 \end{equation} 
\end{lemma}
We prove this lemma in Section \mbox{\ref{lemmasproofs}}.
Next, let us  state some important facts about convolutions and Laplace transforms.  We define these as $(f\ast g) (t):= \int_0^t f(s) g(t-s) ds$, and $ \cal{L}\big|_q \{f\} := \int_0^\infty   e^{-qs} f(s) ds$, respectively. Then we have
\b{equation} 
\label{laplaceconv}
\cal{L}\{ f\ast g \} = \cal{L}\{f\} \times \cal{L}\{ g \}, 
\e{equation}
%  where  we define the convolution operator $(f\ast g) (t):= \int_0^t f(s) g(t-s) ds$, and   
  \b{equation}
  \label{laplacesum}
  \cal{L}\{ \alpha f(\cdot)+ \beta g(\cdot) \} = \alpha\cal{L}\{f\} +\beta \cal{L}\{ g \}, \qquad \alpha,\beta\in\bb{C}.
  \e{equation}
Moreover, for $f(t):=t^{z-1}, g(t):=m t^{m-1}$, %their convolution satisfies
\b{equation} \label{m-1}
    (f\ast g)(t)
    % =       m B(z,m)   t^{z+m-1}  
    =     m\f{\Gamma(z) \Gamma(m)}{  \Gamma(z+m)  }     t^{z+m-1}   = \f{m!}{(z)^{(m)}}t^{z+m-1}, 
\end{equation}
where we recall that $(z)^{(m)}=z\lbrb{z+1}\cdots (z+m-1)$ is the rising factorial function.
 Similarly, with $f(t):=t^{z-1}, g(t):= t^{m}$,
\b{equation} \label{m}
    (f\ast g)(t) 
    %=        B(z,m+1)   t^{z+m}  
    =    \f{m!}{(z)^{(m)}} \f{t^{z+m} }{z+m} =\f{m!}{(z)^{(m)}}   \int_0^t s^{z+m-1}  ds.
\end{equation}
Also, we can expand brackets for convolutions of exponentials, in the sense that
\begin{equation} \label{expout}
( f\times e^{c\cdot} ) \ast (g\times e^{c\cdot} )(t) =   e^{ct}(f\ast g)(t).
\end{equation}

\n Now we are ready to prove Theorem \mbox{\ref{convolutionformula}}.

\begin{proof}[Proof of Theorem \mbox{\ref{convolutionformula}}]   
First, observe that since $\varphi'_q$ is non-increasing, applying $\l(\ref{eq:phi'_phi}\r)$, we have
\[
 \f{    \varphi_q(k+1)  -   \varphi_q(k)  }{  \varphi_q(k)  } =   \f{   \int_k^{k+1} \varphi_q'(x)dx  }{  \varphi_q(k)  }      \leq   \f{    \varphi'_q(k)    }{   \varphi_q(k)  }   \overset{(\ref{eq:phi'_phi})}{\leq} \frac1k<1.
 \]
 Then using the generalised binomial series expansion for $|w|<1$   
\[
(1+w)^{-z}
=      \sum_{m=0}^\infty      \f{ (z)^{(m)}    \ (-w)^m }{ m!   }   ,
\]
 we can rewrite   the product in $\l(\ref{prodformula}\r)$  as 
\b{equation}
\label{sumprod}
\l(\ref{prodformula}\r) =  \f{\Gamma(z+1)}{q  \varphi^z_q(1)} \prod_{k=1}^\infty  \l[ \f{\varphi_q(z+k)}{\varphi_q(k)}  \sum_{m=0}^\infty \l(    \f{ (z)^{(m)}     }{ m!   }    \l(     \f{ -[  \varphi_q(k+1)  -  \varphi_q(k) ] }{ \varphi_q(k)  }   \r)^{m}    \r)     \r].
\e{equation}
 Now, we can write each individual term in $\l(\ref{sumprod}\r)$ as a well-known Laplace transform, which will then allow us to invert the Laplace transform to yield the expression in $\l(\ref{convformula}\r)$ as an infinite convolution. 
 One  easily verifies that the terms in $\l(\ref{sumprod}\r)$ are the following Laplace transforms:
\begin{align}
\label{conv1}   
\f{1}{q} &=   \cal{L}\Big|_q \l\{     1  \r\}  
\\
\label{conv2}     \vphantom{\Bigg(}
 \f{1}{\varphi^z_q(1)} = \f{1}{(\phi(1)+q)^z} &=  \cal{L} \Big|_q  \l\{   \f{  t^{z-1}   }{ \Gamma(z) e^{\phi(1)t}  }      \r\}       
\\
\label{conv3}   \vphantom{\Bigg(}
\f{\varphi_q(z+k)}{\varphi_q(k)} = 1 + \f{\phi(z+k)-\phi(k)}{\phi(k)+q}  &=    \cal{L}\Big|_q  \l\{     \delta_0(dt) +  \f{ \l[ \phi(z+k) - \phi(k)  \r]  }{  e^{  \phi(k)t}  }   \r\} 
\\
\label{conv4}   \vphantom{\Bigg(}
 \hspace{-10pt} \l[     \f{   \varphi_q(k+1)  -  \varphi_q(k)  }{ \varphi_q(k)  }   \r]^{m}      \hspace{-7pt}   =     \hspace{-3pt}  \l[     \f{   \phi(k+1)  -  \phi(k)  }{ \phi(k) +q  }   \r]^{m}     \hspace{-7pt}  &=     \cal{L}\Big|_q      \l\{  [\phi(k+1)  -  \phi(k) ]^m   \f{  t^{m-1}     }{\Gamma(m) e^{\phi(k) t }  }   \r\} \hspace{-2pt} ,     m\geq1,
% \\
%\label{conv4}
 %\l(     \f{   \phi_q(k+1)  -  \phi_q(k)  }{ \phi_q(k)  }   \r)^{0}      =   \l(     \f{   \phi(k+1)  -  \phi(k)  }{ \phi(k) +q  }   \r)^{m} &=     \cal{L}\Big|_q      \l\{     s^{m-1}   e^{- \phi(k) s }      \r\} , \qquad   m>1,
\end{align}
and when $m=0$, the term in $\l(\ref{conv4}\r)$ is simply the Laplace transform of the unit point mass  $\delta_0(dt)$.
%
% consider the summands
%
%
%
%
%\[
 %\f{1}{\Gamma(m+1)\Gamma(z_1-m)}  \l(     \f{   \phi(k+1)  -  \phi(k)  }{ \phi(k) +q  }   \r)^{m}.
%\]
%When $m=0$, it is clear that this is the Laplace transform of a point mass:
%\[
%\f{1}{\Gamma(z+1)} = \cal{L}\Big|_q  \l(  \f{\delta_0(ds)}{ \Gamma(z+1)}   \r).
%\]
%When $m\geq1$, this is the following Laplace transform:
%\[
% \f{1}{\Gamma(m+1)\Gamma(z_1-m)}  \l(     \f{   \phi(k+1)  -  \phi(k)  }{ \phi(k) +q  }   \r)^{m}   = \cal{L}\Big|_q  \l(  \f{  (\phi(k+1) - \phi(k))^m  }{ \Gamma(m+1) \Gamma(z+1-m)}   s^{m-1}   e^{- \phi(k) s } \Gamma(m) ds   \r)  
%\]
%\b{equation}   \label{conv1}
%=   \cal{L}\Big|_q  \l(  \f{  (\phi(k+1) - \phi(k))^m  }{ m \Gamma(z+1-m)}   s^{m-1}   e^{- \phi(k) s }  ds \r) . 
%\e{equation}
%
%Next, we consider $ \f{\phi_q(z+k)}{\phi_q(k)}   $, which is the following Laplace transform:
%\[
%\f{\phi_q(z+k)}{\phi_q(k)} = 1 + \f{\phi(z+k)-\phi(k)}{\phi(k)+q} 
%\]
%\b{equation} \label{conv2}
%=  \cal{L}\Big|_q  \l(     \delta_0(ds) + \l[ \phi(k+z) - \phi(k)  \r] e^{ - \phi(k)s} ds 
%\r).
%\e{equation}
%
%Finally, observe that
%\b{equation} \label{conv3}
%    \f{\Gamma(z+1)^2}{\phi_q(1)^z} = \f{\Gamma(z+1)^2}{(\phi(1)+q)^z} =  \cal{L}\Big|_q  \l(     \Gamma(z+1)^2 s^{z-1} e^{-\phi(1)s} ds       \r) .
%\e{equation}
%
%
%
Substituting  $\l(\ref{conv1}\r),\l(\ref{conv2}\r)$,  $\l(\ref{conv3}\r),$  and $\l(\ref{conv4}\r)$ into the equation  $\l(\ref{sumprod}\r),$    we arrive at the following formula:
\b{flalign}
\label{bss}
  &  \cal{L}\big|_q \l\{   \hspace{0.5pt}  \bb{E}\big[   \hspace{0.5pt}   I_{\phi}^z(t)    \hspace{0.5pt}    \big]   \hspace{0.5pt}    \r\}  
    = 
    \Gamma(z+1) \cal{L}\Big|_q \Bigg\{   1\ast  \f{  t^{z-1}   }{\Gamma(z) e^{\phi(1)t}}  \    \ast  \overset{\infty}{\underset{k=1}{\conv}} \Bigg[   \l( \delta_0(dt)+\f{[\phi(k+z)-\phi(k)]}{e^{\phi(k)t} } \r)       &&
\end{flalign}
\[
 \hspace{170pt}  \ast \l(  \delta_0(dt)+ \sum_{m=1}^\infty      \f{  (z)^{(m)}  (-[\phi(k+1) - \phi(k)])^m   t^{m-1}    }{ m! \  \Gamma(m)   \  e^{\phi(k) t } }       \r)      \Bigg]     \Bigg\}.    
\]

Now, by $\l(\ref{laplacesum}\r)$ and $\l(\ref{expout}\r)$,   noting  $\Gamma(m)=(m-1)!$ and $\delta_0(\cdot) \ast f \equiv f$ for any function $f$, we have:
\begin{align*}
\l(\ref{bss}\r) 
&\overset{\l(\ref{laplacesum}\r)}=          \cal{L}\Big|_q \Bigg\{   1\ast  \f{ z t^{z-1} }{e^{\phi(1)t}}     \    \ast   
\overset{\infty}{\underset{k=1}{\conv}} \Bigg[   \delta_0(dt)+   \sum_{m=1}^\infty      \f{  (z)^{(m)}  (-[\phi(k+1) - \phi(k)])^m  }{ (m!)^2  }   \f{ m  t^{m-1}    }{    e^{ \phi(k) t } }     
\\
&\hspace{27pt}   + \f{[\phi(k+z)-\phi(k)]}{e^{\phi(k)t}   }  +  \f{[\phi(k+z)-\phi(k)]}{e^{\phi(k)t} }  \ast    \sum_{m=1}^\infty      \f{  (z)^{(m)}  (-[\phi(k+1) - \phi(k)])^m  }{ (m!)^2  }   \f{m t^{m-1} }{    e^{\phi(k) t }   }           \Bigg]     \Bigg\}
\\
&\overset{\l(\ref{expout}\r)}=        \cal{L}\Big|_q \Bigg\{   1\ast  \f{z t^{z-1} }{e^{\phi(1)t}}     \    \ast   
\overset{\infty}{\underset{k=1}{\conv}} \Bigg[   \delta_0(dt)+   \sum_{m=1}^\infty      \f{  (z)^{(m)}  (-[\phi(k+1) - \phi(k)])^m  }{ (m!)^2  }   \f{ m  t^{m-1}    }{    e^{ \phi(k) t } }    
\\
&\hspace{42pt}   + \f{[\phi(k+z)-\phi(k)]}{e^{\phi(k)t}   }    +   [\phi(k+z)-\phi(k)]     \sum_{m=1}^\infty      \f{ (z)^{(m)}  (-[\phi(k+1) - \phi(k)])^m  }{ (m!)^2  }   \f{  t^{m} }{    e^{\phi(k) t }   }           \Bigg]     \Bigg\}
\\
&\overset{\phantom{\l(\ref{expout}\r)}}=       \cal{L}\Big|_q \Bigg\{   1\ast  \f{z t^{z-1} }{e^{\phi(1)t}}     \    \ast   
\overset{\infty}{\underset{k=1}{\conv}} \Bigg[   \delta_0(dt)  +  \sum_{m=1}^\infty      \f{  (z)^{(m)}  (-[\phi(k+1) - \phi(k)])^m  }{ (m!)^2  }   \f{ m  t^{m-1}    }{    e^{ \phi(k) t } }     
\\
&\hspace{126pt}+   [\phi(k+z)-\phi(k)]     \sum_{m=0}^\infty      \f{ (z)^{(m)} [-\lbrb{\phi(k+1) - \phi(k)}]^m  }{ (m!)^2  }   \f{  t^{m} }{    e^{\phi(k) t }   }           \Bigg]     \Bigg\},
\end{align*}
and then the desired result $\l(\ref{convformula}\r)$ follows immediately by a simple Laplace inversion.
\end{proof}

%	 \subsection{Proof of Theorem \mbox{\ref{mainthm}}}   

	 \n  Lemma \mbox{\ref{nthtermlemma}}, which we prove in Section \mbox{\ref{lemmasproofs}}, builds upon Theorem \mbox{\ref{convolutionformula}} by evaluating the first $n$ terms in the infinite convolution. In the proof of Theorem  \mbox{\ref{mainthm}}, we will take limits as $n\to\infty$.

\begin{lemma}    
\label{nthtermlemma}  
For all $\phi\in\cal{B}$, $z\in \bb{C}_{(0,\infty)}$, and  $n\in\bb{N}$,  the following truncated convolution satisfies
 \[
   \f{z t^{z-1}}{ e^{\phi(1)t}     } 
\ast \overset{n}{\underset{k=1}{\conv}} \Bigg[   \delta_0(dt)    +  \sum_{m=1}^\infty   \f{  (z)^{(m)} }{  (m!)^2  }  \l( - [\phi(k+1)-\phi(k)]\r)^m   \f{mt^{m-1} }{e^{\phi(k)t}}
\]
\[ 
\hspace{4.95cm} +   [\phi(z+k)-\phi(k)]      \sum_{m=0}^\infty      \f{  (z)^{(m)} }{  (m!)^2  }  \l( - [\phi(k+1)-\phi(k)]\r)^{m}   \f{t^m  }{e^{\phi(k)t}}   \Bigg]
\]
\b{equation}
\label{nthtermeqn}
=
\f{ z t^{z-1}  }{ e^{\phi(n+1)t}    }
+   z \sum_{k=1}^n       \f{   \underset{1\leq j\leq n }{\prod}      [ \phi(z+j) - \phi(k)]    }{   \underset{1\leq j \leq n; j\neq k}{\prod}     [ \phi(j) - \phi(k)]  \hspace{15pt}   }       \f{        \gamma(z,[\phi(n+1)-\phi(k)]t)   }{      [\phi(n+1)-\phi(k)]^z    \    e^{\phi(k)t}   }  ,
\e{equation}
where $\gamma(z,u):= \int_0^u  e^{-x} x^{z-1}dx$ denotes the lower incomplete gamma function.
\end{lemma}

\n	 The proof of  Theorem \mbox{\ref{mainthm}} also requires the following lemmas, which are also proven later in Section \mbox{\ref{lemmasproofs}}.

	 \begin{lemma}
	 \label{phi'lemma}
	 For each $\phi\in\cal{B}$ and $c>0$, there exists a constant $C_{\phi,c}>0$ such that for all   $n\geq1$, 
	 \[
	    \f{\phi'(n)}{\phi'(n+c)} \leq    C_{\phi,c}.
	 \]
	 \end{lemma}

	 \b{lemma}   For  $\phi\in\cal{B}$ satisfying the condition in Definition \mbox{\ref{regularitycondition}} and for each $z\in\cc$,  
\label{absconvlemma}
\b{equation}
\label{absconvlemmaeqn}
  \sum_{k=1}^\infty     \l|      \f{   \underset{1\leq i\leq k }{\prod}      [ \phi(z+i) - \phi(k)]    }{   \underset{1\leq j \leq k-1}{\prod}     [ \phi(j) - \phi(k)]  \hspace{15pt}   }  \f{     e^{-\phi(k)t}   }{ W_{ \phi_{(k)}}(z+1)  }  \r| < \infty.
\e{equation}
\e{lemma}

	 \begin{lemma} 
	 \label{productbound} For all $\phi\in\cal{B}$  and  $z\in\cc$,   there is a constant $c_{\phi,z}>0$ such that for all $k\geq1$, 
\[
\underset{1\leq j\leq k-1 }{\prod}     \f{ | \phi(z+j) - \phi(k)|  }{   \phi(k) - \phi(j) }  \leq  c_{\phi,z}.
\]

	 \end{lemma}
	 The structure of the proof of Theorem \mbox{\ref{mainthm}} is as follows:  First, we will find the termwise limit, as $n\to\infty$, of each summand in $\l(\ref{nthtermeqn}\r)$. This gives a formula for the infinite convolution in $\l(\ref{convformula}\r)$, without the convolution with $1$, which means this limit is an expression for $\f{d}{dt}\bb{E}[I_\phi^z(t)]$.   To show that the termwise limits correspond to the limit of the whole expression in $\l(\ref{nthtermeqn}\r)$, we employ a dominated convergence argument. We   finish the proof of Theorem \mbox{\ref{mainthm}} by integrating each term in our expression for $\f{d}{dt}\bb{E}[I_\phi^z(t)]$, which requires a second dominated convergence argument. 
	 \begin{proof}[Proof of Theorem \mbox{\ref{mainthm}}] 

  Denoting $\phi_{(k)}(w):= \phi(w+k)-\phi(k)$,   we first verify that $\phi_{(k)}$ is itself a Bernstein function, with unchanged drift and rescaled  L\'evy measure of the from $e^{-kx}  \Pi(dx)$. Indeed, this follows by noting that $\l(\ref{eq:Bern}\r)$ gives
\begin{align*}
\phi_{(k)}(w) &= \phi(0)+ (k+w)\textup{d} + \int_0^\infty    (1-e^{-(k+w)x}) \Pi(dx) - \phi(0)-  k\textup{d} - \int_0^\infty    (1-e^{-kx}) \Pi(dx) 
\\
&= w \textup{d} +  \int_0^\infty    (1-e^{-wx})  e^{-kx}  \Pi(dx).
\end{align*}

 \n  Now, rewriting in terms of $\phi_{(k)}$, the   convolution in $ \l(\ref{nthtermeqn}\r)$ satisfies
 %w\in\cc$, 
%we can express the formula from $\l(\ref{nthtermeqn}\r)$  as
\b{align*}
 \l(\ref{nthtermeqn}\r) &=  \f{ z t^{z-1}  }{ e^{\phi(n+1)t}    }
+   z    \sum_{k=1}^n      \f{     \prod_{j=1}^{k-1}   [ \phi(z+j) - \phi(k)]    }{         \prod_{j=1}^{k-1}     [ \phi(j) - \phi(k)]      }  \   \f{     \prod_{j=k}^{n}      [ \phi(z+j) - \phi(k)]    }{         \prod_{j=k+1}^{n}        [ \phi(j) - \phi(k)]       }  \    \f{        \gamma(z,[\phi(n+1)-\phi(k)]t)   }{      [\phi(n+1)-\phi(k)]^z   \     e^{\phi(k)t}   }  
\\
 &=     \f{ z t^{z-1}  }{ e^{\phi(n+1)t}    }
+   z \sum_{k=1}^n   \l(      \prod_{j=1}^{k-1}    \f{ [ \phi(z+j) - \phi(k)]    }{       [ \phi(j) - \phi(k)]      } \r)
  \f{       \prod_{j=0}^{n-k}      \phi_{(k)}(z+j)     }{      \prod_{j=1}^{n-k}      \phi_{(k)}(j)        }  \ 
   \f{        \gamma(z,[\phi(n+1)-\phi(k)]t)   }{     \phi^z_{(k)}(n+1-k)  \         e^{\phi(k)t}   }.  
 \end{align*}
For the termwise limit as $n\to\infty$,  first observe   that  under the conditions of Definition \ref{regularitycondition}, relation \eqref{blumbound} below implies that $\phi(\infty)=\infty$ and hence $\lim_{n\to\infty}     z t^{z-1}  / e^{\phi(n+1)t}    =0$, and  second that $\lim_{n\to\infty} \gamma(z,[\phi(n+1)-\phi(k)]t) = \Gamma(z)$.  Repeatedly using that $W_{\phi_{(k)}}(w+1)=   \phi_{(k)}(w) W_{\phi_{(k)}}(w)$ from $\l(\ref{eq:Bern-Gamma}\r)$, 
noting  $W_{\phi_{(k)}}(1)=1$,  and applying Lemma \mbox{\ref{Wlemma}}, it follows that as $n\to\infty$,
\b{align}
\label{Wtoprod}
   \f{      \prod_{j=0}^{n-k}      \phi_{(k)}(z+j)    }{      \prod_{j=1}^{n-k}      \phi_{(k)}(j)    }   \     \f{1}{\phi^z_{(k)}(n+1-k)} 
%
 %&\overset{\phantom{\ref{Wlemma}}}= \lim_{n\to\infty}
%  \f{  W_{\phi_{(k)}}(z+n+1-k)  W_{\phi_{(k)}}(1)  }{   W_{\phi_{(k)}}(z)   W_{\phi_{(k)}}(n+1-k)   }       \f{1}{ \phi_{(k)}(n+1-k)^z} 
% \\
  &\overset{\phantom{\ref{Wlemma}}}=   \f{ \phi_{(k)}(z)  W_{\phi_{(k)}}(z+n+1-k)     }{   W_{\phi_{(k)}}(z+1)   W_{\phi_{(k)}}(n+1-k)   }    \f{1}{\phi^z_{(k)}(n+1-k)}   
\\
\nonumber
  &\overset{\ref{Wlemma}}\sim    \f{  \phi_{(k)}(z)     }{   W_{\phi_{(k)}}(z+1)       }    =     \f{ [ \phi(z+k)-\phi(k)]     }{   W_{\phi_{(k)}}(z+1)       }   .
\end{align}
This gives us the  limit, as $n\to\infty$, of each summand in $\l(\ref{nthtermeqn}\r)$. Using the dominated convergence theorem, we  will now show that we can exchange the order of limits and summation,  %which is provided by Lemma \mbox{\ref{dom1}}.  
 which yields the following relation
\b{equation}
\label{derivativeformula}
\f{d}{dt} \bb{E}[I_\phi^z(t)] = \lim_{n\to\infty} \l(\ref{nthtermeqn}\r) =  \Gamma(z+1)  \sum_{k=1}^\infty      \f{      \prod_{j=1}^{k}      [ \phi(z+j) - \phi(k)]    }{     \prod_{j=1}^{k-1}      [ \phi(j) - \phi(k)]     }  \ 
   \f{      e^{-\phi(k)t}     }{   W_{\phi_{(k)}}(z+1)      }  .
\e{equation}

\n  In order to   apply the dominated convergence theorem, %which will allow us to exchange  
 we must show that the sum in $\l(\ref{nthtermeqn}\r)$  is dominated by an absolutely convergent sum. 
From $\l(\ref{Wtoprod}\r)$, we can rewrite the sum in $\l(\ref{nthtermeqn}\r)$ as 
\b{equation}
%\nonumber
%\l(\ref{nthtermeqn}\r)   
%&\phantom{=}  
 \sum_{k=1}^\infty   \bbm{1}_{\{  k\leq n   \}}       \f{ \overset{k}{  \underset{j=1  }{\prod}}      [ \phi(z+j) - \phi(k)]    }{   \overset{k-1}{   \underset{j=1 }{\prod}}    [ \phi(j) - \phi(k)]     }      \f{  W_{\phi_{(k)}}(z+n+1-k)  \   \gamma(z,[\phi(n+1)-\phi(k)]t)    }{    W_{\phi_{(k)}}(z+1)    W_{\phi_{(k)}}(n+1-k) \phi^z_{(k)}(n+1-k)       e^{\phi(k)t}   }    
%\\
\label{F_k}
=:    \sum_{k=1}^\infty        F_k(n).
\end{equation}
Then if we can show that there exist $n_0>0$ and $C>0$  such that for all $n\geq n_0$,   and for all $k\geq 1$, 
\b{equation}
\label{desiredbound}
      |F_k(n)| \leq C \times \l|  \f{         \prod_{j=1}^{k}       [ \phi(z+j) - \phi(k)]    }{     \prod_{j=1}^{k-1}       [ \phi(j) - \phi(k)]     } \   \f{     e^{-\phi(k)t}   }{ W_{\phi_{(k)}}(z+1)  }  \r|,
\e{equation}
and moreover if 
\b{equation}
\label{absconv}
  \sum_{k=1}^\infty     \l|      \f{      \prod_{j=1}^{k}        [ \phi(z+j) - \phi(k)]    }{      \prod_{j=1}^{k-1}       [ \phi(j) - \phi(k)]    } \   \f{     e^{-\phi(k)t}   }{ W_{\phi_{(k)}}(z+1)  }  \r| < \infty,  
\e{equation}
then the dominated convergence theorem applies, so
\(
\lim_{n\to\infty}   \sum_{k=1}^\infty        F_k(n)    =      \sum_{k=1}^\infty      \lim_{n\to\infty}    F_k(n) %= \sum_{k=1}^\infty        \f{   \underset{1\leq i\leq k }{\prod}      [ \phi(z+i) - \phi(k)]    }{   \underset{1\leq j \leq k-1}{\prod}     [ \phi(j) - \phi(k)]  \hspace{15pt}   }  \f{     e^{-\phi(k)t}   }{ W_{\phi_{(k)}}(z+1)  } ,
\), 
as required for $\l(\ref{derivativeformula}\r)$.  
The absolute convergence of the sum in $\l(\ref{absconv}\r)$ is proven in Lemma \mbox{\ref{absconvlemma}}, so now let us prove $\l(\ref{desiredbound}\r)$.
 Firstly, one can easily verify that for all $n,k \in \bb{N}$, with $z=a+ib$,   
\b{equation}
\label{lowerincompletebound}
|\gamma(z,[\phi(n+1)-\phi(k)]t) | \leq      \int_0^{[\phi(n+1)-\phi(k)]t} |u^{z-1} e^{-u}| du   \leq     \int_0^{\infty} u^{a-1} e^{-u}du  =    \Gamma(a).
\e{equation}
The remaining term in the expression for $F_k(n)$ from $\l(\ref{F_k}\r)$ which depends on $n$ is
\begin{equation}
\label{remainingWterm}
  \f{  W_{\phi_{(k)}}(z+n+1-k)      }{  W_{ \phi_{(k)}}(n+1-k)  \phi^z _{(k)}(n+1-k)      }.
\end{equation}
Applying  Theorem \mbox{\ref{thm:Stirling}},   
we can rewrite $\l(\ref{remainingWterm}\r)$ as
\b{equation}
\label{phi(k)}
        \f{       \phi_{(k)}(n+1-k)   \phi^\f{1}{2}_{(k)}(n+2-k)   \     e^{ L_{ \phi_{(k)}}(z+n+1-k)   - E_{ \phi_{(k)}}(z+n+1-k)       - L_{ \phi_{(k)}}(n+1-k)   + E_{ \phi_{(k)}}(n+1-k)       }       }{      \phi_{(k)}(z+n+1-k)   \phi^\f{1}{2}_{(k)}(z+n+2-k)   \phi^z_{(k)}(n+1-k)       }       .  
\e{equation}
We will bound the absolute value of the terms in $\l(\ref{phi(k)}\r)$ separately.  
 Firstly, using the result \cite[Prop. 3.1.9]{ps16}  that $|\phi(a+ib)| \geq \phi(a)$ and the fact that $\phi$ is non-decreasing on $\bb{R}_+$, it follows that  
  \b{equation}
  \label{phieasybound}
\hspace{-7pt} \f{       \phi_{(k)}(n+1-k)   \  \phi^\f{1}{2}  _{(k)}(n+2-k)       }{      |\phi_{(k)}(z+n+1-k)| \   |\phi^\f{1}{2}_{(k)}(z+n+2-k)|        }  
 \leq     \f{       \phi_{(k)}(n+1-k)   \  \phi^\f{1}{2}_{(k)}(n+2-k)         }{      \phi_{(k)}(a+n+1-k)  \phi^\f{1}{2}_{(k)}(a+n+2-k)        }     \leq 1. 
 \e{equation}
 %so that 
% \[
%| \l(\ref{phi(k)}\r) | \leq     \l|     \f{       e^{ L_{ \phi_{(k)}}(z+n+1-k)   - E_{ \phi_{(k)}}(z+n+1-k)       - L_{ \phi_{(k)}}(n+1-k)   + E_{ \phi_{(k)}}(n+1-k)       }       }{       \phi_{(k)}(n+1-k)^z       }   \r|    .  
% \]
 %
 %
 %

\n To bound the $E_{ \phi_{(k)}}$ terms in $\l(\ref{phi(k)}\r)$, observe that  by $\l(\ref{eq:Ek}\r)$, 
\(
\sup_{n\geq1}   \sup_{1\leq k \leq n}  |  E_{ \phi_{(k)}}(z+n+1-k) |     \leq 2,
\)
and 
\(
\sup_{n\geq1}   \sup_{1\leq k \leq n}  |  E_{ \phi_{(k)}}(n+1-k) |    \leq 2,
\)
so it follows immediately that for all $k\leq n$,
\b{equation}
\label{Eeasybound}
\l|  e^{   - E_{ \phi_{(k)}}(z+n+1-k)         + E_{ \phi_{(k)}}(n+1-k)       }  \r|   \leq e^4.
\e{equation}

\n Next,   consider the  $L_{ \phi_{(k)}}$ terms in $\l(\ref{phi(k)}\r)$, which are defined in  $\l(\ref{eq:A}\r)$. Part of the integrals cancel, so
\[
L_{ \phi_{(k)}}(z+n+1-k)  - L_{ \phi_{(k)}}(n+1-k)  =   \int_{ 1 \mapsto 1+(z+n+1-k)     }      \hspace{-10pt}    \log_0 ( \phi_{(k)}(w))dw
-   \int_{ 1 \mapsto 1+(n+1-k)     }   \hspace{-10pt}   \log_0 ( \phi_{(k)}(w))dw
\]
\[
=   \int_{ 1+(n+1-k) \mapsto 1+(z+n+1-k)     } \log_0 ( \phi_{(k)}(w))dw =     \int_{ N-k  \mapsto N-k +z    } \log_0 ( \phi_{(k)}(w))dw,
\]
where we substitute $N:=n+2$.  Using the fact that $\log_0(w)= \ln(|w|) + i \arg( w  ) $, it follows that
\b{equation}
\label{Lints}
L_{ \phi_{(k)}}(z+n+1-k)  - L_{ \phi_{(k)}}(n+1-k)  =   \int_{ N-k \mapsto N-k+ z     }  \hspace{-22pt}   \ln (| \phi_{(k)}(w)|)dw   +      i \int_{ N-k \mapsto N-k+ z     }   \hspace{-22pt}   \arg ( \phi_{(k)}(w))dw.
\e{equation}

\n We  consider the integrals in $\l(\ref{Lints}\r)$ along the contours $\gamma_1,\gamma_2$, which are straight lines connecting $N-k$ to $ N-k+a$ and $ N-k+a$ to $ N-k + a + ib$, respectively. Observe that along $\gamma_1$, $\arg( \phi_{(k)}(x))=0$, and note $|\gamma_2|=b$. Then for all $1\leq k\leq n$,   the second integral in $\l(\ref{Lints}\r)$   satisfies   %for the second integral. Now, $|\gamma_2|=b$, so 
\b{equation} \label{l2}
\l| \int_{ N-k \mapsto N-k+ z     }   \hspace{-22pt}   \arg ( \phi_{(k)}(w))dw \r|
 = 
\l|  \int_{\gamma_2} \arg ( \phi_{(k)}(w))dw \r|   \leq  b \sup_{x\in \gamma_2}|  \arg( \phi_{(k)}(x)) | \leq  b \pi.
\e{equation} 

%\[
%\leq    b \sup_{x \in \gamma_2}   |\arg(x)|.
%\]
%since $ \phi_{(k)}$ preserves angular sectors (see \cite[Prop 3.6]{ssv12}).  
%Now, for all $k\in\{1,2,\dots,n\}$, $\sup_{x \in \gamma_2}   |\arg(x)|   \leq  \pi/2$,  and we conclude that uniformly among $k\in\{1,2,\dots,n\}$,
%\b{equation} \label{l2}
%  \l|  \int_{\gamma_2} \arg ( \phi_{(k)}(x))dx \r|      \leq  \f{b\pi}{2}  .
%\end{equation}

\n Next, we consider the first integral in $\l(\ref{Lints}\r)$ over $\gamma_1$. Recalling $N=n+2$, this can be written as
\b{align*}
  \int_{ \gamma_1   } \ln (| \phi_{(k)}(w)|)dw 
  &= \hspace{10pt}
   \int_{  N-k}^{   N-k+a   } \ln (  \phi_{(k)}(x)  )dx %= \int_{  N-k}^{   N-k+a   } \ln ( \phi_{(k)}(x))dx
 \hspace{-28pt} &&=
   \int_{  0}^{   a   } \ln ( \phi_{(k)}(x+N-k))dx 
\\
&=
 \int_{  0}^{   a   } \ln (\phi(x+N) - \phi(k))dx
 \hspace{-28pt} &&=
    \int_{  0}^{   a   } \ln (\phi(x+n+2) - \phi(k))dx .
\end{align*}
Comparing this with the  $1/ \phi^a _{(k)}(n+1-k)= e^{-a\ln(  \phi_{(k)}(n+1-k)  ) }$ term from $\l(\ref{phi(k)}\r)$, we get
%Now, this integral over $\gamma_1$ should cancel with   $1/ \phi_{(k)}(n+1-k)^a = e^{-a\ln(  \phi_{(k)}(n+1-k)  ) }$ in $\l(\ref{LE}\r)$.  Comparing the two, recalling $N=n+2$, we have 
\begin{equation}\label{intln}
\begin{split}
&\int_{  0}^{   a   } \ln (\phi(x+n+2) - \phi(k))dx         -a\ln(  \phi_{(k)}(n+1-k)  )\\
&=\int_{  0}^{   a   } \ln \l(    \f{\phi(x+n+2) - \phi(k)  }{   \phi_{(k)}(n+1-k)  }    \r)dx    =  \int_{  0}^{   a   } \ln \l(    \f{\phi(x+n+2) - \phi(k)  }{   \phi(n+1)  - \phi(k) }    \r)dx  .
\end{split}
\end{equation}
%We would like this to converge to 0, uniformly among $k$. It is easy to see that this is non-negative:
%\[
% \int_{  0}^{   a   } \ln \l(    \f{\phi(x+n+2) - \phi(k)  }{   \phi(n+1)  - \phi(k) }    \r)dx \geq  \int_{  0}^{   a   } \ln \l(    \f{\phi(n+1) - \phi(k)  }{   \phi(n+1)  - \phi(k) }    \r)dx =0.
%\]
%For an upper bound, 
%We would like to bound this, uniformly among $k$. 
Observe that  $ (\phi(x+n+2) - y  )/(   \phi(n+1)  - y ) $  and $\phi(y)$ are non-decreasing in $y$. So for all $k\leq n $,
\[
\l(\ref{intln}\r) \leq  \int_{  0}^{   a   } \ln \l(    \f{\phi(x+n+2) - \phi(n)  }{   \phi(n+1)  - \phi(n) }    \r)dx 
\leq   a  \ln \l(    \f{\phi(a+n+2) - \phi(n)  }{   \phi(n+1)  - \phi(n) }    \r).
\]

\n Now, since $\phi'$ is non-increasing on $\bb{R}_+$, applying Lemma \mbox{\ref{phi'lemma}},  it follows that 
\b{equation}
\label{Lgamma1bound}
\l(\ref{intln}\r) \leq    a  \ln \l(   \f{    \int_{n}^{a+n+2}  \phi'(y)dy  }{  \int_{n}^{n+1}  \phi'(y)dy }       \r) \leq  a  \ln \l(   \f{(a+2)  \phi'(n)}{\phi'(n+1)} \r)   \leq \textup{ constant} .
\e{equation}

\begin{comment}
\n So we have shown that uniformly among $k\in\{1,2,\dots,n\}$, 
\[
\lim_{n\to\infty}  \int_{  0}^{   a   } \ln (\phi(x+n+2) - \phi(k))dx         -a\ln(  \phi_{(k)}(n+1-k)  ) \leq    (a+2) \vee (a+2) \l(   e^A +    \f{ A    e^{A} \ov(A)  }{       \int_0^A    x  \Pi(dx)     }   \r),
\]
which gives in the original context of $\l(\ref{LE}\r)$,
\[
\lim_{n\to\infty}   \l|   \f{ e^{ L_{ \phi_{(k)}}(z+n+1-k)          - L_{ \phi_{(k)}}(n+1-k)               }  }{     \phi_{(k)}(n+1-k)^a       }      \r| \leq          e^{  (a+2) \vee (a+2) \l(   e^A +    \f{ A    e^{A} \ov(A)  }{       \int_0^A    x  \Pi(dx)     }   \r)} = \textup{ constant, }
\] 
which is good enough for a uniform bound.
\end{comment}

\n Finally, consider the first integral in $\l(\ref{Lints}\r)$ over the contour $\gamma_2$. To bound $\l(\ref{phi(k)}\r)$, we   compare it to $1/ \phi_{(k)}(n+1-k)^{ib}$. Then since a real number raised to an imaginary power has absolute value 1, we arrive at
\b{equation}
\label{Lgamma2bound}
\l| \f{ e^{   \int_{\gamma_2} \ln( |  \phi_{(k)}(x)|) dx }  }{    \phi^{ib}_{(k)}(n+1-k)   } \r| 
=
   \l|   \f{ e^{   \int_{N-k+a \mapsto N-k + a + ib} \ln( |  \phi_{(k)}(x)|) dx }  }{    \phi^{ib}_{(k)}(n+1-k)   } \r| 
= 
\l|  \f{ e^{  i  \int_{0}^{b} \ln( |  \phi_{(k)}(N-k+a + iv)|) dv }  }{    \phi^{ib}_{(k)}(n+1-k)   } \r| =1.
\e{equation}
%Now, the absolute value of this is simply 1, since we are raising real numbers to imaginary powers, so uniformly among $k$, we get
%\[
%\l| \f{ e^{   \int_{\gamma_2} \ln( |  \phi_{(k)}(x)|) dx }  }{    \phi_{(k)}(n+1-k)^{ib}   }  \r| =1.
%\]

\n Combining together the bounds $\l(\ref{phieasybound}\r)$, $\l(\ref{Eeasybound}\r)$, $\l(\ref{Lgamma1bound}\r)$, and $\l(\ref{Lgamma2bound}\r)$, we conclude that there exists a constant $C(z,\phi)$ such that uniformly among $n\in\bb{N}$ and $1\leq k \leq n$, 
\(
 | \l(\ref{phi(k)}\r)   | \leq C(z,\phi)
\), 
from which
$\l(\ref{desiredbound}\r)$ follows, and therefore we have shown that $\l(\ref{derivativeformula}\r)$ holds. % follows, so we have
%\b{equation*} 
%\f{d}{dt} \bb{E}[I_\phi^z(t)] = \lim_{n\to\infty} \l(\ref{nthtermeqn}\r) =  \Gamma(z+1)  \sum_{k=1}^\infty      \f{   \underset{1\leq i\leq k }{\prod}   [ \phi(z+i) - \phi(k)]    }{   \underset{1\leq j \leq k-1}{\prod}     [ \phi(j) - \phi(k)]    \hspace{15pt}   } 
 %  \f{      e^{-\phi(k)t}     }{   W_{\phi_{(k)}}(z+1)      }  .
%\e{equation*}  
%
%
%
%
Now, applying $\l(\ref{derivativeformula}\r)$,
 %writing $f_\phi(s):=\bb{E}[I_\phi^z(s)]$ for brevity, observe that
%\(
%\int_t^\infty   f_\phi'(s)ds = f_\phi(\infty) - f_\phi(t),
%\)
%so
 and recalling from $\l(\ref{bgreln}\r)$ that $ \int_0^\infty  \f{d}{ds} \bb{E}[I_\phi^z(s)] ds   = \bb{E}[I_\phi^z(\infty)]=\Gamma(z+1)/W_{\phi}(z+1)$, we can express $\bb{E}[I_\phi^z(t)]$ as 
\b{align}
\nonumber
\bb{E}[I_\phi^z(t)] 
&=
 \int_0^t \f{d}{ds} \bb{E}[I_\phi^z(s)] ds   =    \f{\Gamma(z+1)}{W_{\phi}(z+1)} -   \int_t^\infty  \f{d}{ds} \bb{E}[I_\phi^z(s)] ds
\\
\label{intsumm}
&=
    \f{\Gamma(z+1)}{W_{\phi}(z+1)} - \Gamma(z+1) \int_t^\infty \l(    \sum_{k=1}^\infty        \f{        \prod_{j=1}^{k}        [ \phi(z+j) - \phi(k)]    }{      \prod_{j=1}^{k-1}      [ \phi(j) - \phi(k)]      }  \  \f{     e^{-\phi(k)s}   }{ W_{ \phi_{(k)}}(z+1)  }   \r)  ds.
\end{align}

\n Now, we wish to exchange the order of integration and summation in $\l(\ref{intsumm}\r)$. If we can show that 
\b{equation}
\label{intsumabsconv}
 \sum_{k=1}^\infty  \l(       \int_t^\infty  \l|  \f{  \prod_{i=1}^{k}        [ \phi(z+i) - \phi(k)]    }{   \prod_{j=1}^{k-1}      [ \phi(j) - \phi(k)]     } \   \f{     e^{-\phi(k)s}   }{ W_{ \phi_{(k)}}(z+1)  }     
 \r|
 ds
 \r)      <\infty,
\end{equation}
then by Fubini's theorem, we can exchange the order of integration and summation, yielding the result%  and then observing that $\phi(0)=0$, we get
\b{align}
\nonumber
  \int_t^\infty         \sum_{k=1}^\infty   \f{   \prod_{i=1}^{k}        [ \phi(z+i) - \phi(k)]    }{   \prod_{j=1}^{k-1}      [ \phi(j) - \phi(k)]     }    \f{     e^{-\phi(k)s}   ds }{ W_{ \phi_{(k)}}(z+1)  }      
 &=
 \sum_{k=1}^\infty        \int_t^\infty    \f{  \prod_{i=1}^{k}     [ \phi(z+i) - \phi(k)]    }{  \prod_{j=1}^{k-1}      [ \phi(j) - \phi(k)]    }    \f{     e^{-\phi(k)s}   ds  }{ W_{ \phi_{(k)}}(z+1)  }     
\\
\label{asabove}
&= 
 \sum_{k=1}^\infty          \f{    \prod_{i=1}^{k}     [ \phi(z+i) - \phi(k)]    }{  \prod_{j=1}^{k-1}      [ \phi(j) - \phi(k)]    }    \f{     e^{-\phi(k)t}   }{ \phi(k) W_{ \phi_{(k)}}(z+1)  }     .
%\\
%\nonumber %\label{asabove}
%&=- 
 %\sum_{k=1}^\infty          \f{   \prod_{i=1}^{k}        [ \phi(z+i) - \phi(k)]    }{     \prod_{j=0}^{k-1}     [ \phi(j) - \phi(k)]     } \   \f{     e^{-\phi(k)t}   }{   W_{ \phi_{(k)}}(z+1)  } .
\end{align}
Substituting this into $\l(\ref{intsumm}\r)$, 
we conclude, as required, that
\begin{equation*}
   \bb{E}[I_\phi^z(t)] 
   =  \f{ \Gamma(z+1)   }{  W_{\phi}(z+1)    }  -  
   \sum_{k=1}^\infty      \f{       \prod_{i=1}^{k}     [ \phi(z+i) - \phi(k)]    }{     \prod_{j=1}^{k-1}       [ \phi(j) - \phi(k)]  } 
   \f{     e^{-\phi(k)t}     }{   \phi(k)}   \f{ \Gamma(z+1)     }{   W_{\phi_{(k)}}(z+1)  }
%     \f{\Gamma(z+1)}{W_{\phi}(z+1)}-  
%\Gamma(z+1) \sum_{k=1}^\infty          \f{   \prod_{i=1}^{k}       [ \phi(z+i) - \phi(k)]    }{    \prod_{j=1}^{k-1}   [ \phi(j) - \phi(k)]     }  \  \f{     e^{-\phi(k)t}   }{   W_{ \phi_{(k)}}(z+1)  } 
%\\
% &=
%  \Gamma(z+1)   
% \sum_{k=0}^\infty          \f{   \prod_{i=1}^{k}      [ \phi(z+i) - \phi(k)]    }{  \prod_{j=0}^{k-1}    [ \phi(j) - \phi(k)]      }  \  \f{     e^{-\phi(k)t}   }{ \phi(k)  W_{ \phi_{(k)}}(z+1)  } 
% \\
%  &=
%  \Gamma(z+1)   
% \sum_{k=0}^\infty  \l(      \prod_{j=0}^{k-1}     \f{      [ \phi(z+1+j) - \phi(k)]    }{     [ \phi(j) - \phi(k)]    }  \r) \f{     e^{-\phi(k)t}   }{   W_{ \phi_{(k)}}(z+1)  } .
\end{equation*}
or relation \eqref{theoremformula} holds true. To see that   $\l(\ref{intsumabsconv}\r)$ is finite, following the same argument as in $\l(\ref{asabove}\r)$, noting that $\phi$ is non-decreasing, %  since $\phi(0)=0$, 
\b{align*}
\l(\ref{intsumabsconv}\r) 
&=
 \sum_{k=1}^\infty       \l|  \f{ \prod_{i=1}^{k}       [ \phi(z+i) - \phi(k)]    }{  \prod_{j=1}^{k-1}       [ \phi(j) - \phi(k)]    }  \  \f{        e^{-\phi(k)t}    }{ \phi(k)   W_{ \phi_{(k)}}(z+1)  }     
 \r| 
\\
&\leq \f{1}{\phi(1)} \sum_{k=1}^\infty        \l|  \f{  \prod_{i=1}^{k}      [ \phi(z+i) - \phi(k)]    }{   \prod_{j=1}^{k-1}     [ \phi(j) - \phi(k)]   }  \  \f{      e^{-\phi(k)t}    }{ W_{ \phi_{(k)}}(z+1)  }     
 \r| ,
% \\
%&= \sum_{k=1}^\infty        \l|  \underset{0\leq j \leq k-1}{\prod}    \f{      [ \phi(z+1+j) - \phi(k)]    }{     [ \phi(j) - \phi(k)]   }  \f{      e^{-\phi(k)t}    }{ W_{ \phi_{(k)}}(z+1)  }     
% \r| 
\end{align*}
and then it follows  by Lemma \mbox{\ref{absconvlemma}}   that  $\l(\ref{intsumabsconv}\r)$ is a finite quantity and so the proof of Theorem \mbox{\ref{mainthm}} is complete.

	 \end{proof}

%	 \subsection{Proofs of Proposition \mbox{\ref{svcomparison}}  \ \&  Corollary \mbox{\ref{integermoments}}} 
\n 	 Before we prove Corollary \mbox{\ref{integermoments}},   we   state the following fact about Vandermonde matrices, see e.g.\ \cite[p37]{hj90}. For $a_1,\dots,a_n\in\bb{C}$, the determinant of the $n\times n$ Vandermonde matrix has the   form
\b{equation}
\label{vandermondedeterminant}
\det \begin{pmatrix}
  1 & a_1 & a_1^2 & \cdots & a_1^{n-1} \\
  1 & a_2 & a_2^2 & \cdots &a_2^{n-1} \\
  \vdots  & \vdots & \vdots  & \ddots & \vdots  \\
  1 & a_n & a_n^2 &  \cdots & a_n^{n-1}
 \end{pmatrix}
 =
   \underset{1\leq h<l\leq n }{\prod}     ( a_l- a_h). 
\e{equation}
We derive some elementary consequence of this representation, which perhaps is located somewhere in the literature.
\begin{proposition}\label{prop:Van}
	For any $n+1$ different complex numbers $a_0,a_1,\cdots, a_n$ it holds that
	\b{equation}
	\label{sum=0}
	\sum_{k=0}^{n} \ \ \    \f{  1   }{  \underset{0\leq j \leq n;   \ j\neq k}{\prod}     [ a_l - a_k]  }    =0.
	\e{equation}
\end{proposition} 
 \begin{proof}[Proof of Proposition \mbox{\ref{prop:Van}}]
 	To derive \eqref{sum=0} we multiply through by $\underset{0\leq h<l \leq n}{\prod} [a_h-a_l]$,
 	and we see that $\l(\ref{sum=0}\r)$ holds if and only if
 	\b{equation} \label{vdet}
 	\sum_{k=0}^{n} \l(    (-1)^k    \underset{    h,l\neq k }{ \underset{0\leq h<l \leq n; }{\prod} }     [ a_h - a_l]   \r)     =0.
 	\e{equation}
 	Applying the formula $\l(\ref{vandermondedeterminant}\r)$  for     Vandermonde determinants,
 	%\[
 	%\det \begin{pmatrix}
 	%  1 & a_1 & a_1^2 & \cdots & a_1^{m-1} \\
 	%  1 & a_2 & a_2^2 & \cdots &a_2^{m-1} \\
 	%  \vdots  & \vdots & \vdots  & \ddots & \vdots  \\
 	%  1 & a_m & a_m^2 &  \cdots & a_m^{m-1}
 	% \end{pmatrix}
 	% =
 	%   \underset{1\leq h<l\leq m }{\prod}     ( a_l- a_h). 
 	%\] 
 	%Then
 	one can verify that the sum  in $\l(\ref{vdet}\r)$   equals the following determinant, evaluated using the  cofactor expansion along its first column:
 	\[
 	\det \begin{pmatrix}
 	1& 1 & a_0 & a_0^2 & \cdots & a_0^{n-1} \\
 	1& 1 & a_1 & a_1^2 & \cdots & a_1^{n-1} \\
 	1& 1 & a_2 & a_2^2 &  \cdots & a_2^{n-1} \\
 	\vdots & \vdots & \vdots & \vdots & \ddots & \vdots \\
 	1& 1 & a_n & a_n^2 &  \cdots & a_n^{n-1} \\
 	\end{pmatrix}.
 	\]
 	This determinant is 0 as the first two columns are identical, so    $\l(\ref{vdet}\r)=0$. This proves \eqref{sum=0}.
 \end{proof}

	\begin{proof}[Proof of Corollary \mbox{\ref{integermoments}}] 
		We specialise \eqref{eq:LT} for the case $z=n$ and use the recurrent relation \eqref{eq:Bern-Gamma} to get 
		\begin{equation}\label{eq:LT1}
		\begin{split}
		&\   \cal{L}\big|_q \l\{   \hspace{0.5pt}  \bb{E}\big[   \hspace{0.5pt}   I_{\phi}^n(\cdot)    \hspace{0.5pt}    \big]\r\}=\int_{0}^\infty e^{-qt}\Ebb{I^n_\phi(t)}dt=\frac1q\frac{n!}{\prod_{j=1}^n\lbrb{ q+\phi(j)}}.
		\end{split}
		\end{equation}
		From here we immediately get that for any $t>0$
		\begin{equation*}
		\begin{split}
		&\Ebb{I^n_\phi(t)}= n!\lbrb{1 \ast \overset{n}{\underset{k=1}{\conv}}e^{-\phi(k)t}}.
		\end{split}
		\end{equation*}
		Clearly, for $n=1$,
		\begin{equation*}
		\begin{split}
		&\Ebb{I_\phi(t)}=\int_{0}^te^{-\phi(1)s}ds=\frac{1-e^{-\phi(1)t}}{\phi(1)}=\frac{e^{-\phi^*(0)t}-e^{-\phi^*(1)t}}{\phi^*(1)-\phi^*(0)},
		\end{split}
		\end{equation*}
		which gives the first identity of \eqref{intmoments} for $n=1$. We proceed by induction assuming that the first identity of \eqref{intmoments} holds for $n=N$ and all $t>0$. We consider $\Ebb{I^{N+1}_\phi(t)}$. From above and the inductive hypothesis we have that
		\begin{equation*}
		\begin{split}
		&\Ebb{I^{N+1}_\phi(t)}= (N+1)!\lbrb{1 \ast \overset{N+1}{\underset{k=1}{\conv}}e^{-\phi(k)t}}=(N+1)\int_{0}^t\Ebb{I^{N}_\phi(s)}e^{-\phi(N+1)(t-s)}ds\\
		&=(N+1)!\sum_{k=0}^{N-1}\frac{1}{ \   \underset{0\leq j \leq N; \hspace{2pt}  j \neq k}{\prod}     [ \phi^*(j) - \phi^*(k)]  }   \int_{0}^t  \lbrb{       e^{-\phi^*(k)s}-e^{-\phi^*(N)s}} e^{-\phi^*(N+1)(t-s)}ds\\
		&=(N+1)!\sum_{k=0}^{N-1}\frac{e^{-\phi^*(N+1)t}}{ \   \underset{0\leq j \leq N; \hspace{2pt}  j \neq k}{\prod}     [ \phi^*(j) - \phi^*(k)]  }   \int_{0}^t  \lbrb{       e^{\lbrb{\phi^*(N+1)-\phi^*(k)}s}-e^{\lbrb{\phi^*(N+1)-\phi^*(N)}s}}ds\\
		&=(N+1)!\sum_{k=0}^{N-1}\frac{e^{-\phi^*(k)t}-e^{-\phi^*(N+1)t}}{ \   \underset{0\leq j \leq N+1; \hspace{2pt}  j \neq k}{\prod}     [ \phi^*(j) - \phi^*(k)]} \\
		&\,\, -(N+1)!\frac{e^{-\phi^*(N)t}-e^{-\phi^*(N+1)t}}{\lbbrbb{\phi^*(N+1)-\phi^*(N)} }\sum_{k=0}^{N-1}\frac{1}{ \   \underset{0\leq j \leq N; \hspace{2pt}  j \neq k}{\prod}     [ \phi^*(j) - \phi^*(k)]  }.  
		\end{split}
		\end{equation*}
		Now, since $\phi^*(j),0\leq j\leq N$, are different numbers, from \eqref{sum=0} we conclude that
		\begin{equation}\label{eq:id}
		\begin{split}
		&-\sum_{k=0}^{N-1}\frac{1}{ \   \underset{0\leq j \leq N; \hspace{2pt}  j \neq k}{\prod}     [ \phi^*(j) - \phi^*(k)] }=\frac{1}{ \   \underset{0\leq j \leq N; \hspace{2pt}  j \neq N}{\prod}     [ \phi^*(j) - \phi^*(N)] }.
		\end{split}
		\end{equation}
		Hence, substituting above
		\begin{equation*}
		\begin{split}
		&\Ebb{I^{N+1}_\phi(t)}=(N+1)!\lbrb{\sum_{k=0}^{N-1}\frac{e^{-\phi^*(k)t}-e^{-\phi^*(N+1)t}}{ \   \underset{0\leq j \leq N+1; \hspace{2pt}  j \neq k}{\prod}     [ \phi^*(j) - \phi^*(k)]}+\frac{e^{-\phi^*(N)t}-e^{-\phi^*(N+1)t}}{ \   \underset{0\leq j \leq N+1; \hspace{2pt}  j \neq N}{\prod}     [ \phi^*(j) - \phi^*(N)]  }}\\
		&= (N+1)!\sum_{k=0}^{N}\frac{e^{-\phi^*(k)t}-e^{-\phi^*(N+1)t}}{ \   \underset{0\leq j \leq N+1; \hspace{2pt}  j \neq k}{\prod}     [ \phi^*(j) - \phi^*(k)]},
		\end{split}
		\end{equation*}
		which verifies the inductive hypothesis and thus the first identity of \eqref{intmoments} is proven. For the second we use again that \eqref{eq:id} holds true to deduce that for any $n$
		\[-e^{-\phi^*(n)t}\sum_{k=0}^{n-1}\frac{1}{ \   \underset{0\leq j \leq n; \hspace{2pt}  j \neq k}{\prod}     [ \phi^*(j) - \phi^*(k)]}=\frac{-e^{-\phi^*(n)t}}{\   \underset{0\leq j \leq n; \hspace{2pt}  j \neq n}{\prod}     [ \phi^*(j) - \phi^*(n)] }\]
		and upon substitution, \eqref{eq:id} follows.
	\end{proof}
	
\n The   proof of Lemma  \mbox{\ref{nthtermlemma}}   relies upon the following lemma, which we shall prove after Lemma  \mbox{\ref{nthtermlemma}}.
\begin{lemma} 
\label{nthtermproof} 
For each $\phi\in\cal{B}$, and for all $n\geq1$,
\b{equation} \label{stp0}
  \sum_{k=1}^{n}  (-1)^{k}   \l(          \underset{h,l \neq k}{ \underset{  1\leq h<l \leq n;  }{\prod} }      [ \phi(h) - \phi(l)] \r)    \l( - \underset{1\leq i\leq n-1 }{\prod}      [ \phi(z+i) - \phi(k)]   \r)     =    \underset{1\leq h<l\leq n }{\prod}     [ \phi(l) - \phi(h)]    .
\e{equation}
\end{lemma}

\begin{proof}[Proof of Lemma \mbox{\ref{nthtermlemma}}] We use  a proof by induction. For the base case, $n=1$, we need to verify  
 \b{equation}\begin{split}
 \label{basecase}
   \f{z t^{z-1}}{ e^{\phi(1)t}     } 
\ast  \Bigg[   \delta_0(dt)    &+    \sum_{m=1}^\infty   \f{  (z)^{(m)} }{  (m!)^2  }  \l( - [\phi(2)-\phi(1)]\r)^{m}    \f{ m    t^{m-1}  }{e^{\phi(1)t}}
\\ 
 &+   [\phi(z+1)-\phi(1)]     \sum_{m=0}^\infty      \f{  (z)^{(m)} }{  (m!)^2  }  \l( - [\phi(2)-\phi(1)]\r)^{m}    \f{t^m }{e^{\phi(1)t}}   \Bigg]
\end{split}
\end{equation}
\[
=
\f{ z t^{z-1}  }{ e^{\phi(2)t}    }
+   z       [ \phi(z+1) - \phi(1)]         \f{        \gamma(z,[\phi(2)-\phi(1)]t)   }{      [\phi(2)-\phi(1)]^z    \    e^{\phi(1)t}   }  .
\]

 \n Applying  the properties      of convolutions  $\l(\ref{m-1}\r)$, $\l(\ref{m}\r)$, and $\l(\ref{expout}\r)$, we can write
\begin{align*}
\l(\ref{basecase}\r)
&=   \f{ z  t^{z-1} }{ e^{\phi(1) t} }       +  z \sum_{m=1}^\infty   \f{  (z)^{(m)} }{  (m!)^2  } \l( - [\phi(2)-\phi(1)]\r)^{m}   \f{  ( t^{z-1} \ast  m t^{m-1}  ) }{   e^{\phi(1)t}}
\\
 &+  z [\phi(z+1)-\phi(1)]      \sum_{m=0}^\infty      \f{  (z)^{(m)} }{  (m!)^2  }  \l( - [\phi(2)-\phi(1)]\r)^{m}    \f{ ( t^{z-1} \ast  t^{m}  )}{   e^{\phi(1)t}}
\\  
&=   \f{ z  t^{z-1} }{ e^{\phi(1) t} }       +  \f{z}{ e^{\phi(1)t}} \sum_{m=1}^\infty   \f{  1 }{  m!  }  \l( - [\phi(2)-\phi(1)]\r)^{m}     t^{z+m-1}  
\\
  &+  \f{ z [\phi(z+1)-\phi(1)] }{e^{\phi(1)t} }    \sum_{m=0}^\infty      \f{  1}{  m!  }  \l( - [\phi(2)-\phi(1)]\r)^{m}     \int_0^t  s^{z+m-1} ds
  \\
  &=     \f{zt^{z-1}}{ e^{\phi(1)t}} \sum_{m=0}^\infty   \f{  1 }{  m!  }  \l( - [\phi(2)-\phi(1)]\r)^{m}     t^{m}  
\\
&+  \f{ z [\phi(z+1)-\phi(1)] }{e^{\phi(1)t} }    \int_0^t  s^{z-1}  \l(   \sum_{m=0}^\infty      \f{  1}{  m!  }  \l( - [\phi(2)-\phi(1)]\r)^{m}  s^m     \r)   ds.
\intertext{Using the fact that $e^{w}=\sum_{m=0}^\infty  \f{ w^m}{ m! } $, then applying a change of variables, it follows that }
\l(\ref{basecase}\r) 
&=     \f{zt^{z-1}}{ e^{\phi(1)t}}     e^{ - [\phi(2)-\phi(1)] t }
+  \f{ z [\phi(z+1)-\phi(1)] }{e^{\phi(1)t} }    \int_0^t  s^{z-1} e^{ - [\phi(2)-\phi(1)]s}  ds
\\
&=
\f{ z t^{z-1}  }{ e^{\phi(2)t}    }
+   z       [ \phi(z+1) - \phi(1)]         \f{        \gamma(z,[\phi(2)-\phi(1)]t)   }{      [\phi(2)-\phi(1)]^z    \    e^{\phi(1)t}   } ,
\end{align*}
as required for the base case.  For the inductive step,   assume that the formula in $\l(\ref{nthtermeqn}\r)$ holds for the $(n-1)$th convolution. Noting that \[\f{\gamma(z,[\phi(n)-\phi(k)]t)}{[\phi(n)-\phi(k)]^z}  = \int_0^t       e^{-[\phi(n)-\phi(k)]s} s^{z-1} ds,\] this means that
   \begin{align}
   \begin{split}
  \f{z t^{z-1}}{ e^{\phi(1)t}     } 
\ast \overset{n-1}{\underset{k=1}{\conv}} \Bigg[   \delta_0(dt)    &+  \sum_{m=1}^\infty   \f{  (z)^{(m)} }{  (m!)^2  } m \l( - [\phi(k+1)-\phi(k)]\r)^m  \f{ t^{m-1}}{ e^{\phi(k)t}}
\\
\label{indhyp}
 &+   [\phi(z+k)-\phi(k)]      \sum_{m=0}^\infty      \f{  (z)^{(m)} }{  (m!)^2  }  \l( - [\phi(k+1)-\phi(k)]\r)^{m}  \f{  t^m }{ e^{\phi(k)t}}   \Bigg]
 \end{split}
\\
\nonumber
&=
\f{ z t^{z-1}  }{ e^{\phi(n)t}    }
+   z \sum_{k=1}^{n-1}       \f{   \underset{1\leq j\leq n-1 }{\prod}      [ \phi(z+j) - \phi(k)]    }{   \underset{1\leq j \leq n-1; j\neq k}{\prod}    \hspace{-3pt}   [ \phi(j) - \phi(k)]  \hspace{6pt}   }     e^{-\phi(k)t}   \int_0^t       e^{-[\phi(n)-\phi(k)]s} s^{z-1} ds   .
\end{align}
We are going to evaluate the convolution of $\l(\ref{indhyp}\r)$ with the next term in $\l(\ref{convformula}\r)$, that is, the expression
\begin{equation*}
\begin{split}
&\lbrb{\f{z t^{z-1}}{   e^{\phi(n)t}    } 
	+   z \sum_{k=1}^{n-1}       \f{   \underset{1\leq j\leq n-1 }{\prod}      [ \phi(z+j) - \phi(k)]    }{   \underset{1\leq j \leq n-1; j\neq k}{\prod}  \hspace{-3pt}   [ \phi(j) - \phi(k)]  \hspace{5pt}   }         e^{-\phi(k)t}  \int_0^t       e^{-[\phi(n)-\phi(k)]s} s^{z-1} ds}\\
&\ast\Bigg[ \delta_0(dt)    +  \sum_{m=1}^\infty   \f{  (z)^{(m)} }{  (m!)^2  }  \l( - [\phi(n+1)-\phi(n)]\r)^m  \f{m t^{m-1}}{ e^{\phi(n)t}}
\\
&+   [\phi(z+n)-\phi(n)]      \sum_{m=0}^\infty      \f{  (z)^{(m)} }{  (m!)^2  }  \l( - [\phi(n+1)-\phi(n)]\r)^{m}  \f{ t^m }{e^{\phi(n)t}} \Bigg] .		    
\end{split}
\end{equation*}
\n For brevity, we    label the above convolution as 
\b{equation}
\label{convs}
\l((1A) + \sum_{k=1}^{n-1} (1k)  \r)  \ast \l( \delta_0(dt)  + (HG1) + (HG2)\r).
\e{equation}
First, observe that $\l((1A) + \sum_{k=1}^{n-1} (1k)  \r) \ast \delta_0(dt) = (1A) + \sum_{k=1}^{n-1} (1k) $.  We will see that this contribution  cancels with some of the other terms arising from $\l(\ref{convs}\r)$, which we evaluate individually.

\p{Evaluating the $(1A)\ast(HG1)$ term}     Applying the properties $\l(\ref{expout}\r)$ and     $\l(\ref{m-1}\r)$, we can write
\begin{align}
\nonumber
(1A)\ast(HG1) 
&\overset{\phantom{(\ref{expout})}}=    z t^{z-1}   e^{-\phi(n)t}    
\ast
 \sum_{m=1}^\infty   \f{  (z)^{(m)} }{  (m!)^2  }  \l( - [\phi(n+1)-\phi(n)]\r)^m   \f{mt^{m-1}}{ e^{\phi(n)t}}
\\
\nonumber
&\overset{(\ref{expout})}=     z   \sum_{m=1}^\infty   \f{  (z)^{(m)} }{  (m!)^2  }  \l( - [\phi(n+1)-\phi(n)]\r)^m \f{ (t^{z-1} \ast  m t^{m-1})}{ e^{\phi(n)t}}
\\
\nonumber
&\overset{(\ref{m-1})}=   z   \sum_{m=1}^\infty   \f{  1}{  m!  }  \l( - [\phi(n+1)-\phi(n)]\r)^m \f{ t^{z+m-1} }{ e^{\phi(n)t}}
\\
\nonumber
&\overset{\phantom{(\ref{expout})}}=    \f{ z t^{z-1}}{e^{\phi(n)t}     }     \sum_{m=1}^\infty   \f{ 1}{  m!  }  \l( - [\phi(n+1)-\phi(n)]t\r)^{m}    
\\
\label{1A*HG1}
&\overset{\phantom{(\ref{expout})}}=  \f{ z  t^{z-1} }{e^{\phi(n+1)t}}       -     \f{ z  t^{z-1} }{e^{\phi(n)t} } .
\end{align}
%Here we have the desired first piece, $z  t^{z-1} e^{-\phi(n+1)t}  $, of the $n$th term. Also observe that $    -       z  t^{z-1} e^{-\phi(n)t} $ cancels with the first piece of the  $(n-1)$th term.

\p{Evaluating the $(1A)\ast(HG2)$ term}    Applying the properties $\l(\ref{expout}\r)$ and     $\l(\ref{m}\r)$, we can write
\begin{align}
\nonumber
(1A)\ast(HG2) 
&\overset{\phantom{(\ref{expout})}}=   z t^{z-1}   e^{-\phi(n)t}    
\ast
  [\phi(z+n)-\phi(n)]      \sum_{m=0}^\infty      \f{  (z)^{(m)} }{  (m!)^2  }  \l( - [\phi(n+1)-\phi(n)]\r)^{m}  \f{ t^m }{e^{\phi(n)t}} 
  \\
  \nonumber 
  &\overset{(\ref{expout})}=   
  z   [\phi(z+n)-\phi(n)]      \sum_{m=0}^\infty      \f{  (z)^{(m)} }{  (m!)^2  }  \l( - [\phi(n+1)-\phi(n)]\r)^{m}  \f{  (t^{z-1} \ast t^m  )}{e^{\phi(n)t}} 
  \\
\nonumber
&\overset{(\ref{m})}=  
 z   [\phi(z+n)-\phi(n)]      \sum_{m=0}^\infty      \f{  1 }{  m!  }  \l( - [\phi(n+1)-\phi(n)]\r)^{m}     \f{\int_0^t s^{z+m-1}ds}{e^{\phi(n)t}} 
 \\
 \nonumber
 &\overset{\phantom{(\ref{expout})}}=  \f{ z   [\phi(z+n)-\phi(n)] }{e^{\phi(n)t}}         \int_0^t         s^{z-1}    \sum_{m=0}^\infty      \f{  1 }{  m!  }  \l( - [\phi(n+1)-\phi(n)]s\r)^{m} ds
 \\
 \label{1A*HG2}
 &\overset{\phantom{(\ref{expout})}}=  \f{ z   [\phi(z+n)-\phi(n)] }{e^{\phi(n)t}}         \int_0^t         s^{z-1}  e^{ - [\phi(n+1)-\phi(n)]s} ds.
\end{align}

\p{Evaluating each $(1k)\ast(HG1)$ term}  To evaluate $(1k)\ast(HG1)$, we first consider the quantity
\b{equation}
\label{simple1}
    \l(  e^{-\phi(k)t}  \int_0^t       e^{-[\phi(n)-\phi(k)]s} s^{z-1} ds \r)\ast  \sum_{m=1}^\infty   \f{  (z)^{(m)} }{  (m!)^2  }  \l( - [\phi(n+1)-\phi(n)]\r)^m  \f{m t^{m-1}}{ e^{\phi(n)t}}.
\e{equation}
Observe that \[e^{-\phi(k)t}  \int_0^t       e^{-[\phi(n)-\phi(k)]s} s^{z-1} ds = e^{-\phi(k)t} \ast  t^{z-1}   e^{-\phi(n)t}.\] Then by $\l(\ref{expout}\r)$ and $\l(\ref{m-1}\r)$,
\begin{align}
\nonumber 
\l(\ref{simple1}\r)
&\overset{\phantom{(\ref{expout})}}=
\l(e^{-\phi(k)t} \ast  t^{z-1}   e^{-\phi(n)t}\r) 
\ast 
  \sum_{m=1}^\infty   \f{  (z)^{(m)} }{  (m!)^2  }  \l( - [\phi(n+1)-\phi(n)]\r)^m  \f{m t^{m-1}}{ e^{\phi(n)t}}
\\
\nonumber 
  &\overset{(\ref{expout})}=   
 e^{-\phi(k)t}   
\ast 
  \sum_{m=1}^\infty   \f{  (z)^{(m)} }{  (m!)^2  }  \l( - [\phi(n+1)-\phi(n)]\r)^m  \f{ ( t^{z-1} \ast m t^{m-1})}{ e^{\phi(n)t}}
  \\
\nonumber 
  &\overset{(\ref{m-1})}=   
 e^{-\phi(k)t}   
\ast 
  \sum_{m=1}^\infty   \f{  1 }{  m!  }  \l( - [\phi(n+1)-\phi(n)]\r)^m  \f{  t^{z+m-1}  }{ e^{\phi(n)t}}.
    \intertext{Using the fact that $e^{w}-1=\sum_{m=1}^\infty  \f{ w^m}{ m! } $, we can evaluate this convolution as follows}
\nonumber 
\l(\ref{simple1}\r) &\overset{\phantom{(\ref{expout})}}=
  e^{-\phi(k)t}   
\ast 
   \f{t^{z-1}}{e^{\phi(n)t}}  \sum_{m=1}^\infty   \f{  1 }{  m!  }  \l( - [\phi(n+1)-\phi(n)]t\r)^m 
      \\
\nonumber 
&\overset{\phantom{(\ref{expout})}}=
  e^{-\phi(k)t}   
\ast 
   \f{t^{z-1}}{e^{\phi(n)t}} \l( e^{- [\phi(n+1)-\phi(n)]t} -1\r)
      \\
\nonumber 
&\overset{\phantom{(\ref{expout})}}=
  e^{-\phi(k)t}   
\ast  
\l(
   \f{t^{z-1}}{e^{\phi(n+1)t}}  - \f{t^{z-1}}{e^{\phi(n)t}}
   \r)
         \\
\label{simple1final} 
&\overset{\phantom{(\ref{expout})}}=
 e^{-\phi(k)t}   \int_0^t  e^{-[\phi(n+1)-\phi(k)]s}   s^{z-1} ds    -    e^{-\phi(k)t}   \int_0^t  e^{-[\phi(n)-\phi(k)]s}   s^{z-1} ds.
\end{align}
Now, multiplying $\l(\ref{simple1}\r)$ by suitable constants as in  $\l(\ref{convs}\r)$,  we can express $(1k)\ast(HG1)$ as
\begin{equation}
\begin{split}
\label{1k*HG1}
(1k)\ast(HG1)
&=  z     \f{   \underset{1\leq j\leq n-1 }{\prod}      [ \phi(z+j) - \phi(k)]    }{   \underset{1\leq j \leq n-1; j\neq k}{\prod}   \hspace{-3pt}   [ \phi(j) - \phi(k)]  \hspace{6pt}   }    e^{-\phi(k)t}   \int_0^t  e^{-[\phi(n+1)-\phi(k)]s}   s^{z-1} ds  
\\
& -    z     \f{   \underset{1\leq j\leq n-1 }{\prod}      [ \phi(z+j) - \phi(k)]    }{   \underset{1\leq j \leq n-1; j\neq k}{\prod}   \hspace{-3pt}   [ \phi(j) - \phi(k)]  \hspace{6pt}   }   e^{-\phi(k)t}   \int_0^t  e^{-[\phi(n)-\phi(k)]s}   s^{z-1} ds.
\end{split}
\end{equation}

\p{Evaluating each $(1k)\ast(HG2)$ term}  To evaluate $(1k)\ast(HG2)$, we first consider the quantity
\b{equation}
\label{simple2}
    \l(  e^{-\phi(k)t}  \int_0^t       e^{-[\phi(n)-\phi(k)]s} s^{z-1} ds \r)
    \ast 
       \sum_{m=0}^\infty      \f{  (z)^{(m)} }{  (m!)^2  }  \l( - [\phi(n+1)-\phi(n)]\r)^{m}  \f{ t^m }{e^{\phi(n)t}}.
\e{equation}

\n Recall that $  e^{-\phi(k)t}  \int_0^t       e^{-[\phi(n)-\phi(k)]s} s^{z-1} ds = e^{-\phi(k)t} \ast  t^{z-1}   e^{-\phi(n)t}  $.  Then by $\l(\ref{expout}\r)$ and $\l(\ref{m}\r)$,
\begin{align}
\nonumber 
\l(\ref{simple2}\r)
&\overset{\phantom{(\ref{expout})}}=
\l(e^{-\phi(k)t} \ast  t^{z-1}   e^{-\phi(n)t}\r) 
\ast 
  \sum_{m=0}^\infty      \f{  (z)^{(m)} }{  (m!)^2  }  \l( - [\phi(n+1)-\phi(n)]\r)^{m}  \f{ t^m }{e^{\phi(n)t}}
\\
\nonumber 
  &\overset{(\ref{expout})}=   
 e^{-\phi(k)t}   
\ast 
  \sum_{m=0}^\infty      \f{  (z)^{(m)} }{  (m!)^2  }  \l( - [\phi(n+1)-\phi(n)]\r)^{m}  \f{ (t^{z-1} \ast t^m) }{e^{\phi(n)t}}
  \\
\nonumber 
  &\overset{(\ref{m})}=   
 e^{-\phi(k)t}   
\ast 
  \sum_{m=0}^\infty   \f{  1 }{  m!  }  \l( - [\phi(n+1)-\phi(n)]\r)^m  \f{   \int_0^t v^{z+m-1} dv  }{ e^{\phi(n)t}}.
    \intertext{Changing the order of summation and integration, since $e^{w}=\sum_{m=0}^\infty  \f{ w^m}{ m! } $, it follows that}
    \nonumber 
\l(\ref{simple2}\r)
&\overset{\phantom{(\ref{expout})}}=
 e^{-\phi(k)t}   
\ast   
\l( e^{-\phi(n)t}   \int_0^t   v^{z-1} 
  \sum_{m=0}^\infty   \f{  1 }{  m!  }  \l( - [\phi(n+1)-\phi(n)]v\r)^m  dv  
  \r)
  \\
     \nonumber  
&\overset{\phantom{(\ref{expout})}}=
 e^{-\phi(k)t}   
\ast   
\l( e^{-\phi(n)t}   \int_0^t   v^{z-1} e^{- [\phi(n+1)-\phi(n)]v}  dv  
  \r)
   \\
     \nonumber  
&\overset{\phantom{(\ref{expout})}}=
   \int_0^t         e^{-\phi(k)(t-v)}   \l(  e^{-\phi(n)v} \int_0^v     s^{z-1}    e^{-[\phi(n+1)- \phi(n)] s  }    ds  \r)dv.
    \intertext{Exchanging the order of integration, then evaluating the inner integral, }
    \nonumber 
\l(\ref{simple2}\r)
&\overset{\phantom{(\ref{expout})}}=
  e^{- \phi(k)t}  \int_0^t        s^{z-1}    e^{-[\phi(n+1)- \phi(n)] s  }         \int_s^t    e^{-[\phi(n)-\phi(k)]v}     dv ds
   \\
     \nonumber  
%&\overset{\phantom{(\ref{expout})}}=
%  \f{e^{- \phi(k)t}  }{   [\phi(n)-\phi(k)]    }   \int_0^t        s^{z-1}    e^{-[\phi(n+1)- \phi(n)] s }            \l(e^{-[\phi(n)-\phi(k)]s}    - e^{-[\phi(n)-\phi(k)]t} \r)   ds
%   \\
%     \nonumber  
&\overset{\phantom{(\ref{expout})}}=
  \f{e^{- \phi(k)t}  }{   [\phi(n)-\phi(k)]    }   \int_0^t        s^{z-1}    e^{-[\phi(n+1)- \phi(n)] s  }            \l(e^{-[\phi(n)-\phi(k)]s}    - e^{-[\phi(n)-\phi(k)]t} \r)   ds
    \\
     \nonumber  
&\overset{\phantom{(\ref{expout})}}=
  \f{   e^{- \phi(k)t}   }{   [\phi(n)-\phi(k)]    }    \int_0^t        s^{z-1}    e^{-[\phi(n+1)-\phi(k)]s}       ds
  -    \f{ e^{- \phi(n)t}   }{   [\phi(n)-\phi(k)]    }    \int_0^t        s^{z-1}    e^{-[\phi(n+1)- \phi(n)] s  }                ds.
  \end{align}
Now, multiplying $\l(\ref{simple2}\r)$ by suitable constants as in  $\l(\ref{convs}\r)$,  we can express $(1k)\ast(HG2)$ as
\begin{equation}
\begin{split}
\label{1k*HG2}
(1k)\ast(HG2)
&=  z     \f{    \underset{1\leq j\leq n-1 }{\prod}      [ \phi(z+j) - \phi(k)]    }{   \underset{1\leq j \leq n; j\neq k}{\prod}     [ \phi(j) - \phi(k)]  \hspace{15pt}   }    \f{[ \phi(z+n) - \phi(n)] }{ e^{ \phi(k)t}   } \int_0^t        s^{z-1}    e^{-[\phi(n+1)-\phi(k)]s}       ds
\\
& -    z     \f{   \underset{1\leq j\leq n-1 }{\prod}      [ \phi(z+j) - \phi(k)]    }{   \underset{1\leq j \leq n; j\neq k}{\prod}     [ \phi(j) - \phi(k)]  \hspace{15pt}   }     \f{ [ \phi(z+n) - \phi(n)] }{e^{ \phi(n)t}      } \int_0^t        s^{z-1}    e^{-[\phi(n+1)- \phi(n)] s  }                ds.
\end{split}
\end{equation}

\n Combining $\l(\ref{1A*HG1}\r)$, $\l(\ref{1A*HG2}\r)$ , $\l(\ref{1k*HG1}\r)$,  and $\l(\ref{1k*HG2}\r)$  gives an unwieldly formula  for $\l(\ref{convs}\r)$. Observe that    $\big[(1A) + \sum_{k=1}^{n-1} (1k)  \big] \ast \delta_0(dt) = (1A) + \sum_{k=1}^{n-1} (1k) $   cancels respectively with the second terms from $\l(\ref{1A*HG1}\r)$ and $\l(\ref{1k*HG1}\r)$. Then
\begin{align}
\label{tidy1}
\l(\ref{convs}\r) &=    \f{ z  t^{z-1} }{e^{\phi(n+1)t}}       
+  \f{ z   [\phi(z+n)-\phi(n)] }{e^{\phi(n)t}}         \int_0^t         s^{z-1}  e^{ - [\phi(n+1)-\phi(n)]s} ds
\\
\label{tidy2}
&+
z   \sum_{k=1}^{n-1}   \f{   \underset{1\leq j\leq n-1 }{\prod}      [ \phi(z+j) - \phi(k)]    }{   \underset{1\leq j \leq n-1; j\neq k}{\prod}     [ \phi(j) - \phi(k)]  \hspace{15pt}   }    e^{-\phi(k)t}   \int_0^t  e^{-[\phi(n+1)-\phi(k)]s}   s^{z-1} ds  
\\
\label{tidy3}
&+
 z   \sum_{k=1}^{n-1}   \f{     \underset{1\leq j\leq n-1 }{\prod}      [ \phi(z+j) - \phi(k)]    }{   \underset{1\leq j \leq n; j\neq k}{\prod}     [ \phi(j) - \phi(k)]  \hspace{15pt}   }   \f{ [ \phi(z+n) - \phi(n)] }{   e^{ \phi(k)t} }     \int_0^t        s^{z-1}    e^{-[\phi(n+1)-\phi(k)]s}       ds
 \\
 \label{tidy4}
  &-    z   \sum_{k=1}^{n-1}   \f{    \underset{1\leq j\leq n-1 }{\prod}      [ \phi(z+j) - \phi(k)]    }{   \underset{1\leq j \leq n; j\neq k}{\prod}     [ \phi(j) - \phi(k)]  \hspace{15pt}   }   \f{ [ \phi(z+n) - \phi(n)] }{  e^{\phi(n)t}  }    \int_0^t        s^{z-1}    e^{-[\phi(n+1)- \phi(n)] s  }                ds.
\end{align}
Multiplying $\l(\ref{tidy2}\r)$ by $1=\f{ \phi( n  ) -   \phi(  k)}{ \phi( n  ) -   \phi(  k)}$,  we can combine the terms $\l(\ref{tidy2}\r) $ and $ \l(\ref{tidy3}\r)$, then recalling the definition $\gamma(z,u):=\int_0^u x^{z-1} e^{-x}dx$, we apply a simple  change of variables to get 
\b{align*}
\l(\ref{tidy2}\r) + \l(\ref{tidy3}\r) 
   &= z   \sum_{k=1}^{n-1}   \f{   \underset{1\leq j\leq n }{\prod}      [ \phi(z+j) - \phi(k)]    }{   \underset{1\leq j \leq n; j\neq k}{\prod}     [ \phi(j) - \phi(k)]  \hspace{15pt}   }      e^{- \phi(k)t}      \int_0^t        s^{z-1}    e^{-[\phi(n+1)-\phi(k)]s}       ds
\\
 &= z   \sum_{k=1}^{n-1}   \f{   \underset{1\leq j\leq n }{\prod}      [ \phi(z+j) - \phi(k)]    }{   \underset{1\leq j \leq n; j\neq k}{\prod}     [ \phi(j) - \phi(k)]  \hspace{15pt}   }        \f{        \gamma(z,[\phi(n+1)-\phi(k)]t)   }{      [\phi(n+1)-\phi(k)]^z    \    e^{\phi(k)t}   }  .
\end{align*}
Therefore we see that the terms in $\l(\ref{tidy2}\r) + \l(\ref{tidy3}\r) $ correspond to the desired sum between $1$ and $n-1$  as in   equation $\l(\ref{nthtermeqn}\r)$, and now all that remains for the proof of Lemma \mbox{\ref{nthtermlemma}} is to verify that 
\begin{equation}
\label{almostdone}
\l(\ref{tidy1}\r) + \l(\ref{tidy4}\r) =  \f{ z  t^{z-1} }{e^{\phi(n+1)t}}       
+
 z  \f{   \underset{1\leq j\leq n }{\prod}      [ \phi(z+j) - \phi(n)]    }{   \underset{1\leq j \leq n; j\neq n}{\prod}     [ \phi(j) - \phi(n)]  \hspace{15pt}   }       \f{        \gamma(z,[\phi(n+1)-\phi(n)]t)   }{      [\phi(n+1)-\phi(n)]^z    \    e^{\phi(n)t}   } .
\end{equation}
To show this,   first observe that  integral which appears in $\l(\ref{tidy1}\r)$ and $\l(\ref{tidy4}\r)$  can be written as
\[
    \int_0^t        s^{z-1}    e^{-[\phi(n+1)- \phi(n)] s  }                ds  =     \f{        \gamma(z,[\phi(n+1)-\phi(n)]t)   }{      [\phi(n+1)-\phi(n)]^z       }  , 
 \]
from which it follows that  $\l(\ref{almostdone}\r)$ holds if and only if 
\[
      [\phi(z+n)-\phi(n)]      -        \sum_{k=1}^{n-1}   \f{    \underset{1\leq j\leq n-1 }{\prod}      [ \phi(z+j) - \phi(k)]    }{   \underset{1\leq j \leq n; j\neq k}{\prod}     [ \phi(j) - \phi(k)]  \hspace{15pt}   }  [ \phi(z+n) - \phi(n)]   =  \f{   \underset{1\leq j\leq n }{\prod}      [ \phi(z+j) - \phi(n)]    }{   \underset{1\leq j \leq n; j\neq n}{\prod}     [ \phi(j) - \phi(n)]  \hspace{15pt}   }   ,
\]
  then rearranging and dividing through by $   [\phi(z+n)-\phi(n)] $, we see that  $\l(\ref{almostdone}\r)$ holds if and only if 
\[
1  =        \sum_{k=1}^{n-1}   \f{    \underset{1\leq j\leq n-1 }{\prod}      [ \phi(z+j) - \phi(k)]    }{   \underset{1\leq j \leq n; j\neq k}{\prod}     [ \phi(j) - \phi(k)]  \hspace{15pt}   }    + \f{   \underset{1\leq j\leq n-1 }{\prod}      [ \phi(z+j) - \phi(n)]    }{   \underset{1\leq j \leq n-1 }{\prod}     [ \phi(j) - \phi(n)]  \hspace{15pt}   }
=
  \sum_{k=1}^{n}   \f{    \underset{1\leq j\leq n-1 }{\prod}      [ \phi(z+j) - \phi(k)]    }{   \underset{1\leq j \leq n; j\neq k}{\prod}     [ \phi(j) - \phi(k)]  \hspace{15pt}   } 
   .
\]
Multiplying both sides by  $\underset{1\leq h<l\leq n }{\prod}     [ \phi(l) - \phi(h)]   $, one can verify that  $\l(\ref{almostdone}\r)$ holds if and only if 
\b{equation*} %\label{stp}
  \sum_{k=1}^{n}  (-1)^{k}   \l(          \underset{h,l \neq k}{ \underset{  1\leq h<l \leq n;  }{\prod} }      [ \phi(h) - \phi(l)] \r)   \l( - \underset{1\leq j\leq n-1 }{\prod}      [ \phi(z+j) - \phi(k)]   \r)     =    \underset{1\leq h<l\leq n }{\prod}     [ \phi(l) - \phi(h)]    .
\e{equation*}
But this equation holds    by Lemma \mbox{\ref{nthtermproof}}, and so  the proof of Lemma  \mbox{\ref{nthtermlemma}} is complete. 
\end{proof}

\begin{proof}[Proof of Lemma \mbox{\ref{nthtermproof}}]  
Applying
$\l(\ref{vandermondedeterminant}\r)$, then using elementary row and column operations, we get
 % We will now prove $\l(\ref{stp0}\r)$ using matrices. We need to show that the LHS and RHS in $\l(\ref{stp0}\r)$ are equal. Let us begin with the RHS, which can be written as
\begin{align}
\label{RHS}
 \underset{1\leq h<l\leq n }{\prod}     [ \phi(l) - \phi(h)]  &\overset{\l(\ref{vandermondedeterminant}\r)}=
  \det \hspace{5pt}
   \begin{pmatrix}
  1 & \phi(1) & \cdots &   \phi(1)^{n-2} & \phi(1)^{n-1}  \hspace{5pt}\\
  1 & \phi(2)  & \cdots & \phi(2)^{n-2} & \phi(2)^{n-1} \hspace{5pt}\\
  \vdots    & \vdots  & \ddots &\vdots & \vdots  \hspace{5pt}\\
  1 & \phi(n) &    \cdots &  \phi(n)^{n-2} & \phi(n)^{n-1}     \hspace{5pt}
 \end{pmatrix}
\\ 
\nonumber
&\overset{\phantom{\l(\ref{vandermondedeterminant}\r)}}=\det \begin{pmatrix}
 1&   0 & 0 & \cdots & 0 & 0 \\
0&   1 & \phi(1)&   \cdots & \phi(1)^{n-2} & \phi(1)^{n-1}  \\
 0&  1 & \phi(2)  & \cdots & \phi(2)^{n-2} & \phi(2)^{n-1}  \\
 \vdots & \vdots    & \vdots  & \ddots & \vdots &\vdots  \\
 0&  1 & \phi(n) &    \cdots &\phi(n)^{n-2} & \phi(n)^{n-1} 
 \end{pmatrix}
\\
\nonumber
&\overset{\phantom{\l(\ref{vandermondedeterminant}\r)}}=\det \begin{pmatrix}
 1& 0  & 0 & \cdots & 0 & 1 \\
0&   1 & \phi(1)&   \cdots & \phi(1)^{n-2} & \phi(1)^{n-1}  \\
 0&  1 & \phi(2)   & \cdots & \phi(2)^{n-2} & \phi(2)^{n-1}  \\
 \vdots & \vdots    & \vdots  & \ddots & \vdots &\vdots  \\
 0&  1 & \phi(n)   &  \cdots &\phi(n)^{n-2} & \phi(n)^{n-1} 
 \end{pmatrix}     . 
\end{align}
Next, we add a scalar multiple of the first row to each of the other rows, yielding
 \[
 \l(\ref{RHS}\r)
=\det \begin{pmatrix}
 1& 0 & 0 & \cdots & 0 & 1 \\
-   \overset{n-1}{\underset{i=1}{\prod}}     [ \phi(1) - \phi(z+i)] &   1 & \phi(1) & \cdots & \phi(1)^{n-2} & \phi(1)^{n-1}-   \overset{n-1}{\underset{i=1}{\prod}}  [ \phi(1) - \phi(z+i)]  \vspace{5pt} \\
-   \overset{n-1}{\underset{i=1}{\prod}}    [ \phi(2) - \phi(z+i)]&  1 & \phi(2) &   \cdots & \phi(2)^{n-2} & \phi(2)^{n-1} -   \overset{n-1}{\underset{i=1}{\prod}}  [ \phi(2) - \phi(z+i)] \\
 \vdots & \vdots    & \vdots  & \ddots & \vdots &\vdots  \\
-   \overset{n-1}{\underset{i=1}{\prod}}    [ \phi(n) - \phi(z+i)] &  1 & \phi(n)   &  \cdots &\phi(n)^{n-2} & \phi(n)^{n-1} -   \overset{n-1}{\underset{i=1}{\prod}}    [ \phi(n) - \phi(z+i)] 
 \end{pmatrix}    .   
\]

 \n Evaluating this   using the first row's cofactor expansion, we get two terms, the first of which is
   \[
  \det \begin{pmatrix} 
   1 & \phi(1) & \cdots & \phi(1)^{n-2} & \phi(1)^{n-1} -   \overset{n-1}{\underset{i=1}{\prod}}  [ \phi(1) - \phi(z+i)]   \vspace{5pt} \\
   1 & \phi(2) &   \cdots & \phi(2)^{n-2} & \phi(2)^{n-1} -   \overset{n-1}{\underset{i=1}{\prod}}    [ \phi(2) - \phi(z+i)] \\
  \vdots    & \vdots  & \ddots & \vdots &\vdots  \\
    1 & \phi(n)   &  \cdots &\phi(n)^{n-2} & \phi(n)^{n-1} -   \overset{n-1}{\underset{i=1}{\prod}}    [ \phi(n) - \phi(z+i)] 
 \end{pmatrix} =0,
\]
  since the last column is a linear combination of the other columns, and the remaining term is  
   \begin{align*}
    \l(\ref{RHS}\r)
    &=
    (-1)^{n+1} \  \det \begin{pmatrix} 
-   \overset{n-1}{\underset{i=1}{\prod}}   [ \phi(1) - \phi(z+i)] &   1 & \phi(1) & \cdots & \phi(1)^{n-2}  \vspace{5pt} \\
-   \overset{n-1}{\underset{i=1}{\prod}}    [ \phi(2) - \phi(z+i)] &    1 & \phi(2) &   \cdots & \phi(2)^{n-2}  \\
  \vdots    & \vdots  & \vdots & \ddots &\vdots  \\
 -   \overset{n-1}{\underset{i=1}{\prod}}   [ \phi(n) - \phi(z+i)]  &   1 & \phi(n)   &  \cdots &\phi(n)^{n-2}  
 \end{pmatrix}.
\intertext{We evaluate this using the first column's cofactor expansion. Each minor matrix  is  itself a Vandermonde matrix  % of the form
%
 %\[
%  \begin{pmatrix}
 % 1 & \phi(1) & \cdots &   \phi(1)^{n-2}  \\
%  \vdots    & \vdots  & \ddots &\vdots    \\
%    1 & \phi(k-1) & \cdots &   \phi(k-1)^{n-2}   \\
 %     1 & \phi(k+1) & \cdots &   \phi(k+1)^{n-2}   \\
 %       \vdots    & \vdots  & \ddots &\vdots   \\
 % 1 & \phi(n) &    \cdots &  \phi(n)^{n-2} 
% \end{pmatrix}
% \]
%
with  determinant $  \prod_{  1\leq h<l \leq n; h,l \neq k  }     [ \phi(l) - \phi(h)] $, so we conclude, as required, }
\l(\ref{RHS}\r)&= (-1)^{n+1} 
  \sum_{k=1}^{n}  (-1)^{k}   \l(          \underset{h,l \neq k}{ \underset{  1\leq h<l \leq n;  }{\prod} }      [ \phi(l) - \phi(h)] \r)    \l( - \underset{1\leq i\leq n-1 }{\prod}      [ \phi(k) - \phi(z+i)]   \r)     
\\
 &= 
  \sum_{k=1}^{n}  (-1)^{k}   \l(          \underset{h,l \neq k}{ \underset{  1\leq h<l \leq n;  }{\prod} }      [ \phi(l) - \phi(h)] \r)    \l( - \underset{1\leq i\leq n-1 }{\prod}      [ \phi(z+i) - \phi(k)]   \r)  .  
 \end{align*}

\end{proof}

  \section{Auxiliary results on bivariate Bernstein functions}\label{sec:B2}
Following \eqref{eq:repKappa},  for each $\zeta\in\Cb_{\lbbrb{0,\infty}}$ we can write
\begin{equation}\label{eq:repKappa1}
\begin{split}
\kappa\lbrb{\zeta,z}&=\kappa\lbrb{\zeta,0}+\dr_2z+\IntOI\lbrb{1-e^{-zx_2}}\lbrb{\IntOI e^{-\zeta x_1}\mu\lbrb{dx_1,dx_2}}\\
&=\phzeta(0)+\dr_2z+\IntOI\lbrb{1-e^{-zx}}\muzeta(dx)\\
&=\phzeta(0)+\dr_2z+z\IntOI e^{-zx}\bmuzeta(x)dx:=\phzeta(z),\,z\in \Cb_{\lbbrb{0,\infty}},
\end{split}
\end{equation}
where $\bmuzeta(y)=\int_{y}^\infty \muzeta(dx), y> 0$.
We then have the following elementary claim which we provide without proof.
\begin{lemma}\label{lem:repKappa1}
	Let $\kappa\in\Bc^2$. Then $\kappa\in \Ac^2_{\lbbrb{0,\infty}}$. Moreover,  for any $\zeta\in\Cb_{\lbbrb{0,\infty}}$ we have that $\kappa\lbrb{\zeta,z}=\phzeta(z)$ and $\phzeta\in\Ac_{\lbbrb{0,\infty}}$, see  \eqref{eq:repKappa1}.  The measure $\muzeta$ is a complex measure on $\intervalOI$. If $\zeta\in\Rb$ then $\phzeta\in \Bc$ if and only if $\int_{0}^\infty \min\curly{1,x}\muzeta(dx)<\infty$, which is always the case if $\zeta\geq 0$.
\end{lemma}
The next result collects results which may be regarded as simple extensions of results available for $\phi\in\Bc$  which can be found in \cite[Section 4]{ps19} or \cite[Section 3]{ps16}.
\begin{proposition}\label{prop:kappa}
	Let $\kappa\in\cal{B}^2$. Then each of the following items holds:
	\begin{enumerate}
		\item\label{it:repKappa1}  For all $\lbrb{\zeta, z}\in\Cb^2_{\lbrb{0,\infty}}   $,
		\begin{equation}\label{eq:kappa'}
		\begin{split}
		\kappa_z'(\zeta,z)&= \phzeta'(z)=\dr_2 +\IntOI xe^{-zx}\muzeta(dx)\\
		&=\dr_2+\IntOI e^{-zx}\bmuzeta(x)dx-z\IntOI e^{-zx}x\bmuzeta(x)dx\\
		&=\frac{\phzeta(z)-\phzeta(0)}{z}-z\IntOI e^{-zx}x\bmuzeta(x)dx.
		\end{split}
		\end{equation}
		\item \label{it:elementaryBiv} For each $\zeta\in \lbbrb{0,\infty}$,   $\phzeta$ is  non-decreasing on $[0,\infty)$, and
		$\phzeta'$ is completely monotone, positive, and non-increasing on $[0,\infty)$.  In particular, $\phzeta$ is strictly log-concave on $[0,\infty)$.
		\item \label{it:ratioBounds} For all $\lbrb{\zeta,z}\in\Cb^2_{\lbbrb{0,\infty}}\setminus\curly{\lbrb{0,0}}$,  we have that 
		\begin{equation}\label{eq:ratioBounds}
		\abs{\frac{\kappa\lbrb{\Reta,\Rez}}{\kappa\lbrb{\zeta,z}}}=\abs{\frac{\phReta(\Re{(z)})}{\phzeta(z)}}\leq 1,
		\end{equation}
		and moreover, the inequality $\l(\ref{eq:ratioBounds}\r)$ is valid for $\zeta=z=0$ when $\kappa\lbrb{0,0}>0$.
		\item \label{it:bivBounds} For all $\lbrb{\zeta,z}\in\Cb_{\lbbrb{0,\infty}}\times\cc$, we have
			\begin{equation}\label{eq:kappa'_kappaC}
		\begin{split}
		\abs{\frac{\kappa'_z\lbrb{\zeta,z}}{\kappa\lbrb{\zeta,z}}}&= \abs{\frac{\phzeta'\lbrb{z}}{\phzeta\lbrb{z}}}\leq \frac{1}{|z|}\lbrb{1+\frac{\abs{\phzeta(0)}}{\abs{\phzeta\lbrb{z}}}}+\frac{|z|}{\Re^2\lbrb{z}}\frac{\phReta(\Rez)}{\abs{\phzeta\lbrb{z}}}\\
		&\leq \frac{1}{|z|}\lbrb{1+\frac{\abs{\phzeta(0)}}{\phReta\lbrb{\Rez}}}+\frac{|z|}{\Re^2\lbrb{z}}\\
		&\leq \frac{1}{|z|}\lbrb{1+\frac{\abs{\phzeta(0)}}{\phReta\lbrb{0}}}+\frac{|z|}{\Re^2\lbrb{z}}
		\end{split}
		\end{equation}
		and 
		\begin{equation}\label{eq:kappa''_kappaC}
		\frac{|\kappa_z''(\zeta,z)|}{\abs{\kappa(\zeta,z)}}=\frac{\abs{\phzeta''(z)}}{\abs{\phzeta(z)}}\leq \lbrb{\frac{2}{\Re^2(z)}+\frac{2|z|}{\Re^3(z)}}\frac{\phReta(\Rez)}{\abs{\phzeta\lbrb{z}}} \leq \frac{2}{\Re^2(z)}+\frac{2|z|}{\Re^3(z)}.
		\end{equation}
		 In particular, taking $z=\Re(z)=u\in [0,\infty)$ and $\zeta\in\Cb_{\lbrb{0,\infty}}$, we have
		\begin{equation}\label{eq:kappa'_kappa}
		\begin{split}
		\abs{\frac{\kappa'_z\lbrb{\zeta,u}}{\kappa\lbrb{\zeta,u}}}&= \abs{\frac{\phzeta'\lbrb{u}}{\phzeta\lbrb{u}}}\leq\frac{1}{u}\lbrb{\frac{\abs{\phzeta(u)-\phzeta(0)}}{\abs{\phzeta\lbrb{u}}}+\frac{\phReta(u)}{\abs{\phzeta\lbrb{u}}}}\leq \frac{2}{u},
		\end{split}
		\end{equation}
		%and
		\begin{equation}\label{eq:kappa''_kappa}
		\frac{|\kappa_z''(\zeta,u)|}{\abs{\kappa(\zeta,u)}}=\frac{\abs{\phzeta''(u)}}{\abs{\phzeta(u)}}\leq \frac{4}{u^2}\frac{\phReta(u)}{\abs{\phzeta(u)}}\leq \frac{4}{u^2}.
		\end{equation}
		\item\label{it:limit} For all $\lbrb{\zeta,z}\in\Cb^2_{\lbbrb{0,\infty}}\setminus\curly{\lbrb{0,0}}$, we have that   
		\begin{equation}\label{eq:limSup}
		\Re\lbrb{\kappa\lbrb{\zeta,z}}>0,
		\end{equation}
		and if in addition $\Re(\zeta)>0$ denoting $B_\zeta$ as a ball centred at $\zeta$ with $\overline{B_\zeta}\subset\CbOI$,   we have that $\inf_{\chi\in\overline{B_\zeta}}\Re\lbrb{\kappa\lbrb{\chi,z}}>0$.
		\item\label{it:kappa''} For all $\lbrb{\zeta,z}\in\Cb_{\lbbrb{0,\infty}}\times\cc$ and for all $n \geq 1$
		\begin{equation}\label{eq:kappa''}
			\abs{\kappa^{(n)}_z\lbrb{\zeta,z}}\leq \abs{\kappa^{(n)}_z(\Reta,\Rez)}. 
		\end{equation}
		\item\label{it:bivBounds_1} For all $\lbrb{\zeta,z}\in\Cb_{\lbbrb{0,\infty}}\times\cc$, we have that
		\begin{equation}\label{eq:kappa'_kappa_1}
		\begin{split}
		\abs{\frac{\kappa'_z\lbrb{\zeta,z}}{\kappa\lbrb{\zeta,z}}}\leq \frac{2}{\Rez},
		\end{split}
		\end{equation}
		%and
		\begin{equation}\label{eq:kappa''_kappa_1}
		\frac{|\kappa_z''(\zeta,z)|}{\abs{\kappa(\zeta,z)}}=\frac{\abs{\phzeta''(u)}}{\phzeta(u)}\leq \frac{4}{\Re^2(z)}.
		\end{equation}
		\item\label{it:limi} For $\zeta\in\Cb_{\lbbrb{0,\infty}}$ and any $A>0$ we have that
		\begin{equation}\label{eq:limi}
			\limi{x}\sup_{0\leq v\leq A}\abs{\frac{\kappa(\zeta, x+v)}{\kappa(\zeta,x)}}=1.
		\end{equation}
		%\item \label{it:asyphid}  $\phi(u)\stackrel{\infty}{=} \dr u +\so{u}$ and $\phi'(u)\stackrel{\infty}{=}\dr+\so{1}$. Fix  $a>\aph$, then $\labsrabs{\phi\lbrb{a+ib}}=\labsrabs{a+ib}\lbrb{ \dr+\so{1}}$ as $|b|\to\infty$.
	%	\item\label{it:finPhi} If $\phi\lbrb{\infty}<\infty$ and $\mu$ is absolutely continuous then for any fixed $a>\aph$, $\limi{|b|}\phi\lbrb{\ab}=\phi\lbrb{\infty}$.
	%	\item \label{it:bernstein_cmi} The mapping $u\mapsto \frac{1}{\phi(u)},\, u\in \R^+,$ is completely monotone, i.e.~there exists a positive measure $U$, whose support is contained in $\lbbrb{0,\infty}$, called the potential measure, such that the Laplace transform of $U$ is given via the identity
		%\[ \frac{1}{\phi(u)} = \int_0^{\infty} e^{-uy}U(dy).\]
		%\item \label{it:bernstein_log_concavity_u/p} \mladen{\textbf{used?}}The mapping $u\mapsto \frac{u}{\phi(u)}$ is positive and log-concave on $\R^+$.
		%\item In any case, \label{it:flatphi}
		%\begin{equation}\label{lemmaAsymp1-1}
	%	\lim_{u\to\infty}\frac{\phi(u\pm a)}{\phi(u)}=1\,\,\text{ uniformly for $a$-compact intervals  on $\R^+$.}
		%\end{equation}
		%\item\label{it:unif_rev1}
		%Uniformly, for $u\in\R^+$, we have that\label{it:asympphi}
		%\begin{equation} \label{eq:asympphi}
		%\phi(u) \asymp u \int_0^{\frac{1}{u}} \bar{\mu}(y) dy + \sigma^2 u + m.
		%\end{equation}
	\end{enumerate}
\end{proposition}

\begin{proof}[Proof of Proposition \ref{prop:kappa}]
Item \mbox{\ref{it:repKappa1}} follows immediately by a simple rearrangement using \eqref{eq:repKappa1}, and   item \mbox{\ref{it:elementaryBiv}} follows from the fact that $\phzeta\in\Bc$ for $\zeta\in\lbbrb{0,\infty}$, using item\,\ref{it:bernstein_cm} of Proposition \mbox{\ref{propAsymp1}}. Now, let us prove item \mbox{\ref{it:ratioBounds}}. 
 Let %$\lbrb{\xi,\eta}=\lbrb{\lbrb{\xi_t,\eta_t}}_{t\geq 0}$
 $\lbrb{\lbrb{\xi_t,\eta_t}}_{t\geq 0}$
  denote the (possibly killed) bivariate subordinator associated to $\kappa$. Then for  an independent exponential random variable $e_q$ with parameter $q>0$, we have   % Then, we have that  
\begin{equation*}
	\abs{\frac{q}{q+\kappa\lbrb{\zeta,z}}}
	=
	\l|  \Ebb{e^{-\zeta \xi_{e_q}-z\eta_{e_q}}} \r|  
	\leq 
	%  \Ebb{   \l| e^{-\zeta \xi_{e_q}-z\eta_{e_q}}} \r|  
	%=
	\Ebb{e^{-\Reta \xi_{e_q}-\Re{(z)}\eta_{e_q}}}
	=
	\abs{\frac{q}{q+\kappa\lbrb{\Reta,\Re{(z)}}}},
\end{equation*}
so that \[\abs{\frac{1}{(q+ \kappa(\zeta,z)}} \leq \abs{\frac{1}{(q+ \kappa\lbrb{\Reta,\Re{(z)}} }},\] for all $q>0$, from which it follows, taking limits as $q\to 0$,  that  \eqref{eq:ratioBounds} holds, and similarly \eqref{eq:ratioBounds} is valied, for $\Reta=\Rez=0$, if $\kappa\lbrb{0,0}>0$.
Let us prove item \mbox{\ref{it:bivBounds}}. We estimate the last term in \eqref{eq:kappa'} to get that
\begin{equation}\label{eq:integEst}
\begin{split}
	\abs{z\IntOI e^{-zy}y\bmuzeta(y)dy}&\leq |z|\IntOI e^{-\Rez y}y\abs{\int_{y}^{\infty}\IntOI e^{-\zeta x_1}\mu(dx_1,dx_2)}dx_2\\
	&\leq |z|\IntOI e^{-\Rez y}y\bmuReta(y)dy\leq |z|\frac{\phReta(\Rez)}{\Re^2(z)},
\end{split}
\end{equation}
where the last inequality follows from the fact that $\phReta\in\Bc$, $\phReta'\geq 0$ on $[0,\infty)$ and the left-hand side of inequality \eqref{eq:phi'_phi}. Relation \eqref{eq:integEst} together with  \eqref{eq:kappa'} show the very first inequality in \eqref{eq:kappa'_kappaC} whereas the second follows from an application of \eqref{eq:ratioBounds} and the third  from the monotonicity of $\phReta$ on $[0,\infty)$.
We proceed to establish relation \eqref{eq:kappa''_kappaC}. For this purpose we note from \eqref{eq:kappa'} that
\begin{equation*}
\begin{split}
\abs{\phzeta''(z)}&\leq 2\abs{\IntOI e^{-zy}y\bmuzeta(y)dy}+\abs{z\IntOI e^{-zy}y^2\bmuzeta(y)dy}\\
&\leq 2\IntOI e^{-\Rez y}y\bmuReta(y)dy+|z|\IntOI e^{-\Rez y}y^2\bmuReta(y)dy\\
&\leq 2\frac{\phReta(\Rez)}{\Re^2(z)}+|z|\IntOI e^{-\Rez y}y^2\bmuReta(y)dy,
\end{split}
\end{equation*}
where for the last inequality we have invoked \eqref{eq:integEst}. It remains to bound the very last integral. However, if $\phi\in\Bc$ then $\phi'$ is completely monotone, see item\,\ref{it:bernstein_cm} of Proposition \ref{propAsymp1}, and hence $\phi''\leq 0$ on $[0,\infty)$. Differentiating the last expression of \eqref{eq:phi'} and utilizing once again \eqref{eq:integEst} we thus arrive at
\begin{equation*}
\begin{split}
u\IntOI e^{-uy}y^2\mubar{y}dy&\leq -\phi''(u)+u\IntOI e^{-uy}y^2\mubar{y}dy=2\IntOI e^{-uy}y\mubar{y}dy\leq 2\frac{\phi(u)}{u^2}
\end{split}
\end{equation*}
and hence
\[|z|\IntOI e^{-\Rez y}y^2\bmuReta(y)dy\leq \frac{2|z|}{\Re^3(z)}\phReta\lbrb{\Rez}.\]
Therefore, collecting the terms we obtain that
\begin{equation*}
\begin{split}
&\abs{\phzeta''(z)}\leq 2\frac{\phReta(\Rez)}{\Re^2(z)}+\frac{2|z|}{\Re^3(z)}\phReta\lbrb{\Rez}
\end{split}
\end{equation*} 
 and employing \eqref{eq:ratioBounds} we get that
 \begin{equation*}
 \begin{split}
 \abs{\frac{\phzeta''(z)}{\phzeta(z)}}&\leq 2\frac{\phReta(\Rez)}{\Re^2(z)\abs{\phzeta(z)}}+\frac{2|z|}{\Re^3(z)}\frac{\phReta(\Rez)}{\abs{\phzeta(z)}}\leq \frac{2}{\Re^2(z)}+\frac{2|z|}{\Re^3(z)}.
 \end{split}
 \end{equation*}
 This ends the proof of \eqref{eq:kappa''_kappaC}. 
 Relation  \eqref{eq:kappa''_kappa} follows by a simple substitution $z=u$ in \eqref{eq:kappa''_kappaC}, whereas \eqref{eq:kappa'_kappa} is derived in the following manner. The first inequality is deduced by not splitting the term $\frac{\phzeta(z)-\phzeta(0)}{z}$ in the last identity of \eqref{eq:kappa'} and substituting in the first inequality of \eqref{eq:kappa'_kappaC} with $z=\Rez=u$. To obtain the second inequality we observe from \eqref{eq:ratioBounds} that for any $\lbrb{\zeta,u}\in\Cb_{\lbbrb{0,\infty}}\times\intervalOI$
 \begin{equation}\label{eq:kappa'_kappaRe}
 \abs{\frac{\kappa'_z(\zeta,u)}{\kappa'_z(\zeta,u)}}\leq\frac{\abs{\phzeta(u)-\phzeta(0)}}{u\abs{\phzeta(u)}}+\frac{1}{u}.
 \end{equation}
 Next, we note from \eqref{eq:repKappa1} that
 \begin{equation*}
 \begin{split}
 \abs{\phzeta(u)-\phzeta(0)}&=\abs{\kappa(\zeta,u)-\kappa(\zeta,0)}\leq d_2u+\IntOI\lbrb{1-e^{-ux}}\muReta(dx)\leq \phReta(u).
 \end{split}
 \end{equation*}
 Therefore, \eqref{eq:kappa'_kappaRe} is further estimated from as \eqref{eq:ratioBounds}
 \begin{equation*}
 \abs{\frac{\kappa'_z(\zeta,u)}{\kappa'_z(\zeta,u)}}\leq\frac{\phReta(u)}{u\abs{\phzeta(u)}}+\frac{1}{u}\leq \frac{2}{u},
 \end{equation*}
 which proves \eqref{eq:kappa'_kappa}.
 We proceed with item\,\ref{it:limit}.  
 Let $\lbrb{\xi,\eta}:=\lbrb{\lbrb{\xi_t,\eta_t}}_{t\geq 0}$ be the possibly killed bivariate subordinator associated to $\kappa$. Then, for all $\lbrb{\zeta,z}\in\Cb^2_{\lbbrb{0,\infty}}\setminus\curly{\lbrb{0,0}}$,
  \[1>\Ebb{e^{-\Reta\xi_1-\Re(z)\eta_1}}\geq \abs{\Ebb{e^{-\zeta\xi_1-z\eta_1}}}=e^{-\Re\lbrb{\kappa(\zeta,z)}}\]
 and \eqref{eq:limSup} follows. The last claim follows from the inequality above, the fact that $\inf_{\chi\in\overline{B_\zeta}}\inf\Reta>0$ since the closed ball $\overline{B_\zeta}\subset\CbOI$ and $\Pbb{\xi_1>0}>0$ almost surely. 
Next we prove item \mbox{\ref{it:kappa''}}. Differentiating \eqref{eq:repKappa1} $n\geq1$ times with respect to $z$, then taking absolute values, it follows that
 \begin{equation*}
 \begin{split}
  \abs{\kappa^{(n)}_z(\zeta,z)}&=\abs{d_2\ind{n=1}+(-1)^{n-1}\IntOI y^{n} e^{-yz}\muzeta(dy)}\\
  &\leq d_2\ind{n=1}+\IntOI y^n e^{-\Rez y}\muReta(dy)\\
  &=\abs{\phReta^{(n)}(\Re(z))}=\abs{\kappa^{(n)}_z\lbrb{\Reta,\Rez}},
 \end{split}
 \end{equation*}
 which establishes \eqref{eq:kappa''} and item\,\ref{it:kappa''}. 
For item \mbox{\ref{it:bivBounds_1}}, we apply first \eqref{eq:kappa''}  with $n=1,2$, and then \eqref{eq:ratioBounds} to yield
 \b{equation}
 \label{n=1,2}
 \abs{\frac{\kappa^{(n)}\lbrb{\zeta,z}}{\kappa\lbrb{\zeta,z}}}\leq \abs{\frac{\kappa^{(n)}\lbrb{\Reta,\Rez}}{\kappa\lbrb{\Reta,\Rez}}}\abs{\frac{\kappa\lbrb{\Reta,\Rez}}{\kappa\lbrb{\zeta,z}}}\leq \abs{\frac{\kappa^{(n)}\lbrb{\Reta,\Rez}}{\kappa\lbrb{\Reta,\Rez}}}.
 \e{equation}
 %Then the first factor on the right-hand side is estimated respectively for $l=1$ and $l=2$ via \eqref{eq:kappa'_kappa} and \eqref{eq:kappa''_kappa} and the second factor is simply bounded by $1$ from \eqref{eq:ratioBounds}. This proves \eqref{eq:kappa'_kappa_1}, \eqref{eq:kappa''_kappa_1} and concludes item \ref{it:kappa''}.
Applying \eqref{eq:kappa'_kappa} when $n=1$ in $\l(\ref{n=1,2}\r)$, we deduce that $\l(\ref{eq:kappa'_kappa_1}\r)$ holds.  Similarly, applying $\l(\ref{eq:kappa''_kappa}\r)$  when $n=2$ in $\l(\ref{n=1,2}\r)$, it follows  that $\l(\ref{eq:kappa''_kappa_1}\r)$ holds, as required for  item \mbox{\ref{it:bivBounds_1}}, and the proof is complete. It remains to consider item \ref{it:limi}. We observe that
\begin{equation*}
\begin{split}
&\limi{x}\sup_{0\leq v\leq A}\abs{\frac{\kappa(\zeta, x+v)}{\kappa(\zeta,x)}}=\limi{x}\sup_{0\leq v\leq A}\abs{1+\frac{\kappa(\zeta, x+v)-\kappa(\zeta,x)}{\kappa(\zeta,x)}}\\
&=\limi{x}\sup_{0\leq v\leq A}\abs{1+\frac{\int_{x}^{x+v}\kappa'_z(\zeta,w)dw}{\kappa(\zeta,x)}}.
\end{split}
\end{equation*}
However,
\begin{equation*}
\begin{split}
&\limi{x}\sup_{0\leq v\leq A}\abs{\frac{\int_{x}^{x+v}\kappa'_z(\zeta,w)dw}{\kappa(\zeta,x)}}\leq A\limi{x}\sup_{x\leq v\leq x+A}\frac{\abs{\kappa'_z(\zeta,w)}}{\abs{\kappa(\zeta,x)}}\\
&\leq A\limi{x}\sup_{x\leq v\leq x+A}\frac{\kappa'_z(\Reta,w)}{\abs{\kappa(\zeta,x)}}=A\limi{x}\frac{\kappa'_z(\Reta,x)}{\abs{\kappa(\zeta,x)}}\\
&\leq A\limi{x}\frac{\kappa'_z(\Reta,x)}{\abs{\kappa(\Reta,x)}}=0,
\end{split}
\end{equation*}
where in the second inequality we have used \eqref{eq:kappa''} with $n=1$,  in the first identity we have used that $\kappa'_z(\Reta,\cdot)$ is non-increasing, see item \ref{it:bernstein_cm} of Proposition \ref{propAsymp1}, for the last inequality we have invoked \eqref{eq:ratioBounds} and the evaluation of the last limit follows from \eqref{eq:kappa'_kappa}. 
 \end{proof}

\section{Proofs of Auxiliary Lemmas} \label{lemmasproofs}
 
  \begin{proof}[Proof of Lemma \mbox{\ref{philemma}}] 
Consider  the fact, see \cite[Prop 3.6]{ssv12}, that  each Bernstein function  $\phi\in\cal{B}$ preserves angular sectors, i.e.\ for each $w\in \bb{C} $,
\(
\l| \arg( \phi(w) ) \r| \leq     | \arg(w)| ,
\)
where we use the convention that $\arg(w)\in (-\pi,\pi]$.  Then, write $z=a+ib$ and without loss of generality assume that $a\in\lbbrbb{c,d}, b\in\lbbrbb{c_1,d_1}$, where $-\infty< c<d<\infty$ and $-\infty<c_1<d_1<\infty$. It follows that for all  $u$ large enough that $u> -a$,
\[
\l| \arg\l(  \f{ \phi(u+z)  }{\phi(u)   }  \r)\r|  = \l| \arg(\phi(u+z)) \r| \leq \l| \arg(u+z) \r| 
= \l|\arctan\l(   \f{b}{u+a}   \r)  \r|.
\]
Observe that $\lim_{u\to\infty} \l|\arctan\l( b/(u+a)   \r)  \r|=0$, uniformly for the specified compact range of $a,b$.
Therefore, for each  $\phi\in\cal{B}$, as $u\to\infty$,
\b{equation}
\label{ree}
   \f{\phi(u + z)}{   \phi(u)}  \sim      \f{|\phi(u + z)|}{   \phi(u)} \sim     \f{\Re(\phi(u + z))}{   \phi(u)}  .
\e{equation}
Writing $\phi$ as in   $\l(\ref{eq:Bern'}\r)$,  since $b$ is bounded, we have   $|u+z|\sim   u+a$, as $u\to\infty$, and   by $\l(\ref{phirealresult}\r)$, as $u\to\infty$,
\b{align}
\nonumber
    \f{|\phi(u + z)|}{   \phi(u)}   &\overset{(\ref{eq:Bern'})}=   \f{ |  \phi(0) +(u+z) \textup{d}+   (u+z)\int_0^\infty   e^{-(u+z)y}   \ov(y)dy    |    }{     \phi(0) +u \textup{d}+   u\int_0^\infty   e^{-uy}   \ov(y)dy      }     \vphantom{\Bigg)}  
\\
\nonumber
&\overset{\phantom{(\ref{eq:Bern'})}}\leq           \f{   \phi(0)+|u+z| \textup{d} +   |u+z|\int_0^\infty   |e^{-(u+z)y}  | \ov(y)dy        }{     \phi(0) + u \textup{d}+  u\int_0^\infty   e^{-uy}   \ov(y)dy      }  \vphantom{\Bigg)}
\\
\nonumber
&\overset{\phantom{(\ref{eq:Bern'})}}=         \f{   \phi(0)+|u+z| \textup{d}  +   |u+z|\int_0^\infty  e^{-(u+a)y}   \ov(y)dy        }{     \phi(0)  +u \textup{d}+   u\int_0^\infty   e^{-uy}   \ov(y)dy      }   . \vphantom{\Bigg)}
\\
\label{upperb}
&\overset{\phantom{(\ref{eq:Bern'})}}\sim         \f{   \phi(0)+(u+a) \textup{d}  +   (u+a)\int_0^\infty  e^{-(u+a)y}   \ov(y)dy        }{     \phi(0)  +u \textup{d}+   u\int_0^\infty   e^{-uy}   \ov(y)dy      }   =    \f{\phi(u+a)}{\phi(u)}   \overset{(\ref{phirealresult})}{\sim} 1  . \vphantom{\Bigg)}
%
%\\
%
%&\leq           \f{   k +   (u+z)\int_0^\infty  e^{-(u+a)y}   \ov(y)dy        }{         u\int_0^\infty   e^{-uy}   \ov(y)dy      } . \vphantom{\Bigg)}
%\intertext{Observe that since $z=a+ib$ is fixed,   $u+z\sim u+a\sim u$ as $u\to\infty$, and hence}
%
%
%
%
%
%
\intertext{Now, observe that $-\cos(-by) \geq-1$ for all $b,y\in \bb{R}$. Then  by $\l(\ref{phirealresult}\r)$,  as $u\to\infty$,}
\nonumber
       \f{ \Re(\phi(u + z))}{   \phi(u)}    &=              \f{     \phi(0) +  \Re((u+z)\textup{d})     + \int_0^\infty \re (1- e^{-(u+z)y} )\Pi(dy)        }{   \phi(0) +  u \textup{d}+    \int_0^\infty (1- e^{-uy} )\Pi(dy)       }             \vphantom{\Bigg)}
\\
\nonumber
%&=              \f{     \phi(0) +      \int_0^\infty \re (1- e^{-(u+a+ib)y} )\Pi(dy)        }{   \phi(0) +      \int_0^\infty (1- e^{-uy} )\Pi(dy)       }             \vphantom{\Bigg)}
%\\
&=         \f{     \phi(0) + (u+a)\textup{d}+     \int_0^\infty   (1- e^{-(u+a)y} \cos(-by) )\Pi(dy)        }{   \phi(0) +  u \textup{d}+    \int_0^\infty (1- e^{-uy} )\Pi(dy)       }            \vphantom{\Bigg)}
\\
\label{lowerb}
&\geq         \f{   \phi(0) +  (u+a)\textup{d}+    \int_0^\infty   (1- e^{-(u+a)y}   )\Pi(dy)        }{ \phi(0) + u \textup{d}+     \int_0^\infty (1- e^{-uy} )\Pi(dy)       }   =  \f{\phi(u+a)}{\phi(u)} \overset{(\ref{phirealresult})}{\sim} 1 .   % =  \underset{u\in(0,\infty)}{\lim_{u\to\infty,}}     \f{\phi(u+a)}{\phi(u)}         =1      , 
\vphantom{\Bigg)}
\end{align}
Combining $\l(\ref{ree}\r)$,  $\l(\ref{upperb}\r)$, and $\l(\ref{lowerb}\r)$, the proof of Lemma \mbox{\ref{philemma}} is complete.
 \end{proof}

   \begin{proof}[Proof of Lemma \mbox{\ref{Wlemma}}]   
   Write $z=a+ib$ and without loss of generality assume that $a\in\lbbrbb{c,d}, b\in\lbbrbb{c_1,d_1}$, where $0\leq c<d<\infty$ and $-\infty<c_1<d_1<\infty$. For  $u>-a$ (so that $W_\phi(u+z+1)$ is well-defined), 
applying Lemma \mbox{\ref{philemma}} alongside $\l(\ref{eq:Stirling}\r)$ and $\l(\ref{eq:limEk}\r)$ from Theorem \mbox{\ref{thm:Stirling}},
it follows that as $u\to\infty$, 
\b{align}
\nonumber
%\l(\ref{Wfraction}\r) 
\f{ W_\phi(u+1)  \phi^{z}(u+1)   }{  W_\phi(u+z+1)    }
% &\overset{\hspace{1.5pt}\ref{philemma}\hspace{1.5pt}}{\sim}    \f{ W_\phi(n+1)  \phi(n+1)^{ a- \lfloor a \rfloor  +ib  }   }{  W_\phi( n+1+a-\lfloor a \rfloor + ib)     }  
%
%
%
%
%
%Then we can write 
%\[
%\f{  W_\phi(n+1)   \phi(n+1)^{ a- \lfloor a \rfloor  +ib  }    }{  W_\phi ( n+1+a-\lfloor a \rfloor + ib)    )    }       \]
%\\
%\nonumber
&\overset{\l(\ref{eq:Stirling}\r)}= \phi^{ z }(u+1)      \f{   \phi ( u+ 1+ z    ) \phi^{\f{1}{2}} ( u+2+z )           }{    \phi(u+1)   \phi^{\f{1}{2}}(u+1)       } 
%
%\\
%\nonumber
%
% &\qquad \times 
\f{ e^{  L_\phi(u+1) – E_\phi(u+1)}  }{ e^{  L_\phi(u+1+z   ) -  E_\phi(u+1 + z ) }  }
 \\  
\nonumber
  &\overset{\hspace{1.5pt}\ref{philemma}\hspace{1.5pt}}{\sim}       
       \phi^{ z  }(u+1)   \  e^{  L_\phi(u+1) – E_\phi(u+1)   - L_\phi(u+1 + z ) +  E_\phi(u+1 + z )   }
%
%
%
%
%
%
%\intertext{Applying $\l(\ref{eq:limEk}\r)$, we have $\lim_{n\to \infty} e^{  -E_{\phi} (n+1)  +  E_\phi(n+1 + a- \lfloor a \rfloor  + ib )     }   =1$, so that as $n\to\infty$,}
%
\\
  \label{l0}
    &\overset{\l(\ref{eq:limEk}\r)}\sim  
              \phi^{ z }(u+1)   \  e^{  L_\phi(u+1)  - L_\phi(u+1 +z )     }.
\end{align}
Now, recalling the definition of the $L_\phi$ error terms in $\l(\ref{eq:A}\r)$, part of the integrals cancel, yielding
\b{align}
\nonumber
   L_\phi(u+1) - L_\phi(u+1 + z )
   &= \int_{ 1 \mapsto u+2     } \log_0 (\phi(w))dw    -      \int_{ 1 \mapsto u+2+ z     } \log_0 (\phi(w))dw
\\
\label{l1}
&=   
-      \int_{ u+2 \mapsto u+2+ z     } \log_0 (\phi(w))dw
%=
 %-      \int_{ N \mapsto N+ a- \lfloor a \rfloor  + ib     } \log_0 (\phi(x))dx,
\end{align}
%where we substitute $N=n+2$.  Now, u
Writing $N=u+2$ for brevity, and  noting that $\log_0(\phi(w))= \ln(|\phi(w)|) + i \arg( \phi(w)  ) $, we have
\b{equation} 
\label{integrls}
\l(\ref{l1}\r)
= -      \int_{ N \mapsto N+ z     } \ln (|\phi(w)|)dw \  – \    i \int_{ N \mapsto N+ z    } \arg (\phi(w))dw.
\end{equation} 
We will integrate along the contours $\gamma_1$ and $\gamma_2$, which are straight lines connecting $N$ to $ N+ a $, and $ N+ a $ to $ N+ a + ib = N+z$, respectively. 
Note that along $\gamma_1$, $\arg(\phi(w))=\arg(w)=0$, so we need only consider $\gamma_2$ for the second integral in $\l(\ref{integrls}\r)$. Now, $|\gamma_2|=b$, so

\[
\l|  \int_{\gamma_2} \arg (\phi(w))dw \r|  
 \leq 
  b \sup_{x\in \gamma_2}|  \arg(\phi(w)) |
\leq 
   b \sup_{w \in \gamma_2}   |\arg(w)|.
\]
since $\phi$ preserves angular sectors \cite[Prop 3.6]{ssv12}. 
Now, $\sup_{w \in \gamma_2}   |\arg(w)| = \arctan(b/(N+a ))$, which converges to $0$ as $N=u+2\to\infty$ and $b$ belongs to a compact interval, and we conclude that 
\b{equation} \label{l2}
\lim_{n\to\infty}          \int_{ N \mapsto N+ a + ib     } \arg (\phi(w))dw        =0. 
\end{equation}
Now let us consider the   first integral in $\l(\ref{integrls}\r)$ over $\gamma_1$ and $\gamma_2$ separately. We will compare the $\gamma_1$  contribution with   $   a    \ln(\phi(u+1))  $, and the $\gamma_2$ part with $ ib\ln(\phi(u+1))  $. For the $\gamma_1$ part,
\[
  \int_{ \gamma_1    } \ln (|\phi(w)|)dw    - a \ln(\phi(u+1))   =   \int_{0}^{a  }  \ln(\phi(u+2+ x))dx   - a \ln(\phi(u+1))   
\]
\b{equation}
\label{int1gamma1}
=    \int_{0}^{a } (   \ln(\phi(u+2+ x))   -   \ln(\phi(u+1))   
  )  dx     
=   \int_{0}^{a  }    \ln \l( \f{ \phi(u+2+ x)}{\phi(u+1)}\r)   
    dx .
\e{equation}
By Lemma \mbox{\ref{philemma}}, $\underset{n\to\infty}{\lim} \phi(u+2+ x)/\phi(u+1)=1$, uniformly in $x \in [0,a]$, %\in[0,a-\lfloor a \rfloor ]$, 
so  $\lim_{u\to\infty} \l(\ref{int1gamma1}\r)=0$,   % $\lim_{n\to\infty}\ln\l(\f{ \phi(n+2+ x)}{\phi(n+1)}\r)=1$ uniformly, and hence we conclude that
%\[
%\lim_{n\to\infty}  \l[  \int_{ \gamma_1    } \ln (|\phi(x)|)dx    - (a-\lfloor a \rfloor ) \ln(\phi(n+1))    \r]  = 0,
%\]
and   %in particular, we have
\b{equation} \label{l3}
\lim_{u\to\infty}  \l(  \phi^{ a }(u+1) \ e^{ -\int_{ \gamma_1    } \ln (|\phi(x)|)dx   }    \r)=1  .
\end{equation}
For the   contour $\gamma_2$,  one can verify by applying the same argument   as for $\gamma_1$, that
\b{equation}
\label{l4}
\lim_{n\to\infty}   \l(   \phi^{ib}(u+1)  e^{ i    \int_{ \gamma_2    } \ln (|\phi(w)|)dw    }   \r) %=   e^{ i  \int_0^b   \ln \l( \phi( n+2+a- \lfloor a \rfloor + iv   )   \r) dv       - i b \ln( \phi(n+1))   }
%
%\[
=
\lim_{u\to\infty}
\l(
 e^{ i  \int_0^b   \ln \l(   \f{  \phi( u+2+a  + iv   )    }{ \phi(u+1)    }  \r) dv    } 
 \r)
 = 1.
%\]
\e{equation}
%Now, by Lemma \mbox{\ref{philemma}}, $  \lim_{n\to\infty}   \f{  \phi( n+2+a- \lfloor a \rfloor + iv   )    }{ \phi(n+1)    } =1$, uniformly among $v\in[0,b]$, so $  \lim_{n\to\infty} \ln\l(   \f{  \phi( n+2+a- \lfloor a \rfloor + iv   )    }{ \phi(n+1)    } \r) =0$, and hence 
%
%\b{equation} \label{l4}
 % \lim_{n\to\infty}     \f{ e^{ i  \int_0^b   \ln \l( \phi( n+2+a- \lfloor a \rfloor + iv   )   \r) dv    }   }{    \phi(n+1)^{ib}   }   =1.
%\end{equation} 
%
Finally, substituting the individual limits  $\l(\ref{l2}\r), \l(\ref{l3}\r)$,   and $ \l(\ref{l4}\r)$ into  $\l(\ref{l0}\r)$, %, \l(\ref{l1}\r)$, 
  we conclude that
\[
   \lim_{u\to\infty}          \f{ W_\phi(u+1)  \phi^{z} (u+1)  }{  W_\phi(u+z+1)    }    =1.
\]

  \end{proof}

	 \begin{proof}[Proof of Lemma \mbox{\ref{productformula}}] 
	  Firstly, by $\varphi_q(w)=q+\phi(w)\in\Bc$ and Lemma \mbox{\ref{Wlemma}}, we can rewrite  
\b{equation} 
\label{Wlim}
W_{\varphi_q}(z+1)   = 1 \times W_{\varphi_q}(z+1)   
=  \lim_{n\to\infty}   \f{   W_{\varphi_q}(n+1)     \varphi^z _q (n+1) W_{\varphi_q}(z+1)      }{  W_{\varphi_q}(n+z+1)     }.
\e{equation}
Recalling the  relations $W_{\varphi_q}(w+1)=\varphi_q(w)W_{\varphi_q}(w)$ and $W_{\varphi_q}(1)=1$ from $\l(\ref{eq:Bern-Gamma}\r)$, it follows that
\[
\l(\ref{Wlim}\r) 
 =     \lim_{n\to\infty}   \f{ \l(    \prod_{k=1}^n    \varphi_q(k)   \r)   W_{\varphi_q}(1)     \varphi^z_q (n+1)  W_{\varphi_q}(z+1)          }{ \l(  \prod_{k=1}^n    \varphi_q(z+k)  \r) W_{\varphi_q}(z+1)     } 
=   \lim_{n\to\infty} \l(  \varphi^z_q(n+1) \prod_{k=1}^n  \f{ \varphi_q(k)   }{  \varphi_q(z+k) } \r).
\]
Now, observing that we can rewrite $\varphi_q(n+1)$ in the form
\[
\varphi_q(n+1) = \varphi_q(1) \prod_{k=1}^n \l( \f{   \varphi_q(k+1)   }{ \varphi_q(k)  }   \r)  = \varphi_q(1) \prod_{k=1}^n \l(    1 +   \f{   \varphi_q(k+1)  -  \varphi_q(k)  }{ \varphi_q(k)  }   \r),
\]
%\[
%=  \phi_q(1) \l(1- \f{\phi_q(2)-\phi_q(1)}{\phi_q(1)} \r)   \l(1- \f{\phi_q(3)-\phi_q(2)}{\phi_q(2)} \r)   \cdots \l(1- \f{\phi_q(n)-\phi_q(n-1)}{\phi_q(n-1)} \r) \l(1- \f{\phi_q(n+1)-\phi_q(n)}{\phi_q(n)} \r)
%\]
%\[
%= \phi_q(1) \prod_{k=1}^n \l(    1 +   \f{   \phi_q(k+1)  -  \phi_q(k)  }{ \phi_q(k)  }   \r).
%\]
we conclude, as required for Lemma \mbox{\ref{productformula}}, that
\[
W_{\varphi_q}(z+1)   = \varphi^z _q(1)\prod_{k=1}^\infty \l[   \f{\varphi_q(k)}{\varphi_q(z+k)}  \l(    1 +   \f{   \varphi_q(k+1)  -  \varphi_q(k)  }{ \varphi_q(k)  }   \r)^z     \r].
\]

	 \end{proof}

\begin{proof}[Proof of Lemma \mbox{\ref{phi'lemma}}] 
Fix     $y\geq1$ such that $\int_0^y x \Pi(dx) >0$. Applying simple inequalities, we get
\b{align*}
 \f{  \phi'(n)  }{    \phi'(n+c) }  = \f{\textup{d}+ \int_0^\infty    x e^{-nx} \Pi(dx)    }{   \textup{d}+     \int_0^\infty    x e^{-(n+c)x} \Pi(dx)     }
&\leq   \f{\textup{d}}{\textup{d}}+
    \f{ \int_0^y    x e^{-nx} \Pi(dx)    }{  \int_0^y    x e^{-(n+c)x} \Pi(dx)     } +  \f{ \int_y^\infty    x e^{-nx} \Pi(dx)    }{  \int_0^y    x e^{-(n+c)x} \Pi(dx)     }
%
%\[
%\leq     \f{ \int_0^A    x e^{-nx} \Pi(dx)    }{  e^{-A} \int_0^A    x e^{-nx} \Pi(dx)     } +  \f{ \int_A^\infty    x e^{-nx} \Pi(dx)    }{  \int_0^A    x e^{-(n+1)x} \Pi(dx)     }
%\]
	\\
	 &\leq    1+  e^{cy} +  \f{ \int_y^\infty    nx e^{-nx} \Pi(dx)    }{ n \int_0^y    x e^{-(n+c)x} \Pi(dx)     }.
\intertext{Observe that when $nx\geq y \geq 1$, the quantity $nx e^{-nx} $ is decreasing in $x$,
and thus}
 \f{  \phi'(n)  }{    \phi'(n+c) } 
 &\leq   1+ e^{cy} +  \f{       ny e^{-ny} \int_y^\infty \Pi(dx)    }{ n \int_0^y    x e^{-(n+c)x} \Pi(dx)     }
 \\
 &\leq    1+ e^{cy} +  \f{       y e^{-ny} \int_y^\infty \Pi(dx)    }{ e^{-(n+c)y} \int_0^y    x  \Pi(dx)     }
 \\
 &=   1+ e^{cy} +  \f{       y   \int_y^\infty \Pi(dx)    }{ e^{-cy} \int_0^y    x  \Pi(dx)     },
\end{align*}
which is a finite constant, independent of $n$, as required.

\end{proof}

	 \begin{proof}[Proof of Lemma \mbox{\ref{absconvlemma}}]  To show that the sum in  $\l(\ref{absconvlemmaeqn}\r)$ is absolutely convergent, we consider terms separately. First, consider the $1/W_{ \phi_{(k)}}(z)$ term. Applying $\l(\ref{eq:Stirling}\r)$ and $\l(\ref{eq:Ek}\r)$ from Theorem \mbox{\ref{thm:Stirling}}, 
\b{equation}
\label{feed1}
\l|     \f{ 1}{ W_{ \phi_{(k)}}(z+1)  }  \r|    \leq       \f{ |  \phi_{(k)}(z+1)  \phi^{\f{1}{2}}_{(k)}(z+2)  |   }{   \phi_{(k)}(1)   }     \l|e^{ - L_{ \phi_{(k)}}(z+1)         }\r|   \   e^{2}.
\e{equation}

%=     e^{2} \times   \f{ | \phi(z+1+k)-\phi(k)| \   | \phi(z+2+k)-\phi(k)|^{\f{1}{2}}      }{  |\phi(k+1)-\phi(k)|    }   \ |e^{ - L_{ \phi_{(k)}}(z+1)         }|.

\n First we bound $ |e^{ - L_{ \phi_{(k)}}(z+1)         }|$. Recalling   $\log_0(w)= \ln(|w|) + i \arg( w  ) $,    by $\l(\ref{eq:A}\r)$, with $z=a+ib$, %we can write
\begin{align*}
 L_{ \phi_{(k)}}(z+1)          &=         \int_{ 1 \mapsto 2+z     } \log_0 ( \phi_{(k)}(w))dw 
\\
   &=       \int_{ 1 \mapsto 2+a     } \  \ln (|\phi_{(k)}(w)|)dw       +  \   \int_{ 2+ a \mapsto 2+a  + i b  }  
   \ \ln (|\phi_{(k)}(w)|)dw   
\\
   &+ i  \int_{ 1 \mapsto 2+a     } \arg( \phi_{(k)}(w))dw +   i  \int_{ 2+a \mapsto 2+a +ib    } \arg( \phi_{(k)}(w))dw.
\end{align*}
But $  \int_{ 2+ a \mapsto 2+a  + i b  } \ln (|\phi_{(k)}(w)|)dw   $ is pure imaginary, and   $ \int_{ 1 \mapsto 2+a     } \arg( \phi_{(k)}(w))dw=0$,  so % we can simplify to get
\[
\l|e^{ - L_{ \phi_{(k)}}(z+1)         }\r| =   \l|      e^{- \int_{ 1 }^{ 2+a     } \ln (\phi_{(k)}(x))dx   -   i  \int_{ 2+a \mapsto 2+a +ib    } \arg( \phi_{(k)}(w))dw  } \r|.
\]

\n As $ \phi_{(k)}$ preserves angular sectors\mbox{\cite[Prop 3.6]{ssv12}},   $\big|\int_{ 2+a \mapsto 2+a +ib    } \arg( \phi_{(k)}(w))dw  \big| \leq b \pi/2$, and so 
\[
\l|e^{ - L_{ \phi_{(k)}}(z+1)         }\r| \leq e^{ \f{b\pi}{2}}   \l|      e^{- \int_{ 1 }^{ 2+a     } \ln (\phi_{(k)}(x))dx      } \r| =      e^{ \f{b\pi}{2}}        e^{- \int_{ 1 }^{ 2+a     } \ln (\phi_{(k)}(x))dx      }   .
\]

\n Now, $\phi_{(k)}$ is non-increasing, and $\phi'_{(k)}$ is non-decreasing, so applying Lemma \mbox{\ref{phi'lemma}}, we get
\b{align*}
 \f{ e^{- \int_{ 1 }^{ 2+a     } \ln(\phi_{(k)}(x)) dx  } }{ \phi_{(k)}(1)} 
 &\leq    
  \f{   e^{-(1+a) \ln (\phi_{(k)}(1))  }   }{ \phi_{(k)}(1)}
  =   \phi^{-(2+a)} _{(k)}(1)
  =  [\phi(k+1)-\phi(k)]^{-(2+a)}     
\\
&= \l[  \int_k^{k+1}    \phi'(x)dx    \r]^{-(2+a)}   \leq   \phi'(k+1)^{-(2+a)}  \overset{\ref{phi'lemma}}{\leq} C_{\phi}^{2+a} \  \phi'(k)^{-(2+a)} ,
\end{align*}
where $C_\phi$ depends only on $\phi$. Hence there is a constant $K>0$ so that for all $k\geq1$,  
\[
\l| \f{e^{ - L_{ \phi_{(k)}}(z+1)         }   }{ \phi_{(k)}(1)}  \r|
 \leq 
 K \phi'(k)^{-(2+a)} ,
\]
from which it follows that, with a slightly different constant $K_1$, for all $k\geq1$,
\[
\l(\ref{feed1}\r) \leq K' \   \l|  \phi_{(k)}(z+1) \ \phi^{\f{1}{2}}_{(k)}(z+2)  \r|    \   \phi'(k)^{-(2+a)}.
\]
Now, since \eqref{eq:kappa''} of Proposition \ref{prop:kappa}\,\ref{it:kappa''} implies that $\abs{\phi'(\zeta)}\leq \phi'(\Reta)),\zeta\in \CbOI$ and $\phi'$  is non-increasing on $\lbrb{0,\infty}$, observe that 
\[
 |\phi_{(k)}(z+1  )   |   = | \phi(z+1+k)-\phi(k)| = \l| \int_k^{z+1+k}  \phi'(w)dw   \r| \leq |z+1| \phi'(k),
\] and similarly $  |\phi^{\f{1}{2}}_{(k)}(z+2  )   |   \leq  |z+2|^{\f{1}{2}} \phi'(k)^{\f{1}{2}}    $. Then, substituting above, we have that 
\[
\l(\ref{feed1}\r) \leq K_1 \   | z+1| \ |z+2|^{\f{1}{2}}     \   \phi'(k)^{-\f{1}{2}-a},
\]
 and  plugging this into $\l(\ref{absconvlemmaeqn}\r)$, it follows that for another constant $K_2>0$,
 \b{equation} 
 \label{partialbound}
  \sum_{k=1}^\infty     \l|      \f{  \prod_{i=1}^{k}     [ \phi(z+i) - \phi(k)]    }{   \prod_{j=1}^{k-1}    [ \phi(j) - \phi(k)]  \hspace{7pt}   }  \f{     e^{-\phi(k)t}   }{ W_{ \phi_{(k)}}(z+1)  }  \r|  \leq K_2      \sum_{k=1}^\infty     \l|      \f{   \prod_{i=1}^{k}    [ \phi(z+i) - \phi(k)]    }{   \prod_{j=1}^{k-1}     [ \phi(j) - \phi(k)]  \hspace{7pt}   }  \r|   \f{     e^{-\phi(k)t}   }{  \phi'(k)^{\frac12+a}  } .
 \end{equation}
Now, applying the bounds in Lemma \mbox{\ref{productbound}},  we see that for another constants $K_3,K_4$,
\b{equation}
\label{penultimatebound}
\l(\ref{partialbound}\r) \leq K_3       \sum_{k=1}^\infty       e^{-\phi(k)t}     \abs{\phi(z+k)-\phi(k)} \phi'(k)^{-\f{1}{2}-a}\leq K_4\sum_{k=1}^\infty       e^{-\phi(k)t}\phi'(k)^{\f{1}{2}-a}.   
\e{equation}
Now, recall that in Definition \mbox{\ref{regularitycondition}}, we have imposed that $\beta(\phi')>-1$, where $\beta(\cdot)$ denotes the lower Matuszewska index, see \cite[p68]{bgt89} for the formal definition.   One can verify that if $\beta(\phi')>-1$, then the function $f(x):=x\phi'(x)$ has lower Matuszewska index $\beta(f)>0$. Then applying \cite[Prop 2.6.1(b)]{bgt89} to the function $f$,  it follows that  there exists a constant $c>0$ such that for all $k\geq1$,  
\b{equation}
\label{matuszbound}
k\phi'(k)= f(k)\geq   c \int_0^k \f{f(x)}{x} dx  = c \int_0^k \phi'(x)dx = c( \phi(k)-\phi(0)).
\e{equation}
Now,   consider the lower index of the function $f$,    defined by  $\underline{\textup{ind}}(f)   :=  \liminf_{x\to\infty} \ln(f(x))/\ln(x) $, 
or equivalently, 
\(
\underline{\textup{ind}}(f)   :=   \sup \{ \rho>0 :    \lim_{x\to\infty}  x^{-\rho} f(x) =\infty       \}  ,
\)   see e.g.\ \cite[p39]{b99}. From the result \cite[Prop 2.2.5]{bgt89},  $\underline{\textup{ind}}(f)  \geq \beta(f)$, and hence
\(
\underline{\textup{ind}}(f)  \geq \beta(f) >0.
\)
Moreover, by $\l(\ref{eq:phi'_phi}\r)$, one can   deduce that $ \underline{\textup{ind}}(\phi)  \geq  \underline{\textup{ind}}(f)   >0$, and in particular, there exist $c,\rho>0$ such that for all $k\geq1$, 
\b{equation}
\label{blumbound}
\phi(k)    \geq   c k^\rho   .
\e{equation}
Now, if $a>1/2$, then since $\phi$ is non-decreasing, by $\l(\ref{matuszbound}\r)$ and $\l(\ref{blumbound}\r)$,  for another constant $K_5$, 
\[
\l(\ref{penultimatebound}\r) 
 \leq  
  K_5 \sum_{k=1}^\infty       e^{-ct k^{\rho} }        \f{ c^{a-\f{1}{2}}    [\phi(k)-\phi(0)]^{\f{1}{2}-a}    }{  k^{\f{1}{2}-a}   }
   \leq  
     K_5        c^{a-\f{1}{2}}    [\phi(1)-\phi(0)]^{\f{1}{2}-a}         \sum_{k=1}^\infty       e^{-ct k^{\rho} }    k^{a- \f{1}{2}} ,
\]
and then one can easily verify that this sum is finite, as required. On the other hand, if $a\leq 1/2$, then since $\phi'$ is non-increasing, 
\[
\l(\ref{penultimatebound}\r) 
 \leq  
 K_5       \sum_{k=1}^\infty       e^{-ct k^{\rho} }            \phi'(k)^{\f{1}{2}-a}      \leq     K_5      \phi'(1)^{\f{1}{2}-a}        \sum_{k=1}^\infty       e^{-ct k^{\rho} }            ,
\]
and again one can easily verify that this sum is finite, and the proof of Lemma \mbox{\ref{absconvlemma}} is complete. 

\end{proof}

\begin{proof}[Proof of Lemma \mbox{\ref{productbound}}]
 Writing $z=a+ib$, $a,b\in\bb{R}$, we first consider the product for $1\leq j \leq k-C_z$, where $C_z>a$ is to be determined later.  We are going to change variables in the numerator so that the new variable ranges between $0$ and $k-j$. The change of variables is motivated by the observation that 
\[
\phi(k)-\phi(j)=\int_j^k \phi'(x)dx = \int_0^{k-j} \phi'(j+\tau)d\tau,
\]
so that after this change of variables, we can simply compare terms within each integral. 
\ \\ 

\n Consider first $u\in[0,1]$ chosen so that 
$w = (z+j) (1-u) + k u$. 
Then $(z+j) (1-u) + k u$ ranges between $z+j$ and $ k$, as desired.
Now, $\tau$ is simply defined as 
$\tau := (k-j) u$,
so that the ranges of integration match.
\ \\

\n Now, changing variables to  $\tau$ chosen such that $w=z+j + (1-z/(k-j))\tau $, we can rewrite $|\phi(k) - \phi(z+j)|$ as 
\[
\l|\int_{z+j\mapsto k}  \phi'(w) dw \r|  =\l|  \int_{0}^{k-j}  (1-\frac{z}{k-j}) \phi'\lbrb{z+j + (1-\frac{z}{k-j})\tau} d\tau  \r| 
\]
\[
\leq      \int_{0}^{k-j} | (1-\frac{z}{k-j})|\ | \phi'\lbrb{z+j + (1-\frac{z}{k-j})\tau}| d\tau
\]
\[
\leq     \int_{0}^{k-j} | (1-\frac{z}{k-j})|\  \phi'\lbrb{a+j + (1-\frac{a}{k-j})\tau} d\tau,
\]
where we have used for the last inequality relation \eqref{eq:kappa''} of Proposition \ref{prop:kappa}\,\ref{it:kappa''}.
We can now compare terms. Since $\tau\leq k-j,$ observe that $a-a\tau/(k-j)\geq0$, and hence $a+j + (1-a/(k-j))\tau \geq j+\tau$, from which it follows, since $\phi'$ is non-increasing on $\lbrb{0,\infty},$  that
\[
\phi'\lbrb{a+j + (1-\frac{a}{k-j})\tau} \leq \phi'(j+\tau).
\]
Now we want to show that $| (1-z/(k-j))|\leq 1$. We consider its square for convenience:
\[
| (1-\frac{z}{k-j})|^2=  \f{  (k-j-a)^2 }{(k-j)^2} +  \f{  b^2 }{(k-j)^2}= 1 -  \f{2a(k-j) }{(k-j)^2} +\f{ a^2+ b^2 }{(k-j)^2} .
\]
This does not exceed one if and only if 
\[
-  \f{2a(k-j) }{(k-j)^2} +\f{ a^2+ b^2 }{(k-j)^2} \leq0
\]
\[
\iff       a^2+ b^2    \leq   2a(k-j)    \iff  \f{ a^2+ b^2 }{2a} \leq       k-j
\]
\[
\iff j\leq    k - \f{ a^2+ b^2 }{2a} = k - \f{|z|^2}{2\Re(z)}.
\]
So let us choose $C_z\geq\f{|z|^2}{2\Re(z)} $. Then it follows that for all $1\leq j \leq  k-C_z$,
\[
 \f{ | \phi(z+j) - \phi(k)|  }{   \phi(k) - \phi(j) }    \leq 1,
\]
and hence the product simplifies substantially:
\begin{equation*}
	\begin{split}
	&\underset{1\leq j\leq k-1 }{\prod}     \f{ | \phi(z+j) - \phi(k)|  }{   \phi(k) - \phi(j) }  \leq \underset{k-C_z\leq j\leq k-1 }{\prod}     \f{ | \phi(z+j) - \phi(k)|  }{   \phi(k) - \phi(j) }\\
	& \leq     \underset{k-C_z\leq j\leq k-1 }{\prod}     \f{ | \phi(z+j) - \phi(k)|  }{   \phi(k) - \phi(k-1) }     \leq     \underset{k-C_z\leq j\leq k-1 }{\prod}     \f{ | \int_{z+j\mapsto k}\phi'(w)dw     |  }{   \phi'(k) } \\
	&\leq    \underset{k-C_z\leq j\leq k-1 }{\prod}     \f{ |  z+j- k| \ (\max_{k-C_z\leq j \leq k-1} \sup_{w \in [z+j,k]}|\phi'(w)| ) \vee \phi'(k)     }{   \phi'(k) }\\
	& \leq    \underset{k-C_z\leq j\leq k-1 }{\prod}     \f{ |  z+j- k| \ (\max_{k-C_z\leq j \leq k-1} \phi'(a+j) ) \vee \phi'(k)     }{   \phi'(k) }\\
	&\leq      \underset{k-C_z\leq j\leq k-1 }{\prod}     \f{ |  z+j- k| \ \phi'(k+a-C_z)  \vee \phi'(k)     }{   \phi'(k) }, 
	\end{split}
\end{equation*}
where $\vee$ stands for the maximum function, we have used $|\phi'(w)|\leq \phi'(\Re(w))$, see \eqref{eq:kappa''}, and the monotonicity of $\phi'$.
Now, $|z+j-k|\leq |z|+k-j\leq |z|+C_z$, because $j \geq k-C_z$, so this is
\[
\underset{1\leq j\leq k-1 }{\prod}     \f{ | \phi(z+j) - \phi(k)|  }{   \phi(k) - \phi(j) }\leq    (|z|+C_z)^{C_z}   \underset{k-C_z\leq j\leq k-1 }{\prod}     \f{ \phi'(k+a-C_z)  \vee \phi'(k)     }{   \phi'(k) } .
\]
Finally, observe that if $\phi'(k+a-C_z)  \leq \phi'(k) $ then the remaining product is $\leq 1 $, and on the other hand,  if  $\phi'(k+a-C_z)  \geq \phi'(k) $, then uniformly among all $k$ large enough, we have from Lemma \mbox{\ref{phi'lemma}}
\[\f{ \phi'(k+a-C_z)  \vee \phi'(k)     }{   \phi'(k) }=\f{ \phi'(k+a-C_z)      }{   \phi'(k) }   \leq C,\] where $C>0$ is independent of $k$, and the proof of Lemma \mbox{\ref{productbound}} is complete.   

\end{proof}

\bibliography{bg.bib}
\bibliographystyle{plain}

\end{document}